\documentclass [11pt,reqno]{amsart}
\usepackage {amsmath,amssymb,epsfig,verbatim,geometry}
\usepackage[all]{xy}

\usepackage[backref,bookmarks=false]{hyperref}

\usepackage[usenames,dvipsnames]{color}

\newcommand{\A}{\mathbb{A}}

\newcommand{\C}{\mathbb{C}}
\newcommand{\DD}{\mathbb{D}}
\newcommand{\F}{\mathbb{F}}
\newcommand{\G}{\mathbb{G}}
\newcommand{\Q}{\mathbb{Q}}
\newcommand{\R}{\mathbb{R}}
\newcommand{\Z}{\mathbb{Z}}
\newcommand{\N}{\mathbb{N}}
\renewcommand{\P}{\mathbb{P}}

\newcommand{\fS}{\mathfrak{S}}

\newcommand{\tF}{\widetilde{F}}
\newcommand{\tK}{\widetilde{K}}

\newcommand{\tX}{\widetilde{X}}

\newcommand{\cA}{\mathcal{A}}
\newcommand{\cB}{\mathcal{B}}
\newcommand{\cC}{\mathcal{C}}

\newcommand{\cF}{\mathcal{F}}

\newcommand{\cH}{\mathcal{H}}

\newcommand{\cL}{\mathcal{L}}
\newcommand{\cM}{\mathcal{M}}
\newcommand{\cN}{\mathcal{N}}
\newcommand{\cO}{\mathcal{O}}
\newcommand{\cP}{\mathcal{P}}

\newcommand{\cU}{\mathcal{U}}
\newcommand{\cV}{\mathcal{V}}
\newcommand{\cX}{\mathcal{X}}
\newcommand{\cY}{\mathcal{Y}}
\newcommand{\cZ}{\mathcal{Z}}

\newcommand{\hcA}{\widehat{\mathcal{A}}}
\newcommand{\hcB}{\widehat{\mathcal{B}}}
\newcommand{\hcX}{\widehat{\mathcal{X}}}

\newcommand{\hA}{\widehat A}
\newcommand{\hf}{\widehat f}

\newcommand{\ha}{\widehat a}

\newcommand{\Op}[1]{\operatorname{#1}}

\newcommand{\lau}[1]{(\!(#1)\!)}

\renewcommand{\a}{\alpha}
\renewcommand{\b}{\beta}
\renewcommand{\d}{\delta}
\newcommand{\e}{\varepsilon}

\newcommand{\g}{\gamma}
\newcommand{\la}{\lambda}
\newcommand{\om}{\omega}
\newcommand{\p}{\psi}

\newcommand{\Ga}{\Gamma}
\newcommand{\D}{\Delta}

\newcommand{\eg}{{\rm e.g.\ }}
\newcommand{\ie}{{\rm i.e.\ }}

\newcommand{\an}{\mathrm{an}}
\newcommand{\diag}{\mathrm{diag}}

\newcommand{\op}{\mathrm{op}}

\newcommand{\n}{\|\cdot\|}

\DeclareMathOperator{\Spec}{Spec}

\DeclareMathOperator{\supp}{Supp}
\DeclareMathOperator{\vol}{vol}
\DeclareMathOperator{\codim}{codim}
\DeclareMathOperator{\Pic}{Pic}

\DeclareMathOperator{\env}{P}
\DeclareMathOperator{\en}{E}

\DeclareMathOperator{\Hom}{Hom}

\DeclareMathOperator{\qq}{Q}

\DeclareMathOperator{\PSH}{PSH}

\DeclareMathOperator{\spec}{Spec}

\DeclareMathOperator{\red}{red}
\DeclareMathOperator{\PL}{PL}

\DeclareMathOperator{\dist}{d}
\DeclareMathOperator{\FS}{FS}
\DeclareMathOperator{\GL}{GL}
\DeclareMathOperator{\Ker}{Ker}
\DeclareMathOperator{\Vect}{Vect}

\DeclareMathOperator{\cyc}{cyc}

\DeclareMathOperator{\cz}{C^0}
\DeclareMathOperator{\Hnot}{H^0}
\DeclareMathOperator{\hnot}{h^0}

\DeclareMathOperator{\Rea}{Re}
\DeclareMathOperator{\Ima}{Im}

\DeclareMathOperator{\Spf}{Spf}

\DeclareMathOperator{\rk}{rk}
\DeclareMathOperator{\Tr}{Tr}

\DeclareMathOperator{\Gr}{Gr}

\DeclareMathOperator{\cont}{cont}

\DeclareMathOperator{\gr}{gr}

\newcommand{\Res}{\Op{Res}}

\DeclareMathOperator{\dGI}{\dist_{\infty}}

\DeclareMathOperator{\trop}{Trop}

\newcommand{\latt}{\mathrm{latt}}

\newcommand{\redu}{\mathrm{red}}
\renewcommand{\div}{\mathrm{div}}

\newcommand{\be}{{\bf e}}
\newcommand{\bs}{{\bf s}}

\numberwithin{equation}{section}       

\newtheorem{prop} {Proposition} [section]
\newtheorem{thm}[prop] {Theorem}
\newtheorem{defi}[prop] {Definition}
\newtheorem{lem}[prop] {Lemma}
\newtheorem{cor}[prop]{Corollary}
\newtheorem{conj}[prop]{Conjecture}
\newtheorem{prop-def}[prop]{Proposition-Definition}

\newtheorem*{thmA}{Theorem A}

\newtheorem*{thmC}{Theorem C}
\newtheorem*{thmD}{Theorem D}

\newtheorem*{corB}{Corollary B}

\newtheorem{exam}[prop]{Example}
\newtheorem{rmk}[prop]{Remark}
\theoremstyle{remark}

\usepackage{color}
\usepackage{microtype}

\title[Norms, determinant of cohomology and Fekete points]{Spaces of norms, determinant of cohomology and Fekete points in non-Archimedean geometry}
\date{\today}

\author{S{\'e}bastien Boucksom and Dennis Eriksson}

\address{CNRS-CMLS\\
  \'Ecole Polytechnique\\
  F-91128 Palaiseau Cedex\\
  France}
\email{sebastien.boucksom@polytechnique.edu}

\address{Chalmers University of Technology and University of Gothenburg \\
SE-412 96 \\ 
Gothenburg, Sweden}
\email{dener@chalmers.se}

\begin{document}

\begin{abstract} Let $L$ be an ample line bundle on a (geometrically reduced) projective variety $X$ over any complete valued field. Our main result describes the leading asymptotics of the determinant of cohomology of large powers of $L$, with respect to the supnorm of a continuous metric on the Berkovich analytification of $L$. As a consequence, we establish in this setting the existence of transfinite diameters and equidistribution of Fekete points, following a strategy going back to Berman, Witt Nystr\"om and the first author for complex manifolds. In the non-Archimedean case, our approach relies on a version of the Knudsen--Mumford expansion for the determinant of cohomology on models over the (possibly non-Noetherian) valuation ring, as a replacement for the asymptotic expansion of Bergman kernels in the complex case, and on the reduced fiber theorem, as a replacement for the Bernstein--Markov inequalities. Along the way, a systematic study of spaces of norms and the associated Fubini--Study type metrics is undertaken. 

\end{abstract}

\maketitle

\setcounter{tocdepth}{1}
\tableofcontents

%
%
\section*{Introduction}
%
%
Fekete points and transfinite diameter are classical notions in logarithmic potential theory in the plane. For each $m\in\N$, the \emph{$m$-diameter} $\d_m(K)$ of a compact subset $K\subset\C$ is defined as the supremum of the geometric mean distance between $m+1$ points in $K$, maximizers being called \emph{Fekete configurations}. The $m$-diameter of $K$ admits a limit $\d_\infty(K)$ as $m\to\infty$, called the \emph{transfinite diameter} of $K$, which turns out to coincide with the logarithmic capacity of $K$. Further, Fekete configurations become asymptotically unique in the limit, in the sense that they equidistribute to a certain canonical probability measure on $K$, called its \emph{equilibrium measure}. 

In several complex variables, a similar picture was only rather recently obtained. The first steps were taken by Leja in the 1950's, introducing a notion of $m$-diameter $\d_m(K)$ for a compact subset $K\subset\C^n$, defined in terms of the supremum of certain Vandermonde-type determinants. The existence of the transfinite diameter $\d_\infty(K)=\lim_{m\to\infty}\d_m(K)$ was established by Zaharjuta~\cite{Zah}, and the next key step came with Rumely's observation in~\cite{Rum} that the general results in arithmetic intersection theory developed in~\cite{CLR} yield in particular an exact formula for $\d_\infty(K)$ in terms of pluripotential theory, generalizing the classical Robin formula for the logarithmic capacity in the plane, and involving plurisubharmonic envelopes and mixed Monge--Amp\`ere operators in the sense of Bedford--Taylor. This triggered joint work of the first author with Berman and Witt Nystr\"om~\cite{BB,BBW}, which built on Bergman kernel asymptotics to establish a general version of Rumely's formula in the setting of complex projective manifolds, and combined it with a variational argument to prove the equidistribution of Fekete configurations in this context. 

\smallskip

The main purpose of the present paper is to study versions of these results in non-Archimedean (Berkovich) geometry. While many results hold over an arbitrary non-Archimedean complete valued field $K$, the full picture relies on more refined non-Archimedean pluripotential theory as developed in~\cite{siminag,nama,BG+,trivval}, and hence requires $K$ to be trivially or discretely valued and of residue characteristic $0$.
%
%
\subsection*{Asymptotics of relative volumes}
The Bouche--Catlin--Tian--Zelditch asymptotic expansion of Bergman kernels~\cite{Bouche,Cat,Tian,Zel} is a fundamental result in complex geometry, which describes the asymptotic behavior of the $L^2$-norms associated to large tensor powers of a positive Hermitian line bundle. As noticed in~\cite{BB}, it can be reformulated as an asymptotic expansion for the logarithmic volume ratio of such $L^2$-norms, as follows. 

First, define the \emph{relative volume} of any two norms $\n,\n'$ on an $N$-dimensional complex vector space $V$ as 
\begin{equation}\label{equ:relvol}
\vol(\n,\n'):=\log\left(\frac{\det\n'}{\det\n}\right),
\end{equation}
where $\det\n$, $\det\n'$ denote the induced norms on the determinant line $\det V=\bigwedge^N V$. In terms of the unit balls $B,B'\subset V$ of the two norms, 
$$
\vol(\n,\n')=\frac 12\log\left(\frac{\vol B}{\vol B'}\right)+O(N\log N), 
$$
where the error term $O(N\log N)$ vanishes when the two norms are Hermitian. 

Let next $X$ be an $n$-dimensional complex projective manifold. Every smooth, positively curved metric $\phi$ on an (ample) line bundle $L$ over $X$ induces, for each $m\in\N$, an $L^2$-norm $\n_{L^2(m\phi)}$ on the space of global sections $\Hnot(mL)=\Hnot(X,L^{\otimes m})$. The asymptotic expansion of Bergman kernels mentioned above turns out to be equivalent to the existence of a full asymptotic expansion
$$
\vol\left(\n_{L^2(m\phi)},\n_{L^2(m\p)}\right)=m^{n+1} a_{n+1}+m^n a_n+\ldots+O(m^{-\infty})
$$
for any two such metrics $\phi,\p$. Up to a multiplicative constant, the leading order coefficient $a_{n+1}$ can further be identified with a fundamental functional in K\"ahler geometry, the \emph{relative Monge-Amp\`ere energy}\footnote{Note that the present normalization, which is more convenient for the purpose of this paper, is not uniform accross the literature.}
\begin{equation}\label{equ:Ecomp}
\en(\phi,\p):=\frac{1}{n+1}\sum_{j=0}^n\int_X(\phi-\p)(dd^c\phi)^j\wedge(dd^c\p)^{n-j}.
\end{equation}
We use additive notation for metrics on line bundles, so that $\phi-\p$ is a function on $X$, and $dd^c\phi$, $dd^c\p$ denote the curvature $(1,1)$-forms of $\phi,\p$. 
%

\bigskip

Consider now an arbitrary complete valued field $K$, \ie a field $K$ complete with respect to an absolute value. In the Archimedean case, $K=\R$ or $\C$, by the Gelfand--Mazur theorem. In the non-Archimedean case, the main examples are $K=\Q_p$, $\C_p$, or fields of formal Laurent series and their completed algebraic closure, but the trivially valued case is also important for the study of K-stability (see~\cite{nakstab}). Let $X$ be a geometrically reduced projective scheme over $K$, $L$ be a line bundle on $X$, and $\phi,\p$ be continuous metrics on (the Berkovich analytification of) $L$. Building on a result of Chen and Maclean based on Okounkov bodies~\cite{CMac}, we show the existence of the \emph{relative volume of $\phi,\p$} 
\begin{equation}\label{equ:relvol2}
\vol(L,\phi,\p):=\lim_{m\to\infty}\frac{n!}{m^{n+1}}\vol\left(\n_{m\phi},\n_{m\p}\right)\in\R,
\end{equation}
cf.~Theorem~\ref{thm:relvolmetr}. The main result of the present paper is as follows. 

\begin{thmA} Let $X$ be a geometrically reduced projective scheme over any complete valued field $K$, and $\phi,\p$ be continuous psh\footnote{A shorthand for plurisubharmonic.} metrics on an ample line bundle $L$ over $X$. Then 
$$
\vol(L,\phi,\p)=\en(\phi,\p).
$$
\end{thmA} 
In the Archimedean case, it follows from results of Demailly that a continuous metric on $L$ is psh iff it is a uniform limit of Fubini--Study metrics (see Theorem~\ref{thm:dem}). In the non-Archimedean case, we use this as a definition, and prove compatibility with the more common notion of semipositive metric  in this context, due to S.W.~Zhang and involving nef models, when $K$ is nontrivially valued. We also show that a continuous metric $\phi$ is psh iff it becomes psh after ground field extension. The relative Monge--Amp\`ere energy of continuous psh metrics can still be defined by~\eqref{equ:Ecomp}, the latter being understood in the sense of Chamber-Loir and Ducros~\cite{CLD} in the non-Archimedean sense.

\smallskip

In the Archimedean case, Theorem A reduces to~\cite{BB} after passing to a resolution of singularities. In the discretely valued case, it was established in~\cite[Theorem A]{BG+}. The main contribution of the present paper is thus to establish Theorem A for a densely valued, non-Archimedean ground field $K$. 

\bigskip

The \emph{psh envelope} $\env(\phi)$ of a continuous metric $\phi$ on $L$ is defined as the pointwise supremum of the family of all continuous psh metrics $\p$ on $L$ such that $\p\le\phi$. We say that \emph{continuity of envelopes} holds for $(X,L)$ if $\env(\phi)$ is continuous (and hence psh) for each continuous metric $\phi$ on $L$. 

\begin{corB} Assume that continuity of envelopes holds for $(X,L)$. For any two continuous metrics $\phi,\p$ on $L$, we then have
$$
\vol(L,\phi,\p)=\en(\env(\phi),\env(\p)).
$$
\end{corB}
In the Archimedean case, classical results in pluripotential imply that continuity of envelopes holds whenever $X$ is normal, and Corollary B is indeed also a consequence of~\cite{BB} in that case. 

In the non-Archimedean case, we similarly expect continuity of envelopes to holds as soon as $X$ is normal.  As of this writing, continuity of envelopes (and hence Corollary B) has been established when $X$ is smooth, and one of the following is satisfied: 
\begin{itemize}
\item $X$ is a curve, as a consequence of A.~Thuillier's work~\cite{Thu} (see~\cite{GJKM});
\item $K$ is discretely or trivially valued, of residue characteristic $0$~\cite{siminag,trivval}, building on
multiplier ideals and the Nadel vanishing theorem;
\item $K$ is discretely valued of characteristic $p$, $(X,L)$ is defined over a function field of
transcendence degree $d$, and resolution of singularities is assumed in dimension $n+d$~\cite{GJKM}, replacing multiplier ideals with test ideals.
\end{itemize}
%
%
\subsection*{Sketch of the proof} 
As already mentioned, Theorem A basically follows from~\cite{BB} in the Archimedean case, and we henceforth assume that $K$ is non-Archimedean. The main tools involved in the proof can be summarized as follows. 

\subsubsection*{The reduced fiber theorem} The proof of Theorem A is fairly easily reduced to the case where each metric is induced by an ample model $(\cX,\cL)$ of $(X,L)$, \ie an ample line bundle $\cL$ extending $L$ to a projective model $\cX$ of $X$ over the valuation ring $K^\circ$. Besides the supnorm $\n_{m\phi}$ defined by the model metric $\phi=\phi_\cL$, the space of sections $\Hnot(mL)$ is then also equipped with the lattice norm $\n_{\Hnot(m\cL)}$ induced by the $K^\circ$-module $\Hnot(m\cL)=\Hnot(\cX,\cL^{\otimes m})$. This lattice norm, which is to some extent the analogue of the $L^2$-norm in the present non-Archimedean context, coincides with the supnorm when $\cX$ has a reduced special fiber, but not in general. Using the Bosch--L\"utkebohmert--Raynaud reduced fiber theorem~\cite{BLR}, we prove however that the distortion between $\n_{m\phi}$ and $\n_{\Hnot(m\cL)}$ remains bounded as $m\to\infty$, which enables us to replace the supnorms with the lattice norms in proving Theorem A.  

\subsubsection*{Knudsen--Mumford expansion} Our main tool is then the Knudsen--Mumford expansion of the determinant of cohomology~\cite{KM}, which plays the role of the asymptotic expansion of Bergman kernels in the complex case, and provides for $m\gg 1$ a polynomial expansion (in additive notation for $\Q$-line bundles over $\spec K^\circ$)
\begin{equation}\label{knudsenmumford}
\det \Hnot(m\cL)=\frac{m^{n+1}}{(n+1)!}\langle\cL^{n+1}\rangle+\dots
\end{equation}
with leading order term the Deligne pairing $\langle\cL^{n+1}\rangle$. Since models over $K^\circ$ are non-Noetherian in the densely valued case, some care is however required to apply this result, and the relevant explanations are provided in Appendix A, based on F.~Ducrot's approach~\cite{Duc}. 

\subsubsection*{Metrics on Deligne pairings} The Knudsen--Mumford expansion yields an expression of the relative volume of the metrics as a difference of model metrics on the Deligne pairing $\langle L^{n+1}\rangle$. The final ingredient in the proof of Theorem A consists in relating model metrics on Deligne pairings with mixed Monge--Amp\`ere integrals, which is accomplished via the Poincar\'e--Lelong formula of~\cite{CLD} and a careful monotone regularization argument. 

%
%
\subsection*{Transfinite diameter and Fekete points}
Following the strategy developed in~\cite{BB,BBW}, we establish the existence of transfinite diameters, and combine it with a differentiability result proved in~\cite{nama,BG+,trivval} under appropriate assumptions on the ground field $K$ to infer equidistribution of Fekete points. 

Let as above $X$ be a geometrically reduced projective scheme over a complete valued field $K$. Let $L$ be a line bundle on $X$, and set $N:=\dim\Hnot(X,L)$. The data of a basis $\bs=(s_i)$ of $\Hnot(X,L)$ determines a generator $s_1\wedge\dots\wedge s_N$ of $\det\Hnot(X,L)$, as well as a section 
$$
\det\bs\in\Hnot\left(X^N,L^{\boxtimes N}\right), 
$$
expressed as the Vandermonde determinant
$$
(\det\bs)(x_1,\ldots,x_{N})=\det\left(s_i(x_j)\right).
$$
Every continuous metric $\phi$ on $L$ induces a continuous metric $\phi^{\boxtimes N}$ on $L^{\boxtimes N}$, and a \emph{Fekete configuration} for $\phi$ is a point $P\in (X^N)^\an$ that computes the supnorm $\|\det\bs\|_{\phi^{\boxtimes N}}$. The choice of a norm $\n$ on $\Hnot(X,L)$ induces a norm $\det\n$ on $\det\Hnot(X,L)$, and the \emph{diameter} of $\phi$ with respect to $\n$ is defined as the normalized supnorm
$$
\d(\phi,\n):=\frac{\|\det\bs\|_{\phi^{\boxtimes N}}}{\det\|s_1\wedge\ldots\wedge s_{N}\|},
$$
which is independent of the choice of basis $\bs=(s_i)$. 

\begin{thmC} Let $L$ be any line bundle on a geometrically reduced projective scheme $X$ over a complete valued field $K$. For any two continuous metrics $\phi,\p$ on $L$, the \emph{transfinite diameter}
$$
\d_{\infty}(\phi,\p):=\lim_{m\to\infty}\d(m\phi,\n_{m\p})^{n!/m^{n+1}}
$$
exists in $\R_{>0}$, and satisfies $\log\d_{\infty}(\phi,\p)=\vol(L,\p,\phi)$.
\end{thmC}
This is inferred from the existence of the limit~\eqref{equ:relvol2} defining the relative volume $\vol(L,\phi,\p)=-\vol(L,\p,\phi)$ via an estimate of the operator norms of the embeddings $\det \Hnot(mL)\hookrightarrow \Hnot\left((mL)^{\boxtimes N_m}\right)$ with respect to the norms induced by $\phi$ on both sides, which is again ultimately deduced from the reduced fiber theorem. 

From now on we assume that $L$ is ample. To state our last main result, we shall say that \emph{differentiability holds} at a continuous metric $\phi$ on $L$ if: 
\begin{itemize}
\item[(i)] the psh envelope $\env(\phi)$ is continuous (hence psh);
\item[(ii)] for all $f\in \cz(X^\an)$ we have 
$$
\frac{d}{dt}\bigg|_{t=0}\vol(L,\phi+t f,\phi)=\int_{X^\an} f\,\left(dd^c\env(\phi)\right)^n.
$$
\end{itemize}
 
Differentiability is known to hold at all continuous metrics when $X$ is smooth and one of the following conditions is satisfied: 
\begin{itemize}
\item $K$ is Archimedean~\cite{BB}; 
\item $K$ is non-Archimedean, trivially or discretely valued, of residue characteristic zero~\cite{siminag,BG+,trivval};
\item $K$ is discretely valued of characteristic $p$, $(X,L)$ is defined over a function field of
transcendence degree $d$, and resolution of singularities is assumed in dimension $n+d$~\cite{GJKM,BG+}.
\end{itemize}
In a forthcoming paper~\cite{BGM}, it will be shown that continuity of envelopes implies differentiability at all continuous metrics. 

Importing the variational argument of~\cite{BBW}, itself based on an idea of~\cite{SUZ}, we prove:  

\begin{thmD} Let $X$ be a geometrically reduced, projective scheme over a complete valued field $K$. Let $\phi$ be a continuous metric on an ample line bundle $L$ over $X$, of volume $V=(L^n)$, and assume that differentiability holds at $\phi$. For each $m\gg 1$, pick a Fekete configuration $P_m\in(X^{N_m})^\an$ for $m\phi$. Then $P_m$ equidistributes to the probability measure 
$$
\mu_\phi:=V^{-1}(dd^c\env(\phi))^n. 
$$
as $m\to\infty$. 
\end{thmD} 
\subsection*{Organization of the paper}
The paper is organized as follows: 
\begin{itemize}
    \item Section 1 contains background material on norms on finite dimensional vector spaces over a complete valued field. We present the results in the Archimedean and non-Archimedean cases as uniformly as possible.
    \item In Section 2 we discuss determinants of norms and their relation to relative spectra.  
    \item Section 3, which stands somewhat apart from the rest of paper, applies the previous result to construct metrics on spaces of norms, following~\cite{Ger}. 
    \item Sections 4 and 5 contain background material on Berkovich spaces and metrics on line bundles. We recall the standard constructions of model metrics, and compare them to Fubini--Study metrics. We discuss the relation between the reduced fiber theorem and finiteness of integral closure. 
    \item Section 6 establishes our first key tool, to wit boundedness of the distortion between the supnorms and lattice norms induced by a model. As a consequence, we obtain a precise description of the behavior of supnorms under ground field extension. 
    \item In Section 7 we study limits of Fubini--Study metrics, compare them to Zhang's definition of semipositive metrics, and discuss the related notion of semipositive envelope. 
    \item In Section 8 we review the Bedford--Taylor/Chambert--Loir--Ducros mixed Monge--Amp\`ere operators, and relate them to Deligne pairings.    
    \item Section 9 contains the proof of Theorem A and Corollary B. 
    \item Section 10 shows the existence of transfinite diameters and equidistribution of Fekete points, \ie Theorem C and Theorem D. We also show how the results can be applied in the case of toric varieties. 
    \item In the Appendix we explain how F.~Ducrot's approach to the Knudsen--Mumford expansion for the determinant of cohomology~\cite{Duc} and the related notion of Deligne pairings can be extended from the usual Noetherian case to arbitrary schemes.
    \end{itemize}
\subsection*{Acknowledgement}
We are grateful for many helpful discussions and remarks from many colleagues. In particular, we thank Robert Berman, Dario Cordero-Erausquin, Antoine Ducros, Gerard Freixas, Gilles Godefroy, Mattias Jonsson, Klaus K\"unnemann,  Marco Maculan, Florent Martin, David Rydh. We are especially grateful to Walter Gubler for pointing out a number of inaccuracies in a first version of this paper, and for his help with several arguments. The first author was partially supported by the ANR project GRACK (ANR-15-CE40-0003). The second author wants to thank the Max-Planck Institut in Bonn, where part of this work was effectuated, for their excellent working conditions and hospitality. 
\newpage

%
%

\part{Spaces of norms and determinants}

\section{Spaces of norms}\label{sec:normsdet}
The goal of this section is to review some basic material on finite dimensional normed vector spaces, treating in parallel the Archimedean and non-Archimedean cases (including the trivially valued case). All the results are more or less well-known, but our proof of density of diagonalizable norms among all norms in the non-Archimedean case (Theorem~\ref{thm:diag}) appears to be new.   
%
%
\subsection{Complete valued fields} 
Here and throughout the article, $K$ denotes a \emph{complete valued field}, \ie a field endowed with a (possibly trivial) absolute value $|\cdot|: K \to \R_{\ge 0}$, with respect to which it becomes a complete metric space. The value group $|K^\times|$ is a subgroup of $\R_{>0}$, and is thus either discrete or dense. We sometimes use the \emph{additive value group} $\Ga_K:=\log|K^\times|\subset\R$. 

Recall that $K$ is \emph{Archimedean} if, for each nonzero $ a\in K$, there exists $n\in\Z$ with $|n a|>1$. This implies that $K$ has characteristic $0$, and hence contains $\Q$, and that the restriction of $|\cdot|$ to $\Q$ is equivalent to the standard absolute $|\cdot|_\infty$, by Ostrowski's theorem. As a result, $K$ is a complete field extension of $\R$, and hence $K=\R$ or $\C$ (up to normalization of the absolute value), by the Gelfand--Mazur theorem. 

Otherwise, $K$ is \emph{non-Archimedean}, and this holds if and only if $|\cdot|$ satisfies the ultrametric inequality $|a+b|\le\max\{|a|,|b|\}$ for all $a,b\in K$. We then have a corresponding real-valued valuation $v_K:=-\log|\cdot|$ on $K$, whose valuation ring $K^\circ$ is thus the closed unit ball of $K$, with maximal ideal $K^{\circ\circ}$ the open unit ball and residue field $\tK:=K^\circ/K^{\circ\circ}$. Since the value group is a subgroup of $\R$, the valuation ring $K^\circ$ is of Krull dimension at most $1$, and it is Noetherian if and only if $|K^\times|$ is discrete. In that case, $K^{\circ\circ}$ is a principal ideal; a generator $\pi_K$ of $K^{\circ\circ}$ is called a uniformizing parameter, and is unique up to multiplication by a unit. 

If on the other hand $K$ is algebraically closed, then $\tK$ is algebraically closed as well, and $|K^\times|$ is divisible. In particular, $K$ is then either trivially valued or densely valued. The completion of an algebraic closure of any non-Archimedean field $K$ is denoted by $\C_K$. It is the smallest complete algebraically closed extension of $K$. 

A field $K$ is \emph{local} if its unit ball $K^\circ$ is compact. This holds if and only if $K$ is either Archimedean, or non-Archimedean with finite residue field $\tK$. In the latter case, $K$ is trivially or discretely valued, and in fact either a finite trivially valued field, or isomorphic to a finite extension of $\Q_p$ or $\F_p{\lau t}$ (up to normalization of the absolute value). 

An \emph{immediate extension} of a non-Archimedean field $K$ is a complete field extension $F/K$ with the same value group and residue field as $K$. The field $K$ is \emph{maximally complete} if it admits no nontrivial immediate extension. By~\cite{Kap}, $K$ is maximally complete iff it is spherically complete, which means that any decreasing sequence of closed balls has non-empty intersection. A discretely valued field is maximally complete, and every algebraically closed field admits an immediate maximally complete extension, unique up to isomorphism~\cite[Theorem 5]{Kap}. 

\begin{exam} Let $k$ be an algebraically closed field of characteristic $0$, and endow the field $K=k{\lau t}$ of formal Laurent series with the $t$-adic valuation. An algebraic closure of $K$ is given by the field of Puiseux series 
$$
k\lau{ t^{1/\infty}}=\bigcup_{n\ge 1} k\lau{t^{1/n}},
$$
whose completion $\C_K$ is realized as the field of formal series $f=\sum_{r\in\Q} a_r t^r$, $a_r\in k$,  with support $\supp f=\{r\in\Q\mid a_r\ne 0\}$ containing only finitely many elements with a given upper bound. The immediate maximally complete extension of $\C_K$ is given by the Malcev--Neumann field $k\lau {t^\Q}$ of power series $f=\sum_{r\in\Q} a_r t^r$ with well-ordered support. Note that $f=\sum_{n\ge 1} t^{-1/n}$ is in $k\lau{t^\Q}\setminus\C_K$, so that $\C_K$ is not maximally complete. 
\end{exam}

\begin{exam} Similarly, the completion $\C_p$ of an algebraic closure of $\Q_p$ is also not maximally complete. Denote by $A$ the Witt ring of $\overline{\mathbb{F}}_p$, \ie the valuation ring of the completion of the maximal unramified extension of $\Q_p$. The immediate maximally complete extension of $\C_p$ is obtained as the quotient of the Malcev--Neumann ring $A\lau {t^\Q}$ by the ideal of formal power series $f=\sum_{r\in\Q} a_r t^r$ such that $\sum_{n\in\Z} a_{r+n}p^n=0$ in $\Z_p$ for all $r\in\Q$, cf.~\cite[\S 4]{Poo}. 
\end{exam}
%
%
\subsection{The space of norms}\label{sec:norms}
Let $V$ be a fixed finite dimensional $K$-vector space, and set $N:=\dim V$. 

\begin{defi}\label{defi:seminorm} A \emph{seminorm} on $V$ is a function $\n:V\to\R_+$ such that
\begin{itemize}
\item[(i)] $\| a v\|=|a|\|v\|$ for all $ a\in K$, $v\in V$; 
\item[(ii)] $\|v+w\|\le\|v\|+\|w\|$ (resp.~$\|v+w\|\le\max\{\|v\|,\|w\|\}$) for all $v,w\in V$ if $K$ is Archimedean (resp.~non-Archimedean). 
\end{itemize}
We say that $\n$ is \emph{pure} if it takes values in $|K|\subset\R_{\ge 0}$. A \emph{norm} is a seminorm $\n$ such that $\|v\|=0\Longleftrightarrow v=0$. We denote by $\cN(V)$ the set of all norms on $V$.
\end{defi}

The group $\GL(V)$ acts on $\cN(V)$ by composition. 

\begin{exam} If $K$ is Archimedean, mapping a norm $\n$ to its closed unit ball $B$ (centered at $0$) sets up a one-to-one correspondence between $\cN(V)$ and the set of all convex bodies of $V$ that are centrally symmetric when $K=\R$, and $S^1$-invariant when $K=\C$. The inverse map is obtained by setting
$$
\|v\|=\inf\left\{r\ge 0\mid v\in rB\right\}.
$$
\end{exam}

\begin{exam}\label{exam:filtr} If $K$ is trivially valued, the closed balls $B_r$, of radius $r$, of a norm on $V$ form an increasing filtration of $V$ by linear subspaces, which is exhaustive ($B_r=V$ for $r\gg 1$), separating ($B_r=\{0\}$ for $r\ll 1$), and right-continuous ($B_r=\bigcap_{r'>r} B_{r'}$). Conversely, any such filtration defines a norm by setting
$$
\|v\|=\inf\{r\ge 0\mid v\in B_r\}.
$$
In other words, the data of a norm with respect to the trivial absolute value is equivalent to that of an increasing flag $\{0\}=V_0\subset V_1\subset\dots\subset V_n=V$ of linear subspaces, together with a increasing sequence $0=r_0<r_1<\dots<r_n$. 
\end{exam}

Equivalence of norms over $\R$ and $\C$ is usually established as a consequence of the compactness of the unit cube. Crucially, equivalence of norms still holds over any complete valued field. 

\begin{prop}\label{prop:equiv} Any two norms $\n$, $\n'$ on $V$ are equivalent, \ie there exists $C>0$ such that $C^{-1}\n\le\n'\le C\n$. 
\end{prop}
\begin{proof} For the convenience of the reader, we repeat the simple standard argument, in order to show that it applies to the trivially valued case as well. Note first that the result implies that $V$ is complete with respect to any norm $\n$. Indeed, after choosing a basis $(e_i)$ of $V$, $\n$ will be equivalent to the $\ell^\infty$-norm $\n_\infty$ associated to $(e_i)$, which is complete since it is isometrically isomorphic to $K^N$. 
We argue by induction on $N=\dim V$, the desired result being trivial for $N=1$. We are going to show that any given norm $\n$ on $V$ is equivalent to $\n_\infty$. For each subspace $W\ne V$, the restriction of $\n$ to $W$ is complete, by induction. As a result, $W$ is closed with respect to $\n$, and hence $\inf_{w\in W}\|v+w\|>0$ for each $v\in V\setminus W$. In particular, 
$$
c_i:=\inf_{ a\in K^N}\left\|e_i+\sum_{j\ne i} a_j e_j\right\|>0
$$
for all $i$. For each $ a\in K^N$ and each $i$ with $ a_i\ne 0$, we get 
$$
\left\|\sum_j a_j e_j\right\|=| a_i|\left\|e_i+\sum_{j\ne i}\frac{ a_j}{ a_i} e_j\right\|\ge c_i| a_i|, 
$$
and hence $\left\|\sum_j a_j e_j\right\|\ge(\min_j c_j)\max_j| a_j|$. By the triangle inequality, we also have $\left\|\sum_j a_j e_i\right\|\le N\max_j| a_j|\|e_i\|$, which proves that $\n$ and $\n_\infty$ are indeed equivalent.  
\end{proof}
As noticed during the proof, each linear subspace $W\subset V$ is closed with respect to any norm $\n$, and $\n$ thus induces a quotient norm $\n_{V/W}$ on $V/W$, defined as usual by 
$$
\|\bar v\|_{V/W}:=\inf_{w\in W}\|v+w\|
$$
for each $v\in V$ with image $\bar v\in V/W$. 

By Proposition~\ref{prop:equiv}, we can endow $\cN(V)$ with the \emph{Goldman--Iwahori} metric $\dGI$ (named after~\cite{GI}), defined by 
\begin{equation}\label{equ:dist}
\dGI\left(\n,\n'\right):=\sup_{v\in V\setminus\{0\}}\left|\log\|v\|-\log\|v\|'\right|.
\end{equation}
The exponential of $\dGI\left(\n,\n'\right)$ is thus the distortion betwen the two norms, \ie the smallest constant $C\ge 1$ such that
$$
C^{-1}\n\le\n'\le C\n.
$$ 
The action of $\GL(V)$ on $\cN(V)$ preserves $\dGI$. 

\begin{exam}\label{exam:norm1D} Assume $\dim V=1$, and pick a nonzero $v\in V$. Then $\n\mapsto\log\|v\|$ defines an isometry $(\cN(V),\dGI)\simeq\R$, and the action of $\GL(V)$ on $\cN(V)$ is equivalent to the action of the additive value group $\Ga_K=\log|K^\times|$ on $\R$ by translation. 
\end{exam}

The basic topological properties of $\cN(V)$ are as follows. 

\begin{prop}\label{prop:topnorm} The metric space $(\cN(V),\dGI)$ is complete. If $K$ is local, \ie its unit ball $K^\circ$ is compact, then any closed bounded subset of $\cN(V)$ is compact. 
\end{prop} 

\begin{proof} If $(\n_n)$ is a Cauchy sequence in $\cN(V)$, then $\log\|v\|_n$ is a Cauchy sequence for each nonzero $v\in V$. We easily conclude that $\n_n$ converges in $\cN(V)$, which proves the first assertion. Assume now that $K$ is local. 
After choosing a basis, we may assume that $V=K^N$, which we equip with the $\ell^\infty$-norm $\n_\infty$. By the triangle inequality, each norm $\n$ with $\dGI(\n,\n_\infty)\le C$ restricts to a $C$-Lipschitz continuous function on the compact set $(K^\circ)^N$. By the Arzel\`a--Ascoli theorem, the closed balls of $\cN(V)$ (and hence any bounded closed subset) are thus compact. 
\end{proof}

\begin{rmk} Conversely, if $\dim V>1$, one can show that $\cN(V)$ is locally compact only if $K^\circ$ is compact.
\end{rmk}
%

\subsection{Diagonalizable norms}\label{sec:diag}
In order to treat in parallel the Archimedean and non-Archimedean cases, we will use the following terminology. 

\begin{defi}\label{defi:diag} A norm $\n$ on $V$ is \emph{diagonalizable} if there exists a basis $(e_i)$ such that we have for all $ a\in K^N$:
\begin{itemize}
\item[(i)] $\|\sum_i  a_i e_i\|^2=\sum_i\| a_i e_i\|^2$ (Archimedean case); 
\item[(ii)] $\|\sum_i a_i e_i\|=\max_i\| a_i e_i\|$ (non-Archimedean case).
\end{itemize}
The basis $(e_i)$ is then said to be \emph{orthogonal} for $\n$, and it is \emph{orthonormal} if it further satisfies $\|e_i\|=1$. We denote by 
$$
\cN^\diag(V)\subset\cN(V)
$$ 
the set of diagonalizable norms. 
\end{defi}
Note that a norm admits an orthonormal basis iff it is diagonalizable and pure (cf.~Definition~\ref{defi:seminorm}). In the Archimedean case, all norms are pure, and a norm is diagonalizable if and only if it derives from a Euclidian/Hermitian scalar product. In the non-Archimedean case, $\cN^\diag(V)$ is a dense subset of $\cN(V)$ (see~\S\ref{sec:dense} below), and $\cN^\diag(V)=\cN(V)$ for several important classes of fields $K$. 

\begin{exam}\label{exam:diagtriv} If $K$ is trivially valued, any norm $\n$ is diagonalizable. Indeed, $\n$ determines a flag of subspaces (cf.~Example~\ref{exam:filtr}), and a basis $(e_i)$ of $V$ is orthogonal for $\n$ if and only if $(e_{\sigma(i)})$ is compatible with the flag, for some permutation $\sigma\in\fS_N$. 
\end{exam}
More generally, we have: 

\begin{lem} Assume that $K$ is non-Archimedean. The following properties are equivalent: 
\begin{itemize}
\item[(i)] $K$ is maximally complete;
\item[(ii)] every norm on a finite dimensional $K$-vector space is diagonalizable.
\end{itemize}
\end{lem}
\begin{proof} (i)$\Longrightarrow$(ii) is proved in~\cite[2.4.2/3]{BGR}. Assume conversely that $K$ is not maximally complete, and pick a nontrivial immediate extension $F/K$. For any choice of $a\in F\setminus K$, we claim that the restriction $\n$ of the absolute value of $F$ to $K+a K\simeq K^2$ is not diagonalizable. Assume the contrary. Since $\n$ takes values in $|F|=|K|$, it then admits an orthonormal basis $(e_1,e_2)$. Using $\widetilde F=\tK$, we find a unit $u\in K^\circ$ such that $e_1-u e_2\in L^{\circ\circ}$, \ie $\|e_1-u e_2\|<1$, contradicting the orthonormality of $(e_1,e_2)$. 
\end{proof}

In the Archimedean case, diagonalizable norms are of course preserved by restriction and quotient. This is also true in the non-Archimedean case (cf.~\cite[2.4.1/5]{BGR}): 

\begin{lem}\label{lem:quotientdiag} Let $0\to V'\to V\to V''\to 0$ be an exact sequence of finite dimensional $K$-vector spaces. If $\n$ is a diagonalizable norm on $V$, then the induced norms on $V'$ and $V''$ are also diagonalizable.
\end{lem}

The following codiagonalization result is crucial for what follows. It will be proved in \S\ref{sec:duality}, after the basic facts on duality have been discussed. 

\begin{prop}\label{prop:codiag} For any two diagonalizable norms $\n,\n'\in\cN^\diag(V)$, there exists a basis $(e_i)$ of $V$ that is orthogonal for both $\n$ and $\n'$. 
\end{prop}
This can be used to give a simple description of the restriction of $\dGI$ to $\cN^\diag(V)$. 
\begin{lem}\label{lem:distdiag} If $\n,\n'\in\cN^\diag(V)$ are codiagonalized in a basis $(e_i)$, then 
$$
\dGI(\n,\n')=\max_i\left|\log\frac{\|e_i\|}{\|e_i\|'}\right|.
$$
If $K$ is non-Archimedean, we have more generally
$$
\dGI(\n,\n')=\log\max\left\{\max_i\frac{\|e_i\|'}{\|e_i\|},\max_i\frac{\|e'_i\|}{\|e'_i\|'}\right\}
$$
whenever $(e_i)$ and $(e'_i)$ are orthogonal bases for $\n$ and $\n'$, respectively. 
\end{lem} 

\begin{proof} Assume first that $K$ is Archimedean. We trivially have 
$$
\dGI(\n,\n')=\sup_{v\in V\setminus\{0\}}\left|\log\frac{\|v\|}{\|v\|'}\right|\ge m:=\max_i\left|\log\frac{\|e_i\|}{\|e_i\|'}\right|
$$
Consider conversely $v=\sum_i a_i e_i$ with $ a\in K^N$. Then 
$$
\|v\|'^2=\sum_i| a_i|^2\|e_i\|'^2\le e^{2m}\sum_i| a_i|\|e_i\|^2=e^{2m}\|v\|^2. 
$$
By symmetry, this shows that $e^{-m}\n\le\n'\le e^m\n$, \ie $\dGI(\n,\n')\le m$. In the non-Archimedean case the result follows from Lemma~\ref{lem:supdiag} below. 
\end{proof}

\begin{lem}\label{lem:supdiag} Assume that $K$ is non-Archimedean. Let $\n$ be a diagonalizable norm with orthogonal basis $(e_i)$, and $\n'$ be any (ultrametric) seminorm on $V$. Then 
$$
\sup_{v\in V\setminus\{0\}}\frac{\|v\|'}{\|v\|}=\max_i\frac{\|e_i\|'}{\|e_i\|}.
$$
\end{lem}
\begin{proof} As above, we trivially have $\sup_{v\in V\setminus\{0\}}\|v\|'/\|v\|\ge m:=\max_i\|e_i\|'/\|e_i\|$, and $v=\sum_i a_i e_i$ satisfies $\|v\|'\le\max_i| a_i|\|e_i\|'\le m\max_i| a_i|\|e_i\|=m\|v\|$. 
\end{proof}
We now discuss in more detail the structure of the set $\cN^\diag(V)$ of diagonalizable norms. To each basis $\be=(e_i)$ of $V$ is associated an injective map
$$
\iota_\be:\R^N\hookrightarrow\cN^\diag(V),
$$
which takes $\la\in\R^N$ to the unique norm $\n_{\be,\la}$ that is diagonalized in $(e_i)$ and such that $\|e_i\|_{\bf e,\la}=e^{-\la_i}$. The image 
$$
\A_\be:=\iota_\be(\R^N)\subset\cN^\diag(V)
$$
is thus the set of norms that are diagonalized in the given basis $\be$, and is called an \emph{apartment} (or \emph{flat}) of $\cN^\diag(V)$. By definition, $\cN^\diag(V)=\bigcup_\be\A_\be$, and the Goldman--Iwahori metric $\dGI$ can then be characterized as follows. 

\begin{prop}\label{prop:GIisom} The restriction of $\dGI$ to $\cN^\diag(V)$ is the unique metric such that each $\iota_\be:\R^N\hookrightarrow\cN^\diag(V)$ becomes an isometric embedding with respect to the $\ell^\infty$-norm on $\R^N$. 
\end{prop}
\begin{proof} Each $\iota_\be$ is an isometric embedding by Lemma~\ref{lem:distdiag}, and uniqueness follows from the fact that any two points of $\cN^\diag(V)$ belong to the image of some $\iota_\be$, by codiagonalization (Proposition~\ref{prop:codiag}).
\end{proof}

This picture will be generalized to any symmetric norm on $\R^N$ (and in particular to the $\ell^2$-norm) in Section~\ref{sec:building}, leading to the description of $\cN^\diag(V)$ as a Riemannian symmetric space/Euclidian building. In the present setting, the general construction of retractions onto an apartment in building theory specializes as follows (compare~\cite{Ger}). 

\begin{defi}\label{defi:retraction} Let $\be=(e_i)$ be a basis of $V$, with apartment $\A_\be=\iota_\be(\R^N)\subset\cN^\diag(V)$. The \emph{Gram--Schmidt projection} $\rho_\be:\cN(V)\to\A_\be$ is defined by sending a norm $\n$ to the unique norm $\n_\be$ that is diagonalized in $\be$ and such that
$$
\|e_i\|_\be=\inf_{ a\in K^N}\left\|e_i+\sum_{j<i} a_j e_j\right\|
$$
for $i=1,\ldots,N$. 
\end{defi}
Setting $W_i:=\Vect(e_1,\ldots,e_i)$ defines a complete flag $W_\bullet$ in $V$, and $\n$ induces a subquotient norm on each graded piece $W_i/W_{i-1}$, and hence a diagonalizable norm on the graded object $\Gr V=\bigoplus_{1\le i\le N} W_i/W_{i-1}$. The norm $\n_\be$ can then be described as the corresponding norm on $V$ under the isomorphism $V\simeq\Gr V$ defined by $(e_i)$. It is straightforward to see that $\rho_\be:\cN(V)\to\A_\be$ is a retraction, \ie restricts to the identity on $\A_\be$. 

The chosen terminology comes from the Archimedean case, where the Gram--Schmidt orthogonalization process associates to a Euclidian/Hermitian norm $\n$ and a basis $(e_i)$ the orthogonal basis $(e'_i)$ obtained by projection of each $e_i$ orthogonal to $W_{i-1}$, which satisfies $\|e_i\|_\be=\|e'_i\|$. 

%
%
\subsection{Approximation by diagonalizable norms}\label{sec:dense}
The goal of this section is to study the closure in $\cN(V)$ of the set $\cN^\diag(V)$ of diagonalizable norms. 

\begin{thm}\label{thm:diag} The space of diagonalizable norms $\cN^\diag(V)$ satisfies the following properties. 
\begin{itemize}
\item[(i)] If $K$ is Archimedean, then $\cN^\diag(V)$ is closed in $\cN(V)$, and each norm $\n\in\cN(V)$ is at distance at most $\tfrac 12\log N$ of $\cN^\diag(V)$. 
\item[(ii)] If $K$ is non-Archimedean, then $\cN^\diag(V)$ is dense in $\cN(V)$.
\end{itemize}
\end{thm} 

Closedness in (i) follows from the fact that Euclidian/Hermitian norms are characterized by the pararallelogram law 
$$
\|u+v\|^2+\|u-v\|^2=2\|u\|^2+2\|v\|^2,  
$$ 
and the second half of (i) can be deduced from the John ellipsoid theorem; one can also use the simpler Auerbach lemma, whose proof will be basically repeated below. Density in (ii) is equivalent to the existence of $\a$-cartesian bases in the sense of~\cite[2.6.1/3]{BGR}, which will be recovered below by imitating the Auerbach argument.\\

Denote by $V^\vee$ the dual of $V$, and by $\det V^\vee=\bigwedge^N V^\vee$ its determinant line. Viewing an element $\om\in\det V^\vee$ as a multilinear form on $V$, we define its \emph{operator norm} as
$$
\|\om\|_\op:=\sup_{(v_1,\ldots,v_N)\in (V\setminus\{0\})^N}\frac{|\om(v_1,\ldots,v_N)|}{\|v_1\|\cdot\ldots\cdot\|v_N\|}. 
$$
This supremum is indeed finite by equivalence of norms. 

\begin{lem}\label{lem:Auer} Let $\n$ be a norm on $V$, and pick a nonzero $\om\in\det V^\vee$. For each basis $(e_i)$ of $V$ and all $ a\in K^N$, we then have 
$$
\max_i\| a_i e_i\|\le\left(\frac{\|\om\|_\op\|e_1\|\cdot\ldots\cdot\|e_N\|}{|\om(e_1,\ldots,e_N)|}\right)\|\sum_i a_i e_i\|. 
$$
\end{lem}
\begin{proof} The dual basis $(e_i^\vee)$ satisfies 
$$
\langle e_i^\vee,v\rangle=\frac{\om(e_1,\ldots,e_{i-1},v,e_{i+1},\ldots,e_N)}{\om(e_1,\ldots,e_N)}, 
$$
and hence
$$
\max_i|\langle e_i^\vee,v\rangle|\|e_i\|\le\left(\frac{\|\om\|_\op\|e_1\|\dots\|e_N\|}{|\om(e_1,\ldots,e_N)|}\right)\|v\|,
$$
which is equivalent to the desired result. 
\end{proof}

\begin{proof}[Proof of Theorem~\ref{thm:diag}] Assume first that $K$ is Archimedean. As noted above, closedness in (i) follows from the characterization of diagonalizable norms in terms of the parallelogram law. Let $\n$ be any norm on $V$, and fix a nonzero determinant $\om\in\det V^\vee$. By compactness, we may choose a basis $(e_i)$ of $V$ with $\|e_i\|=1$ and $\|\om\|_\op=|\om(e_1,\ldots,e_N)|$. For each $p\in[1,\infty]$, denote by $\n_p$ the $\ell^p$-norm in the basis $(e_i)$. Lemma~\ref{lem:Auer} and the triangle inequality yield $\n_\infty\le\n\le\n_1$. Since $N^{-1/2}\n_1\le\n_2\le N^{1/2}\n_\infty$, it follows that $N^{-1/2}\n_2\le\n\le N^{1/2}\n_2$, and hence $\dGI(\n,\n_2)\le\tfrac12\log N$, which proves (i) since $\n_2\in\cN^\diag(V)$. 

Assume now that $K$ is non-Archimedean, and pick any norm $\n\in\cN(V)$. For any $\e>0$, there exists a basis $(e_i)$ such that 
$$
\|\om\|_\op\le(1+\e)\frac{|\om(e_1,\ldots,e_N)|}{\|e_1\|\dots\|e_N\|}, 
$$
and Lemma~\ref{lem:Auer} yields 
\begin{equation}\label{equ:Auer}
\|\sum_i a_i e_i\|\le\max_i\| a_i e_i\|\le(1+\e)\|\sum_i a_i e_i\|
\end{equation}
for all $ a\in K^N$. Denoting by $\n'$ the norm diagonalized in $(e_i)$ and such that $\|e_i\|'=\|e_i\|$, we infer  $\n\le\n'\le(1+\e)\n$, hence $\dGI\left(\n,\n'\right)\le\log(1+\e)$, and $\n$ is thus in the closure of $\cN^\diag(V)$. 
\end{proof}
%
%
\subsection{Duality}\label{sec:duality}
To each norm $\|\cdot\|$ on $V$ is associated a dual norm $\|\cdot\|^\vee$ on the dual vector space $V^\vee$, defined by the usual formula 
$$
\|\mu\|^\vee=\sup_{v\in V\setminus\{0\}}\frac{|\langle\mu,v\rangle|}{\|v\|}.
$$
Again, the supremum is finite by equivalence of norms. 

\begin{thm}\label{thm:dual} The duality map $\cN(V)\to\cN(V^\vee)$ is an involutive isometry with respect to the Goldman--Iwahori distances. 
\end{thm}

\begin{lem}\label{lem:dualdiag} If $\n\in\cN^\diag(V)$ is diagonalizable, then so is $\n^\vee$. Further, if $(e_i)$ is an orthogonal basis for $\n$, then the dual basis $(e_i^\vee)$ is orthogonal for $\n^\vee$, and $\|e_i^{\vee}\|^\vee=\|e_i\|^{-1}$. 
\end{lem} 
\begin{proof} In the Archimedean case, $\n$ is Euclidian/Hermitian, and the result is well-known. In the non-Archimedean case, the result is a simple consequence of Lemma~\ref{lem:supdiag}. 
\end{proof}

\begin{proof}[Proof of Theorem~\ref{thm:dual}] It is straightforward to see that $\n\mapsto\n^\vee$ is $1$-Lipschitz, it is enough to show that $(\n^\vee)^\vee=\n$ for all $\n\in\cN(V)$. In the Archimedean case, this is a consequence of the Hahn--Banach theorem. In the non-Archimedean case, it follows from Lemma~\ref{lem:dualdiag} and the density of diagonalizable norms in $\cN(V)$ (Theorem~\ref{thm:diag}). 
\end{proof}

\begin{lem}\label{lem:dualsub} If $W\subset V$ is a linear subspace, the canonical embedding $(V/W)^\vee\hookrightarrow V^\vee$ identifies the dual of the quotient norm $\n_{V/W}$ with the restriction of $\n^\vee$. 
\end{lem}
\begin{proof} The image of $(V/W)^\vee$ in $V^\vee$ is the space $W^\perp$ of linear forms $\mu\in V^\vee$ that vanish on $W$. Denoting by $\bar v\in V/W$ the image of $v\in V$ and by $\tilde\mu\in V^\vee$ the image of $\mu\in (V/W)^\vee$, we have by definition
$$
\|\tilde\mu\|^\vee=\sup_{v\in V\setminus\{0\}}\frac{|\langle\tilde\mu,v\rangle|}{\|v\|}
$$
and
$$
\|\mu\|_{V/W}^\vee=\sup_{v\in V- W}\frac{|\langle\tilde\mu,v\rangle|}{\|\bar v\|_{V/W}}. 
$$

Since $\|\bar v\|_{V/W}=\inf_{w\in W}\|v+w\|\le\|v\|$ and $\langle\tilde\mu,v\rangle=0$ for $v\in W$, we trivially have $\|\tilde\mu\|^\vee\le\|\mu\|_{V/W}^\vee$. Conversely, we have for each $v\in V- W$ and $w\in W$ 
$$
\frac{|\langle\tilde\mu,v\rangle|}{\|v+w\|}=\frac{|\langle\tilde\mu,v+w\rangle|}{\|v+w\|}\le\|\tilde\mu\|^\vee, 
$$
hence 
$$
\frac{|\langle\tilde\mu,v\rangle|}{\|\bar v\|_{V/W}}=\sup_{w\in W}\frac{|\langle\tilde \mu,v\rangle|}{\|v+w\|}\le\|\tilde\mu\|^\vee, $$
and taking the supremum over $v$ yields the inequality in the other direction $\|\mu\|_{V/W}^\vee \le \|\tilde \mu\|^\vee$ and we conclude.
\end{proof}

We are now in a position to prove the codiagonalization result promised in Proposition~\ref{prop:codiag}. 

\begin{proof}[Proof of Proposition~\ref{prop:codiag}] That any two diagonalizable norms $\n,\n'\in\cN^\diag(V)$ are codiagonalizable is a standard fact in the Archimedean case, and we henceforth assume that $K$ is non-Archimedean. Our argument extends the classical one of~\cite{GI}, which treats the case of a local field, following a suggestion of Marco Maculan, whom we warmly thank. Recall that a direct sum decomposition $V=V_1\oplus\dots\oplus V_r$ is \emph{orthogonal for $\n$} if 
$$
\left\|\sum_i v_i\right\|=\max_i\|v_i\|
$$
for all $v_i\in V_i$. Given $v\in V$ and a linear form $\mu\in V^\vee$ with $\langle\mu,v\rangle\ne 0$, it is straightforward to check that the decomposition $V=Kv\oplus\Ker\mu$ is orthogonal for $\n$ if and only if 
$$
\frac{|\langle\mu,w\rangle|}{|\langle\mu,v\rangle|}\le\frac{\|w\|}{\|v\|}
$$
for all $w\in V$. Arguing by induction on $\dim V$, we will thus be done if we prove the existence of $v\in V$ and $\mu\in V^\vee$ with $\langle \mu,v\rangle\ne 0$ such that
\begin{equation}\label{equ:diagsup}
\frac{|\langle\mu,w\rangle|}{|\langle\mu,v\rangle|}\le\frac{\|w\|}{\|v\|}\le\frac{\|w\|'}{\|v\|'}
\end{equation}
for all $w\in V$, since $V=Kv\oplus\ker\mu$ will then be orthogonal for both $\n$ and $\n'$. Let $(e'_i)$ be an orthogonal basis for $\n'$. By Lemma~\ref{lem:supdiag}, we have  
$$
\sup_{w\in W\setminus\{0\}}\frac{\|w\|}{\|w\|'}=\frac{\|e'_i\|}{\|e'_i\|'}
$$
for some $i$, and $v:=e'_i$ therefore satisfies $\|w\|/\|v\|\le\|w\|'/\|v\|'$ for all $w\in V$. Let now $(e_j)$ be an orthogonal basis for $\n$. Since the dual basis $(e_j^\vee)$ is orthogonal for the dual norm $\n^\vee$, we similarly get 
$$
\|v\|=\sup_{\mu\in V^\vee\setminus\{0\}}\frac{|\langle\mu,v\rangle|}{\|\mu\|^\vee}=\frac{|\langle e_j^\vee,v\rangle|}{\|e_j^\vee\|^\vee}
$$
for some $j$. It follows that $\mu:=e_j^\vee$ satisfies 
$$
|\langle\mu,v\rangle|=\|\mu\|^\vee\|v\|\ge\frac{|\langle\mu,w\rangle|}{\|w\|}\|v\|
$$
for all $w\in V\setminus\{0\}$, and (\ref{equ:diagsup}) follows. 
\end{proof}

%
%
\subsection{Ground field extension}\label{sec:ground}
Let $V$ be a finite dimensional $K$-vector space $V$, $F/K$ be a complete field extension, and set $V_F:=V\otimes_K F$. When $K$ is Archimedean, the only nontrivial case is $K=\R$, $F=\C$, by the Gelfand--Mazur theorem. 

\begin{defi}\label{defi:ground} The \emph{ground field extension} of a norm $\n$ on $V$ is the norm $\n_F$ on $V_F$ defined by
\begin{itemize}
\item $\|w\|_F:=\inf\sum_i |b_i|\|v_i\|$ if if $K$ is Archimedean; 
\item $\|w\|_F:=\inf\max_i |b_i|\|v_i\|$ if $K$ is non-Archimedean; 
\end{itemize}
where the infimum ranges in both cases over all decompositions $w=\sum_i b_i v_i$ with $b_i\in F$ and $v_i\in V$.
\end{defi}

\begin{prop}\label{prop:ground} Let $F/K$ be a complete field extension. 
\begin{itemize}
\item[(i)] For any norm $\n$ on $V$, the restriction of $\n_F$ to $V$ coincides with $\n$, and $\n_F$ is the maximal norm on $V_F$ with this property. 
\item[(ii)] The map $\cN(V)\to\cN(V_F)$ $\n\mapsto\n_F$ is an isometric embedding with respect to the Goldman--Iwahori distances.
\item[(iii)] If $K=\R$ and $F=\C$, then $\n_\C$ is conjugation invariant, and any conjugation invariant norm $\n'$ on $V_\C$ that coincides with $\n$ on $V$ satisfies 
$$
\tfrac 12\n_\C\le\n'\le\n_\C.
$$
\item[(iv)] Assume that $K$ is non-Archimedean, and let $\n$ be a diagonalizable norm with orthogonal basis $(e_i)$. Then $\n_F$ is also a diagonalizable norm with orthogonal basis $(e_i)$, viewed as a basis of $V_F$, and $\|e_i\|_F=\|e_i\|$. 
\item[(v)] Assume further that $\tilde F=\tilde K$ and 
\begin{equation}\label{equ:normindep}
\left\{\|e_i\|/\|e_i\|\mid 1\le i,j\le N\right\}\cap|F^\times|=\{1\}.
\end{equation}
Then $\n_F$ is the only norm that extends $\n$ to $V_F$. 
\end{itemize}
\end{prop}
Note that (v) is direct generalization of~\cite[Lemma 1.12]{CM}. 

\begin{proof} (i) is immediate, and implies for any two norms $\n,\n'$ on $V$
$$
\dGI(\n_F,\n'_F)\ge\sup_{v\in V\setminus\{0\}}|\log\|v\|-\log\|v\|'|=\dGI(\n,\n').
$$
The converse inequality $\dGI(\n_F,\n'_F)\le\dGI(\n,\n')$ is straightforward from the definition, hence (ii). 

To prove (iii), pick any conjugation invariant norm $\n'$ on $V_\C$ that coincides with $\n$ on $V$. For each $w\in V_\C$ we then have $\|w\|'=\|\bar w\|'$, hence $\|\Rea w\|=\|\Rea w\|'=\|(w+\bar w)/2\|'\le\|w\|'$, $\|\Ima w\|=\|\Ima w\|'\le\|w\|'$. Since $w=\Rea w+i\Ima w$, we infer
$$
\|w\|_\C\le\|\Rea w\|+\|\Ima w\|\le 2\|w\|'.
$$
To see (iv), pick $w\in V_F$, and write $w=\sum_j b_j e_j$ with $b_j\in F$. By definition, $\|w\|_F\le\max_j|b_j|\|e_j\|$. Conversely, pick any decomposition $v=\sum_i b_i v_i$ with $b_i\in F$ and $v_i\in V$, and write $v_i=\sum_j a_{ij} e_j$ with $a_{ij}\in K$. Then $v=\sum_j c_j e_j$ with $c_j=\sum_i b_i b_{ij}$, and we need to show that
$$
\max_j |c_j|\|e_j\|\le\max_i |b_i|\|v_i\|.
$$
This follows indeed from $\|v_i\|=\max_j |a_{ij}|\|e_j\|$ and $|c_j|\le\max_i|b_i||a_{ij}|$. 

Now assume that $\tilde F=\tilde K$ and~\eqref{equ:normindep} hold, and pick an extension $\n'$ of $\n$ to $V_F$. Given a nonzero tuple $(b_i)$ in $F$, we need to show that 
$$
\|\sum_i b_i e_i\|'=c:=\max_i|b_i|\|e_i\|.
$$
After reindexing, we may assume that 
$$
c=|b_1|\|e_1\|=\dots=|b_r|\|e_r\|>|b_{r+1}|\|e_{r+1}\|\ge\dots
$$
For $1\le i\le r$, we have $\|e_i\|/\|e_1\|=|b_1 b_i^{-1}|\in |F^\times|$, and~\eqref{equ:normindep} thus implies $\|e_i\|=\|e_1\|$, and hence also $|b_i b_1^{-1}|=1$. Since $\tK=\tF$, we can pick a unit $u_i\in K^\circ$ such that $|b_ib_1^{-1}-u_i|<1$. 
Set $v_1:=b_1\sum_{i\le r} u_i e_i$, $v_2=\sum_{i\le r}(b_i-b_1 u_i)e_i$, and $v_3:=\sum_{i>r} b_i e_i$, so that $v_1+v_2+v_3=\sum_i b_i e_i$. 
Then 
$$
\|v_1\|'=|b_1|\|\sum_{i\le r} u_i e_i\|'=|b_1|\|\sum_{i\le r} u_i e_i\|=|b_1|\max_{i\le r}|u_i|\|e_i\|=c,
$$
$$
\|v_2\|'<|b_1|\max_{i\le r}\|e_i\|=c,
$$
and 
$$
\|v_3\|'\le\max_{i>r}|b_i|\|e_i\|<|b_1|\|e_1\|=c, 
$$
and we infer as desired $\|\sum_i b_i e_i\|'=\|v_1+v_2+v_3\|'=c$. 
\end{proof}

As in the proof of~\cite[Theorem 4.1]{CM}, we infer: 

\begin{lem}\label{lem:CM} Assume that $K$ is trivially valued, and let $(V_i,\n_i)_{i\in I}$ be an at most countable family of norms on finite dimensional $K$-vector spaces. We may then find a complete extension $F/K$ with $F$ nontrivially valued such that for each $ i\in I$, $(\n_i)_F$ is the only norm on $(V_i)_F$ that coincides with $\n_i$ on $V_i$. 
\end{lem}
\begin{proof} For each $i$, pick an orthogonal basis $(e_{ij})_{j\in J_i}$ for $(V_i,\n_i)$ (see Example~\ref{exam:diagtriv}), and consider the finite set 
$$
S_i:=\left\{\log\|e_{ij}\|-\log\|e_{ij'}\|\mid j,j'\in J_i\right\}.
$$
Set $F:=K\lau{t}$, pick $\a\in\R_{>0}$, and endow $F$ with the valuation $v_F$ equal to $\a$ times the $t$-adic valuation. Then $\log|F^\times|=v_F(F^\times)=\a\Z$ and $\tilde F=\tilde K$. According to Proposition~\ref{prop:ground}~(v), it will thus be enough to show that $\a$ can be chosen so that $\a\Z\cap S_i=\{0\}$ for all $i$. Now $\bigcup_i\Q S_i$ is countable, and it is then enough to choose $\a\in\R_{>0}\setminus\bigcup_i \Q S_i$. 
\end{proof}

\subsection{Lattice norms}
In this section, $K$ is non-Archimedean, with associated real-valued valuation $v_K=-\log|\cdot|$. As for any valuation ring, $K^\circ$ has the property that every finitely generated ideal is principal, which implies that a $K^\circ$-module $M$ is flat if and only if it is torsion-free. If $M$ is further finitely generated, then it is free (since $K^\circ$ is local). 

\begin{defi} A \emph{lattice}\footnote{In the densely valued case, the present notion of lattice is more restrictive than the one used in~\cite[\S 1.3.3]{CM}.} of $V$ is a finite $K^\circ$-submodule $\cV$ of $V$ such that $\cV\otimes_{K^\circ}K=V$
\end{defi}
Any lattice of $V$ is thus of the form $\cV=\sum_i K^\circ e_i$ with $(e_i)$ a basis of $V$. A lattice $\cV$ determines a \emph{lattice norm} $\n_\cV$ on $V$, by setting 
$$
\|v\|_\cV:=\inf\left\{|a|\mid a\in K,\,v\in a\cV\right\}. 
$$
\begin{lem}\label{lem:latticenorm} Let $(e_i)$ be a $K^\circ$-basis of a lattice $\cV$ of $V$. Then $(e_i)$ is an orthonormal basis of $\n_\cV$. In particular, $\cV$ coincides with the unit ball of $\n_\cV$. 
\end{lem}
\begin{proof} Pick $v\in V$, and write $v=\sum_i a_i e_i$ with $ a\in K^N$. Given $ a\in K$, we then have $v\in a\cV$ if and only if $| a_i|\le|a|$, and hence $\|v\|_\cV=\max_i| a_i|$. This means that $(e_i)$ is orthonormal for $\n_\cV$, and also implies that $\cV$ is the unit ball of $\n_\cV$. 
\end{proof}

\begin{lem}\label{lem:latticedense} Denote by $\cN^{\latt}(V)\subset\cN^\diag(V)$ the set of lattice norms. 
\begin{itemize}
\item[(i)] A norm is a lattice norm if and only if it is a pure diagonalizable norm, \ie a norm that admits an orthonormal basis. 
\item[(ii)] If $K$ is trivially valued, then $\cN^{\latt}(V)$ reduces to the trivial norm on $V$. If $K$ is discretely valued, with uniformizing parameter $\pi_K$, then $\cN^{\latt}(V)$ is discrete and closed in $\cN(V)=\cN^\diag(V)$. Further, the (closed) unit ball $B$ of any norm $\n$ is a lattice, whose associated lattice norm $\n_B$ satisfies
$$
\dGI(\n,\n_B)\le  v_K(\pi_K). 
$$ 
\item[(iii)] If $K$ is densely valued, $\cN^{\latt}(V)$ is dense in $\cN^\diag(V)$, and hence also in $\cN(V)$. Further, the unit ball $B$ of a norm $\n$ is a lattice if and only $\n$ is a lattice norm. 
\item[(iv)] Let $\n$ be the lattice norm determined by a lattice $\cV$ of $V$, and $F/K$ be a complete extension. Then the ground field extension $\n_F$ is the lattice norm determined by the lattice $\cV\otimes_{K^\circ} F^\circ$ of $V_F=V\otimes_K F$. 
\end{itemize}
\end{lem}
\begin{proof} (i) is a direct consequence of Lemma~\ref{lem:latticenorm}, and implies the first part of (ii). Assume that $K$ is discretely valued. Let $\n_\cV\ne\n_{\cV'}$ be two distinct lattice norms, and pick a joint orthogonal basis $(e_i)$ as in Proposition~\ref{prop:codiag}. We then have $\cV=\sum_i K^\circ\pi_K^{m_i} e_i$ and $\cV'=\sum_i K^\circ\pi_K^{m'_i} e_i$ for some integers $m_i,m'_i\in\Z$, and hence
$$
\dGI(\n_\cV,\n_{\cV'})=\max_i\left|\log\frac{\|e_i\|_{\cV}}{\|e_i\|_{\cV'}}\right|
$$
$$
= v_K(\pi_K)\max_i|m_i-m'_i|\ge  v_K(\pi_K). 
$$ 
This shows that $\cN^{\latt}(V)$ is discrete and closed. Next, let $\n$ be any diagonalizable norm, pick an orthogonal basis $(e_i)$ for $\n$, and write a given $v\in V$ as $v=\sum_iu_i\pi_K^{n_i} e_i$ with $n_i\in\Z$ and $u_i$ a unit. Then $\|v\|\le 1$ if and only if $|\pi_K|^{n_i}\|e_i\|\le 1$ for all $i$, and we infer $B=\sum_i K^\circ\pi_K^{m_i}e_i$ with $m_i:=\lceil\log\|e_i\|/ v_K(\pi_K)\rceil$. In particular, $B$ is a lattice with basis $(\pi_K^{m_i}e_i)$, and hence
$$
\dGI(\n,\n_B)=\max_i\left|\log\frac{\|e_i\|_B}{\|e_i\|}\right|=\max_i\left|m_i v_K(\pi_K)-\log\|e_i\|\right|\le v_K(\pi_K). 
$$
Assume finally that $K$ is densely valued. That $\cN^{\latt}(V)$ is dense in $\cN^{\diag}(V)$ is easily seen by approximating the values of a given diagonalizable norm $\n$ on an orthogonal basis $(e_i)$ by elements of the dense subset $|K^\times|$ of $\R_{>0}$. Similarly, any norm $\n$ is determined by its closed unit ball $B$, via
$$
\|v\|=\inf\left\{|a|\mid a\in K,\,v\in a B\right\}.
$$
As a result, $\n$ is a lattice norm if and only if $B$ is a lattice. Finally, (iv) is a direct consequence of Lemma~\ref{lem:latticenorm} and Proposition~\ref{prop:ground}. 
\end{proof} 

As an illustration of these considerations, we have: 
\begin{exam} Let $K$ be a densely valued non-Archimedean field, and $F/K$ be a finite extension. If the ring extension $F^\circ/K^\circ$ is finite, then $F/K$ is necessarily unramified, \ie $|F|=|K|$. Indeed, the absolute value $|\cdot|_L$ of $L$ is then a lattice norm, by Lemma~\ref{lem:latticedense}~(iii), and hence $|L|=|K|$. 
\end{exam}

%
%
\section{Determinants and relative spectra}\label{sec:det}
%
%
The goal of this section is to investigate induced norms on the determinant line, leading to the notion of \emph{relative volume} of two norms. We relate the latter to the relative spectrum via a Minkowski-type theorem.  

As before, $V$ denotes a finite dimensional vector space over a complete valued field $K$, and we set $N:=\dim V$. 
 %
%
\subsection{The determinant of a norm} \label{subsec:detnorm}
The \emph{determinant line} of $V$ is $\det V:=\bigwedge^N V$. We have a natural isomorphism  $\det(V^\vee)\simeq(\det V)^\vee$, induced by the pairing
$$
\langle v_1\wedge\ldots\wedge v_N,\mu_1\wedge\ldots\wedge\mu_N\rangle=\det\left(\langle v_i,\mu_j\rangle\right). 
$$
In particular, if $(e_i)$ is basis of $V$ with dual basis $(e_i^\vee)$, then 
$$
\left(e_1\wedge\ldots\wedge e_N\right)^\vee=e_1^\vee\wedge\ldots\wedge e_N^\vee. 
$$
\begin{defi}\label{defi:det} To each norm $\n$ on $V$, we associate a norm $\det\n$ on $\det V$ by setting for $\tau\in\det V$
$$
\det\|\tau\|=\inf_{\tau=v_1\wedge\ldots\wedge v_N}\prod_i\|v_i\|,
$$
where the infimum runs over all decompositions $\tau=v_1\wedge\ldots\wedge v_N$ with $v_i\in V$. 
\end{defi}
We abuse the notation slightly by writing $\det\|\tau\|$ for the value of $\det\|\cdot\|$ on $\tau$. By construction, $\det\n$ is the largest seminorm on $\det V$ with the submultiplicativity property
\begin{equation}\label{eq:detsubmult}
\det\|v_1\wedge\ldots\wedge v_N\|\le\prod_i\|v_i\|
\end{equation}
for all $v_1,\dots,v_N\in V$. That it is actually a norm follows from the next result, which is readily checked.  

\begin{lem}\label{lem:detdual} If we view the dual $\tau^\vee\in\det V^\vee$ of a nonzero $\tau\in\det V$ as a multilinear form on $V$, then $\left(\det\|\tau\|\right)^{-1}$ coincides with the operator norm
$$
\|\tau^\vee\|_\op:=\sup_{v_1,\ldots,v_N\in V\setminus\{0\}}\frac{|\tau^\vee(v_1,\ldots,v_N)|}{\|v_1\|\dots\|v_N\|}.
$$
\end{lem}
From the definition, we immediately get: 
\begin{lem}\label{lem:detlip} The map $\det:\cN(V)\to\cN(\det V)$ is $N$-Lipschitz continuous with respect to $\dGI$-metrics, \ie we have 
$$
\dGI(\det\n,\det\n')\le N\dGI(\n,\n')
$$
for any two norms $\n,\n'$ on $V$.
\end{lem}

Computing the determinant of a norm is typically a hard problem in the Archimedean case. 

\begin{exam}\label{exam:detarch} Consider the usual $\ell^p$-norm $\n_p$ on $\R^N$, $p\in[1,\infty]$, and set $\tau:=e_1\wedge\ldots\wedge e_N$, with $(e_i)$ the canonical basis. Then 
\begin{itemize}
\item $\det\|\tau\|_p=1$ for $p\in[1,2]$;
\item $\det\|\tau\|_p<1$ for $p>2$. 
\end{itemize}
For $p=\infty$, determining the precise value of $\det\|\tau\|_\infty$ amounts to maximizing the determinant of a $N\times N$-matrix with entries in $\{\pm 1\}$, and is known as the \emph{Hadamard maximal determinant problem}. By~\cite{CL}, we have for instance 
$$
N^{-\frac{N}{2}}\le\det\|\tau\|_\infty\le N^{-\frac{N}{2}\left(1-\frac{\log(4/3)}{\log N}\right)}. 
$$
The lower bound is achieved if and only if there exists an $\ell^2$-orthogonal basis with entries in $\{\pm 1\}$ (which implies that $N$ is a multiple of $4$), but the exact value $\det\|\tau\|_\infty$ is unknown in the general case. 
\end{exam}

%
%
\subsection{Determinants of diagonalizable norms}\label{det:diagnorms}
As we shall see in this section, the determinant of a diagonalizable norm is very well-behaved. In the non-Archimedean case, this will extend to all norms, by density of diagonalizable norms. 

\begin{lem}\label{lem:detdiag} If $\n$ is diagonalizable, then a basis $(e_i)$ of $V$ satisfies
$$
\det\|e_1\wedge\ldots\wedge e_N\|=\prod_i\|e_i\|
$$
if and only if $(e_i)$ is orthogonal for $\n$. 
\end{lem}

\begin{cor}\label{cor:detmodel} Assume that $K$ is non-Archimedean, and let $\cV$ be a lattice of $V$. Then $\det\n_\cV=\n_{\det\cV}$ is the norm determined by the lattice $\det\cV:=\bigwedge^N\cV$ of $\det V$. 
\end{cor}

\begin{proof}[Proof of Lemma~\ref{lem:detdiag}] When $K$ is Archimedean, the result is equivalent to the classical Hadamard inequality for the determinant of a matrix. Assume now that $K$ is non-Archimedean, and let $(e_i)$ be an orthogonal basis for $\n$.  We need to show that each basis $(v_i)$ such that $v_1\wedge\ldots\wedge v_N=e_1\wedge\ldots\wedge e_N$ satisfies $\prod_i\|e_i\|\le\prod_i\|v_i\|$. If we write $v_i=\sum_j a_{ij} e_j$ with $ a_{ij}\in K$, then $\det( a_{ij})=1$. Expanding out the determinant and using the ultrametric inequality, we get $\prod_i | a_{i\sigma(i)}|\ge 1$ for some permutation $\sigma$. Since $\n$ is diagonalized in $(e_i)$, we have 
$$
\|v_i\|=\max_j| a_{ij}|\|e_j\|\ge \left| a_{i\sigma(i)}\right|\left\|e_{\sigma(i)}\right\|,
$$ 
and we obtain as desired
$$
\prod_i\|v_i\|\ge\prod_i\left| a_{i\sigma(i)}\right|\left\|e_{\sigma(i)\|}\right\|=\left(\prod_i| a_{i\sigma(i)}|\right)\left(\prod_i\|e_{\sigma(i)}\right)\ge\prod_i\|e_i\|.
$$
Conversely, any basis $(e_i)$ satisfying $\det\|e_1\wedge\ldots\wedge e_N\|=\prod_i\|e_i\|$ is orthogonal for $\n$, as a direct consequence of Lemma~\ref{lem:Auer}. 
\end{proof}

\begin{lem}\label{lem:detdualdiag} If $\n$ is a diagonalizable norm on $V$, then $\det\left(\n^\vee\right)=\left(\det\n\right)^\vee$ under the canonical isomorphism $\det\left(V^\vee\right)\simeq\left(\det V\right)^\vee$. 
\end{lem}

\begin{proof} Let $(e_i)$ be an orthogonal basis for $\n$. By Lemma~\ref{lem:dualdiag}, the dual basis $(e_i^\vee)$ is orthogonal for $\n^\vee$, and $\|e_i^\vee\|^\vee=\|e_i\|^{-1}$. By Lemma~\ref{lem:detdiag}, we infer 
$$
\det\left\|e_1^\vee\wedge\ldots\wedge e_N^\vee\right\|^\vee=\prod_i\left\|e_i^\vee\right\|^\vee=\left(\prod_i\|e_i\|\right)^{-1}=\left(\det\|e_1\wedge\ldots\wedge e_N\|\right)^{-1},
$$
hence the result. 
\end{proof}

\begin{lem}\label{lem:detexact} Let $\n$ be a diagonalizable norm on $V$, and consider an exact sequence of vector spaces
$$
0\to V'\to V\to V''\to 0,
$$
with induced norms $\n',\n''$ on $V',V''$. Under the canonical isomorphism 
$$
\det V\simeq\det V'\otimes\det V'',
$$ 
we then have  
$$
\det\n=\det\n'\otimes\det\n''. 
$$ 
\end{lem}

\begin{proof} Set $N':=\dim V'$, $N'':=\dim V''$, and denote by $\pi:V\to V''$ the given surjection. Pick nonzero $\tau'\in\det V'$, $\tau''\in\det V''$ and $\e>0$. Definition~\ref{defi:det}, we may then find $v'_1,\ldots,v'_{N'}\in V'$ and $v''_1,\ldots,v''_{N''}\in V'$ such that $\tau'=v'_1\wedge\ldots\wedge v'_{N'}$, $\tau''=\pi(v''_1)\wedge\ldots\wedge\pi(v''_{N''})$, 
$$
\prod_i\|v'_i\|\le(1+\e)\det\|\tau'\|
$$
and
$$
\prod_i\|v''_i\|\le(1+\e)\det\|\tau''\|
$$
The isomorphism $\det V'\otimes\det V''\simeq\det V$ maps $\tau'\otimes\tau''$ to 
$$
v'_1\wedge\ldots\wedge v'_{N'}\wedge v''_1\wedge\ldots\wedge v''_{N''},
$$
which satisfies 
$$
\det\|\tau'\otimes\tau''\|=\det\|v'_1\wedge\ldots\wedge v'_{N'}\wedge v''_1\wedge\ldots\wedge v''_{N''}\|
$$
$$
\le\prod_i\|v'_i\|\prod_i\|v''_j\|\le(1+\e)^2(\det\|\tau'\|')(\det\|\tau''\|''),
$$
hence $\det\n\le\det\n'\otimes\det\n''$. By Lemma~\ref{lem:dualsub}, we dually have 
$$
\det(\n^\vee)\le\det(\n''^\vee)\otimes\det(\n'^\vee).
$$
Since $\n$ is diagonalizable, so are $\n'$ and $\n''$, by Lemma~\ref{lem:quotientdiag}. By Lemma~\ref{lem:detdualdiag}, we thus have $(\det\n)^{-1}\le(\det\n'')^{-1}\otimes(\det\n')^{-1}$, hence the result. 
\end{proof}

\begin{lem}\label{lem:detground} Let $\n$ be a diagonalizable norm on $V$, and $F/K$ be a complete field extension. Then $\det(\n_F)=(\det\n)_F$. 
\end{lem}
\begin{proof} This follows from Proposition~\ref{prop:ground} together with Lemma~\ref{lem:detdiag}.
\end{proof}

\begin{cor}\label{cor:detexact} If $K$ is non-Archimedean, Lemma~\ref{lem:detdualdiag}, Lemma~\ref{lem:detexact} and Lemma~\ref{lem:detground} hold for all norms.  
\end{cor}

\begin{proof} By Theorem~\ref{thm:diag}, diagonalizable norms are dense in the set of all norms, and we conclude by continuity of $\det$ (Lemma~\ref{lem:detlip}). 
\end{proof}

In the Archimedean case, both Lemma~\ref{lem:detdualdiag} and Lemma~\ref{lem:detexact} fail in general for non-diagonalizable norms. 

\begin{exam} Let $\n$ be the $\ell^\infty$-norm on $\R^N$. The dual norm $\n^\vee$ is the $\ell^1$-norm, and Example~\ref{exam:detarch} thus shows that $\det(\n^\vee)\ne(\det\n)^\vee$. Also, the exact sequence 
$$
0\to V'\to V\to V''\to 0
$$ 
with $V'=Ke_1$, $V''=Ke_2$ shows that $\det\n<(\det\n')\otimes(\det\n'')$. 
\end{exam}

Finally, recall from \S\ref{sec:diag} that each basis $\be=(e_i)$ of $V$ defines an apartment $\A_\be=\iota_\be(\R^N)$ in $\cN^\diag(V)$ and a Gram--Schmidt projection $\rho_\be:\cN(V)\to\A_\be$. For later use, we show: 

\begin{lem}\label{lem:detretr} For each diagonalizable norm $\n$, we have $\det\n=\det\rho_\be(\n)$. 
\end{lem}
\begin{proof} Denote by $W_i=\Vect(e_1,\ldots,e_i)$ the complete flag defined by $\be$. By Lemma~\ref{lem:detexact}, we have 
$$
\det\n=\bigotimes_i\det\n_{W_i/W_{i-1}}, 
$$ 
under the identification $\det V\simeq\bigotimes_i \det(W_i/W_{i-1})$. By definition of $\n_\be:=\rho_\be(\n)$, we infer 
$$
\det\|e_1\wedge\ldots\wedge e_N\|=\prod_i\|e_i\|_\be=\det\|e_1\wedge\ldots\wedge e_N\|_\be,
$$
where the second equality follows from Lemma~\ref{lem:detdiag}. 
\end{proof}
%
%
\subsection{Relative volume} 
We introduce the following 'additive version' of~\cite[2.4.3]{Tem}. 

\begin{defi}\label{defi:Euler} The \emph{relative volume} of two norms $\n,\n'$ on $V$ is defined as the real number
$$
\vol(\n,\n'):=\log\left(\frac{\det\n'}{\det\n}\right).
$$
\end{defi}

\begin{prop}\label{prop:vol} The relative volume satisfies the following properties.
\begin{itemize}
\item[(i)] cocycle formula:
$$
\vol(\n,\n')=\vol(\n,\n'')+\vol(\n'',\n').
$$
\item[(ii)] homogeneity:
$$
\vol(\n,e^c\n')=\vol(\n,\n')+cN\text{ for all }c\in\R.
$$ 
\item[(iii)] monotonicity:
$$
\n'\le\n''\Longrightarrow\vol(\n,\n')\le\vol(\n,\n'').
$$ 
\item[(iv)] Lipschitz continuity:
$$
\left|\vol(\n_1,\n'_1)-\vol(\n_2,\n'_2)\right|\le N\left(\dGI(\n_1,\n_2)+\dGI(\n'_1,\n'_2)\right). 
$$
\item[(v)] if $F/K$ is a complete field extension and $\n,\n'$ are two norms on $V$ with ground field extensions $\n_F,\n'_F$ to $V_F$, then 
$$
\vol(\n_F,\n'_F)=\vol(\n,\n')
$$
in the diagonalizable or non-Archimedean case, and 
$$
\left|\vol(\n_F,\n'_F)-\vol(\n,\n')\right|\le 2N\log N
$$
in the general Archimedean case. 
\item[(vi)] if $0\to V'\to V\to V''\to 0$ is an exact sequence, the induced norms $\n_{V'},\n'_{V'}$ and $\n_{V''}$, $\n'_{V''}$ satisfy
$$
\vol(\n,\n')=\vol\left(\n_{V'},\n'_{V'}\right)+\vol\left(\n_{V''},\n'_{V''}\right)
$$
in the diagonalizable or non-Archimedean case, and 
$$
\left|\vol(\n,\n')-\vol\left(\n_{V'},\n'_{V'}\right)-\vol\left(\n_{V''},\n'_{V''}\right)\right|\le 2N\log N 
$$
in general. 
\end{itemize}
\end{prop}
\begin{proof} The first three properties are obvious, and imply the fourth one as a formal consequence. In the diagonalizable and non-Archimedean case, (v) and (vi) are consequences of Lemma~\ref{lem:detground}, Lemma~\ref{lem:detexact} and Corollary~\ref{cor:detexact}. If $K$ is Archimedean, Theorem~\ref{thm:diag} yields diagonalizable norms $\n_0,\n'_0$ with 
$$
\dGI(\n,\n_0)\le\tfrac 12\log N,\,\,\,\dGI(\n',\n'_0)\le\tfrac 12\log N.
$$
Then 
$$
\vol((\n_0)_F,(\n'_0)_F)=\vol(\n_0,\n'_0),
$$
while (iv) yields 
$$
\left|\vol(\n,\n')-\vol(\n_0,\n'_0)\right|\le N\log N
$$
and 
$$
\left|\vol(\n_F,\n'_F)-\vol((\n_0)_F,(\n'_0)_F)\right|\le N\log N,
$$ 
since ground field extension is an isometry for $\dGI$, by Proposition~\ref{prop:ground}. Thus 
$$
\left|\vol(\n_F,\n'_F)-\vol(\n,\n')\right|\le\left|\vol(\n_F,\n'_F)-\vol((\n_0)_F,(\n'_0)_F)\right|
$$
$$
+\left|\vol(\n_0,\n'_0)-\vol(\n,\n')\right|\le 2N\log N.
$$
The proof of (vi) in the general Archimedean case is similar. 
\end{proof}

As we next show, in the Archimedean case, the relative volume is equivalent to the (logarithmic) volume ratio. The non-Archimedean case will be analyzed in the next section. 

\begin{prop}\label{prop:volarch} Assume that $K$ is Archimedean, and let $\n,\n'$ be two norms on $V$, with unit balls $B,B'$. 
\begin{itemize}
\item[(i)] If $\n,\n'$ are diagonalizable, then 
$$
\vol(\n,\n')=\frac{1}{[K:\R]}\log\left(\frac{\vol(B)}{\vol(B')}\right), 
$$
where $\vol$ is any choice of Haar measure on $V$. 
\item[(ii)] In the general case, we have 
\begin{equation}\label{equ:logvol}
\left|\vol(\n,\n')-\frac{1}{[K:\R]}\log\left(\frac{\vol(B)}{\vol(B')}\right)\right|\le 2N\log N. 
\end{equation}
\end{itemize}
\end{prop}

\begin{proof} If $\n$, $\n'$ are diagonalizable, we can pick an orthonormal basis $(e_i)$ for $\n$ in which $\n'$ is diagonalized, and the change-of-variable formula yields
$$
\vol(\n,\n')=\log\prod_i\|e_i\|'=\frac{1}{[K:\R]}\log\left(\frac{\vol(B)}{\vol(B')}\right), 
$$
which proves (i). The proof of (ii) is entirely similar to that of (v) in Proposition~\ref{prop:vol}. 
\end{proof}

Here again, the error term in~\eqref{equ:logvol} is generally nonzero in the Archimedean non-diagonalizable case. 

\begin{exam} By Example~\ref{exam:detarch}, the $\ell^1$ and $\ell^2$ norms $\n_1,\n_2$ on $\R^N$ satisfy $\vol(\n_1,\n_2)=0$, while the volume of the unit ball of $\n_1$ is strictly smaller than that of $\n_2$. 
\end{exam} 
%
%
\subsection{The content of a torsion module}
We assume in this section that $K$ is non-Archimedean and nontrivially valued, with associated valuation $v_K=-\log|\cdot|$. The next result was proved for instance in~\cite[Proposition 2.10]{Scho} (see also~\cite[Corollary 2.3.8]{Tem}). 

\begin{lem}\label{lem:finpres} Every finitely presented torsion $K^\circ$-module $M$ is isomorphic to a finite direct sum of cyclic modules
$$
M\simeq\bigoplus_{i=1}^r K^\circ/a_i K^\circ. 
$$
with $a_i\in K^{\circ\circ}$. Further, $r$ and the sequence $v_K(a_1),\ldots,v_K(a_r)$ are uniquely determined by $M$, up to permutation. 
\end{lem}
\begin{proof} Pick a presentation $(K^\circ)^a\to (K^\circ)^b\to M\to 0$. The image $\cV'$ of $(K^\circ)^a$ in $\cV:=(K^\circ)^b$ is finitely generated and torsion free, hence a free submodule. Since $\cV/\cV'$ is torsion, $\cV'\subset\cV$ are both lattices in $V:=K^b$. We claim that $V$ admits a basis $(e_i)$ such that $\cV=\bigoplus_i K^\circ e_i$ and $\cV'=\bigoplus_i K^\circ a_i e_i$ for some nonzero $a_i\in K^\circ$, which will yield as desired $M\simeq\bigoplus_{i=1}^r K^\circ/a_i K^\circ$ with $v_K(a_i)>0$ (since $K^\circ/a_i K^\circ=0$ when $v_K(a_i)=0$). Indeed, Proposition~\ref{prop:codiag} yields a basis $(e_i)$ that jointly diagonalizes the lattice norms $\n_\cV$ and $\n_{\cV'}$. Since lattice norms take values in $|K|$, we can arrange that $\|e_i\|_\cV=1$ after multiplying each $e_i$ by a scalar. Since $\cV'\subset\cV$, we then have $\|e_i\|_{\cV'}=|a_i^{-1}|$ for some nonzero $a_i\in K^\circ$, and we get the claim by Lemma~\ref{lem:latticenorm}. The uniqueness part is proved as in the usual structure theorem for torsion modules over PID's. 
\end{proof}

\begin{defi}\label{def:cont} The \emph{content} of a finitely presented torsion module $M$ is defined as 
$$
\cont(M):=\sum_{i=1}^r v_K(a_i)\in(0,+\infty).
$$
\end{defi}
Note that this is $-\log$ of the content as defined in~\cite[2.6.1]{Tem}. Alternatively, $\cont(M)$ is obtained by applying $v_K$ to the fractional ideal sheaf $\{ a\in K\mid a\cdot\det M=0\}$.  

\begin{exam}\label{exam:content} If $K$ is discretely valued with uniformizing parameter $\pi_K$, then $\cont(M)=v_K(\pi_K)\ell(M)$ with $\ell(M)$ the length of $M$. 
\end{exam}

The content is closely related to the relative volume. Indeed, as we saw during the proof of Lemma~\ref{lem:finpres}, every finitely presented torsion module is the quotient of two lattices in the same vector space, and we have: 

\begin{lem}\label{lem:contEuler} If $\cV'\subset\cV$ are lattices in a finite dimensional vector space $V$, then 
$$
\cont(\cV/\cV')=\vol\left(\n_\cV,\n_{\cV'}\right).
$$
\end{lem} 

Assuming now that $K$ is discretely valued, we conclude this section with an analogue of Proposition~\ref{prop:volarch}, relating the relative volume to the virtual length used in~\cite{BG+}. Recall that the virtual length $\ell(\cV/\cV')\in\Z$ of two lattices $\cV,\cV'$ in a $K$-vector space is defined as
$$
\ell(\cV/\cV'):=\ell(\cV/\cV'')-\ell(\cV'/\cV'')
$$
for any lattice $\cV''$ contained in both $\cV$ and $\cV'$ (cf.~\cite[Definition 4.1.1]{BG+},~\cite[III,\S1]{Ser}). 

\begin{prop}\label{prop:voldiscr} Assume that $K$ is discretely valued with uniformizing parameter $\pi_K$, and let $\n,\n'$ be two norms on $V$. Denote by $B,B'$ their unit balls, and by $\n_B,\n_{B'}$ the associated lattice norms. Then 
$$
\vol\left(\n_B,\n_{B'}\right)=v_K(\pi_K)\ell(B/B')=\vol(\n,\n')+O(N).  
$$
\end{prop} 
Note that the absolute value of $K$ is normalized by $ v_K(\pi_K)=1$ in~\cite{BG+}.
\begin{proof}[Proof of Proposition~\ref{prop:voldiscr}] The first equality follows from Example~\ref{exam:content}. By Lemma~\ref{lem:latticenorm}, we further have $\dGI(\n,\n_B)\le v_K(\pi_K)$ and $\dGI(\n',\n_{B'})\le v_K(\pi_K)$, and hence $\vol(\n,\n')=\vol(\n_B,\n_{B'})+O(N)$ by $N$-Lipschitz continuity of $\vol$. 
\end{proof}

%
%
\subsection{Relative spectra}\label{sec:successive}
We introduce a general notion of relative spectrum of two norms, and relate it to relative volumes. In the Archimedean (resp.~discretely valued) case, the results of this section can be traced back to Minkowski's work on successive minima (resp.~Mahler's paper~\cite{Mah}).

\begin{defi}\label{defi:spectrum} Let $\n,\n'\in\cN(V)$ be two norms on $V$. We define the \emph{relative spectrum} of $\n$ with respect to $\n'$ as the finite decreasing sequence
$$
\la_1(\n,\n')\ge\dots\ge\la_N(\n,\n')
$$
defined by the minmax-type formulas
\begin{equation}\label{equ:minmax}
\la_i(\n,\n'):=\sup_{W\subset V,\,\dim W\ge i}\left(\inf_{w\in W\setminus\{0\}}\log\frac{\|w\|'}{\|w\|}\right).
\end{equation}
\end{defi}

The following properties are immediate to check. 

\begin{lem}\label{lem:spectrum} For any two norms $\n,\n'$ on $V$ we have:
\begin{itemize} 
\item[(i)] $\la_1(\n,\n')=\sup_{v\in V\setminus\{0\}}\left(\log\|v\|'-\log\|v\|\right)$; 
\item[(ii)] $\la_N(\n,\n')=\inf_{v\in V\setminus\{0\}}\left(\log\|v\|'-\log\|v\|\right)=-\la_1(\n',\n)$; 
\item[(iii)] $\dGI(\n,\n')=\max\left\{\la_1(\n,\n'),\la_1(\n',\n)\right\}$; 
\item[(iv)] for each $i$, $\la_i(\n,\n')$ is a $1$-Lipschitz continuous function of $\n,\n'$ with respect to the Goldman--Iwahori distance $\dGI$.  
\end{itemize}
\end{lem}

The next result justifies the chosen terminology. 

\begin{prop}\label{prop:succ} Assume that $\n,\n'\in\cN^\diag(V)$ are diagonalizable. Choose a basis $(e_i)$ of $V$ in which both norms are diagonalized, as in Proposition~\ref{prop:codiag}, and order it so that 
$$
\frac{\|e_1\|'}{\|e_1\|}\ge\dots\ge\frac{\|e_N\|'}{\|e_N\|}. 
$$ 
Then 
$$
\la_i(\n,\n')=\log\frac{\|e_i\|'}{\|e_i\|}. 
$$
\end{prop}
\begin{proof} Set $W_i:=\Vect(e_1,\ldots,e_i)$, $W'_i:=\Vect(e_i,\ldots,e_N)$, and observe that
$$
\log\frac{\|e_i\|'}{\|e_i\|}=\inf_{w\in W_i\setminus\{0\}}\log\frac{\|w\|'}{\|w\|}=\sup_{w\in W'_i\setminus\{0\}}\log\frac{\|w\|'}{\|w\|}.
$$
By (\ref{equ:minmax}), the first equality yields $\log\frac{\|e_i\|'}{\|e_i\|}\le\la_i(\n,\n')$. On the other hand, each subspace $W\subset V$ with $\dim W\ge i$ satisfies $W\cap W'_i\ne\{0\}$ for dimension reason, and hence 
$$
\inf_{w\in W\setminus\{0\}}\log\frac{\|w\|'}{\|w\|}\le\sup_{w\in W'_i\setminus\{0\}}\log\frac{\|w\|'}{\|w\|}=\log\frac{\|e_i\|'}{\|e_i\|},
$$
and using (\ref{equ:minmax}) again yields $\la_i(\n,\n')\le\log\frac{\|e_i\|'}{\|e_i\|}$. 
\end{proof}

We are now in a position to prove the following analogue of Minkowski's second theorem. 

\begin{thm}\label{thm:mink} Let $\n,\n'$ be two norms on $V$. If $\n,\n'$ are diagonalizable, or $K$ is non-Archimedean, then 
$$
\vol(\n,\n')=\sum_{i=1}^N\la_i(\n,\n'). 
$$
Otherwise, 
$$
\left|\vol(\n,\n')-\sum_{i=1}^N\la_i(\n,\n')\right|\le 2N\log N.
$$
\end{thm}
\begin{proof} Note that $\vol(\n,\n')$ and $\sum_i\la_i(\n,\n')$ are both $N$-Lipschitz continuous in $\n,\n'$, by Proposition~\ref{prop:vol} and Lemma~\ref{lem:spectrum}, respectively. When both norms are diagonalizable, Lemma~\ref{lem:detdiag} and Proposition~\ref{prop:succ} yield
$$
\vol(\n,\n')=\sum_{i=1}^N\la_i(\n,\n').
$$
If $K$ is non-Archimedean, this propagates to all norms, by density of diagonalizable norms and (Lipschitz) continuity. Finally, the case of arbitrary norms in the Archimedean case is handled just as point (v) in Proposition~\ref{prop:vol}, by choosing diagonalizable norms at distance at most $\tfrac 12\log N$ of $\n$, $\n'$. 
\end{proof}

\begin{exam} By Example~\ref{exam:detarch}, the $\ell^1$ and $\ell^2$ norms $\n_1,\n_2$ on $\R^2$  satisfy $\vol(\n_2,\n_1)=0$, while $\la_1(\n_2,\n_1)=\log\sqrt 2$ and $\la_2(\n_2,\n_1)=0$. 
\end{exam}

%
%
\section{Alternative metric structures on spaces of norms}\label{sec:building}
The goal of this section, which stands somewhat apart from the rest of the paper, is to exploit the properties of determinants of norms to endow the space of diagonalizable norms with natural metric structures, recovering in particular the Bruhat--Tits metric in the non-Archimedean case. 

In this section, $K$ is again an arbitrary complete valued field, and $V$ an $N$-dimensional $K$-vector space. 
%
%
\subsection{The triangle inequality} 
The relative spectrum of two norms $\n,\n'\in\cN(V)$ defines a point $\la\left(\n,\n'\right)$ in the rational polyhedral cone
$$
\cC:=\left\{\la\in\R^N\mid\la_1\ge\dots\ge\la_N\right\}\simeq\R^N/\fS_N, 
$$
which is a Weyl chamber for the Weyl group $\fS_N$ of $\GL_N$. Given an $\mathfrak{S}_N$-invariant norm $\chi$ on $\R^N$, set
$$
\dist_\chi\left(\n,\n'\right):=\chi\left(\la\left(\n,\n'\right)\right). $$
The resulting function $\dist_\chi:\cN(V)\times\cN(V)\to\R_+$ is symmetric, with $\dist_\chi(\n,\n')=0$ if and only if $\n=\n'$. The following result is inspired by Gerardin's proof of~\cite[2.4.7, Corollaire 2]{Ger}. 

\begin{thm}\label{thm:normdist} For each $\fS_N$-invariant norm $\chi$ on $\R^N$, $\dist_\chi$ satisfies the triangle inequality on $\cN^\diag(V)$, and is characterized as the unique metric on $\cN^\diag(V)$ for which 
$$
\iota_\be:(\R^N,\chi)\hookrightarrow(\cN^\diag(V),\dist_\chi)
$$ 
is an isometric embedding for each basis $\be$ of $V$.
\end{thm} 

By Proposition~\ref{prop:GIisom}, the Goldman--Iwahori metric $\dGI$ on $\cN^\diag(V)$ corresponds to the $\ell^\infty$-norm on $\R^N$. By equivalence of norms on $\R^N$, any metric $\dist_\chi$ produced by Theorem~\ref{thm:normdist} is Lipschitz equivalent to $\dGI$. Besides the latter, the most important case is the Euclidian metric $\dist_2$ induced by the $\ell^2$-norm:

\begin{exam} When $K$ is Archimedean, $\dist_2$ coincides with the Riemannian metric of the symmetric space $\cN^\diag(V)\simeq\GL_N(K)/\mathrm{U}_N(K)$ (see Theorem~\ref{thm:Finsler} below). When $K$ is non-Archimedean, $(\cN^\diag(V),\dist_2)$ is a realization of the Bruhat--Tits building of $\GL_N(K)$ with its Euclidian metric, see for instance~\cite[Chapter III]{Par}. In both cases, $(\cN^\diag(V),\dist_2)$ is a $\mathrm{CAT}(0)$ metric space. 
\end{exam}

\begin{cor} If $K$ is non-Archimedean, the space $\cN(V)$ is complete with respect to the metric $\dist_\chi$, which is characterized as the unique compatible metric such that $\iota_\be:(\R^N,\chi)\hookrightarrow(\cN(V),\dist_\chi)$ is an isometric embedding for all bases $\be$. For $\chi=\ell^2$, $(\cN(V),\dist_2)$ is a $\mathrm{CAT}(0)$ metric space. 
\end{cor}
\begin{proof} On $\cN^\diag(V)$, $\dist_\chi$ is equivalent to $\dist_\infty$ and satisfies the triangle inquality. This is also the case on $\cN(V)$, by density of $\cN^\diag(V)$. Since $\cN(V)$ is complete for $d_\infty$, it is also complete for $d_\chi$, and is thus the completion of $(\cN^\diag(V),d_\chi)$. Conversely, any metric on $\cN(V)$ with the stated property must coincide with $\dist_\chi$ on the dense subset $\cN^\diag(V)$, hence everywhere. For $\chi=\ell^2$, the Bruhat--Tits building $(\cN^\diag(V),\dist_2)$ is a $\mathrm{CAT}(0)$ metric space, and this property is preserved under completion. 
\end{proof}

As in the construction of the Euclidian metric on any Euclidian building, the key to the proof of Theorem~\ref{thm:normdist} is to show that the Gram--Schmidt projections introduced in Definition~\ref{defi:retraction} are distance-decreasing. 

\begin{lem}\label{lem:normdist} For each basis $\be$ of $V$, the Gram--Schmidt projection $\rho_\be:\cN^\diag(V)\to\A_\be$ satisfies 
\begin{equation}\label{equ:normretr}
\dist_\chi\left(\rho_\be(\n),\rho_\be(\n')\right)\le \dist_\chi(\n,\n')
\end{equation}
for all $\n,\n'\in\cN^\diag(V)$. 
\end{lem} 

We first recall some elementary facts. 

\begin{defi} Given $\la,\la'\in\cC=\left\{\la\in\R^N\mid\la_1\ge\dots\ge\la_N\right\}$, one says that $\la$ is \emph{majorized by $\la'$}, written $\la\preceq\la'$, if $\la_1+\dots+\la_i\le\la'_1+\dots+\la'_i$ for all $i$, with equality for $i=N$. \end{defi}

\begin{lem}\label{lem:convex} We have $\la\preceq\la'$ if and only $\la$ belongs to the convex envelope of the $\fS_N$-orbit of $\la'$, and then $\chi(\la')\le \chi(\la)$ for any $\fS_N$-invariant norm $\chi$ on $\R^N$.  
\end{lem} 
\begin{proof} It is straightforward to see that any $\la$ in the convex envelope of the $\fS_N$-orbit of $\la'$ satisfies $\la\preceq\la'$. As observed in~\cite{Rad}, the converse is a simple consequence of the Hahn--Banach theorem. Assuming indeed that $\la\preceq\la'$, it is enough to show that for each $\mu\in\R^N$ there exists $\rho\in\fS_N$ with 
$$
\sum_i\mu_i\la_i\le\sum_i\mu_i\la'_{\sigma(i)}.
$$
Choose $\sigma$ such that $\mu_{\sigma(1)}\ge\dots\ge\mu_{\sigma(N)}$. Then 
$$
\sum_i\mu_i\la_i=\sum_i\mu_{\sigma(i)}\la_{\sigma(i)}=\sum_{i<N}\left(\mu_{\sigma(1)}-\mu_{\sigma(i+1)}\right)\left(\la_{\sigma(1)}+\dots+\la_{\sigma(i)}\right)+\mu_{\sigma(N)}\left(\la_{\sigma(1)}+\dots+\la_{\sigma(N)}\right)
$$
$$
\le\sum_{i<N}\left(\mu_{\sigma(1)}-\mu_{\sigma(i+1)}\right)\left(\la'_1+\dots+\la'_i\right)+\mu_{\sigma(N)}\left(\la'_1+\dots+\la'_N\right)=\sum_i\mu_{\sigma(i)}\la'_i=\sum_i\mu_i\la'_{\sigma^{-1}(i)},
$$
and it remains to set $\rho:=\sigma^{-1}$. Assume finally that $\la\preceq\la'$, and hence $\la=\sum_j t_j\mu_j$ with $\mu_i$ $\fS_N$-equivalent to $\la'$, $t_j\in\R_{\ge 0}$ and $\sum_j t_j=1$, by what we just saw. By $\fS_N$-invariance of $\chi$, we infer
$$
\chi(\la)\le\sum_j t_j\chi(\mu_i)=\sum_j t_j\chi(\la')=\chi(\la').
$$
\end{proof}

\begin{proof}[Proof of Lemma~\ref{lem:normdist}] By Lemma~\ref{lem:convex},  it will be enough to show that the relative spectrum $\la$ (resp. $\mu$) of $\rho_\be(\n')$ with respect to $\rho_\be(\n')$ (resp. $\n'$ with respect to $\n$) satisfy 
$\la_1+\dots+\la_i\le\mu_1+\dots+\mu_i$ for all $i$, with equality for $i=N$. Set $W:=\Vect(e_1,\ldots,e_i)$, and observe that the restriction $\rho_\be(\n)_W$ of $\rho_\be(\n)$ to $W$ satisfies by definition
$$
\rho_\be(\n)_W=\rho_{\be_W}(\n_W)
$$ 
with $\be_{W}=(e_1,\ldots,e_i)$. By Lemma~\ref{lem:detretr}, we thus have $\det(\rho_\be(\n)_W)=\det(\n_W)$, and Theorem~\ref{thm:mink} yields
$$
\la_1+\dots+\la_i=\vol\left(\rho_\be(\n)_W,\rho_\be(\n')_W\right)
$$
$$
=\vol\left(\n_{W},\n'_{W}\right)=\sum_{j\le i}\la_j(\n_{W},\n'_{W})
$$
$$
\le\sum_{j\le i}\la_j(\n,\n')=\mu_1+\dots+\mu_i,
$$
where the last inequality follows directly from~\eqref{equ:minmax}, and is an equality when $i=N$, \ie $W=V$. 
\end{proof}

\begin{proof}[Proof of Theorem~\ref{thm:normdist}] By construction, $\iota_\be$ is an isometric embedding with respect to $\chi$ and $\dist_\chi$, and the latter therefore satisfies the triangle inequality on each apartment $\A_\be=\iota_\be(\R^N)$. Pick three diagonalizable norms $\n_i\in\cN^\diag(V)$, $i=1,2,3$. We may then choose a basis $\be$ with $\n_1,\n_2\in\A_\be$. Since $\dist_\chi$ satisfies the triangle inequality on $\A_\be$, we have 
$$
\dist_\chi\left(\n_1,\n_2\right)=\dist_\chi\left(\rho_\be(\n_1),\rho_\be(\n_2)\right)
$$
$$
\le \dist_\chi\left(\rho_\be(\n_1),\rho_\be(\n_3)\right)+\dist_\chi\left(\rho_\be(\n_3),\rho_\be(\n_2)\right),
$$
which shows that $\dist_\chi$ satisfies the triangle inequality on $\cN^\diag(V)$, by (\ref{equ:normretr}). 
\end{proof}

%
%
\subsection{The Archimedean case: Finsler metrics}
Assume that $K$ is Archimedean, \ie $K=\R$ or $\C$, and denote by $H(V)$ the real vector space of all quadratic/Hermitian forms $h$ on $V$. Each diagonalizable norm $\n\in\cN^\diag(V)$ is then associated to a positive definite form $\g(v):=\|v\|^2$, thereby defining an embedding of $\cN^\diag(V)$ as an open convex subset of $H(V)$. Further, $\cN^\diag(V)$ is diffeomorphic to the symmetric space $\GL_N(K)/\mathrm{U}_N(K)$, with $\mathrm{U}_N(K)$ the unitary/orthogonal group.

As we now show, the metric $\dist_\chi$ constructed above is induced by a natural Finsler metric on $\cN^\diag(V)$. Recall first that for each $h,\g\in H(V)$ with $\g$ positive definite, one can find a basis $\be=(e_i)$ of $V$ which is orthonormal for $\g$ and orthogonal for $h$, \ie
$$
\g\left(\sum_i  a_i e_i\right)=\sum_i| a_i|^2
$$
and 
$$
h\left(\sum_i  a_i e_i\right)=\sum_i\la_i| a_i|^2
$$
for all $ a_i\in K$. The spectrum $(\la_i)$ is independent of the choice of $\be$ up to ordering, hence defines a point $\la_\g(h)\in\cC$, and we then have for any two $h,h'\in H(V)$
$$
\la_\g(h+h')\preceq\la_\g(h)+\la_\g(h').
$$
This is indeed a simple consequence of the min-max principle, known as the \emph{Ky Fan inequality}. Given a symmetric norm $\chi$ on $\R^N$, it follows as in Lemma~\ref{lem:convex} that setting for each $\g\in\cN^\diag(V)$
$$
|h|_{\chi,\g}:=\chi\left(\la_g(h)\right)
$$
defines a norm on $H(V)$, and we thus get a continuous Finsler norm $|\cdot|_\chi$ on the tangent bundle of $\cN^\diag(V)$.

\begin{thm}\label{thm:Finsler} The metric $\dist_\chi$ on $\cN^\diag(V)$ in Theorem~\ref{thm:normdist} coincides with the length metric defined by the Finsler norm $|\cdot|_\chi$. In other words, for any two $\g,\g'\in\cN^\diag(V)$, $\dist_\chi(\g,\g')$ is the infimum over all smooth paths $(\g_t)_{t\in[0,1]}$ in $\cN^\diag(V)$ joining $\g$ to $\g'$ of the corresponding length
$$
\ell_\chi(\g):=\int_0^1|\dot\g_t|_{\chi,\g_t}dt.
$$
\end{thm}
As a consequence, the distance $d_\chi$ coincides with the one constructed in~\cite[\S 4]{Bha03} and (for $\chi=\ell^p$) in~\cite{DLR}. 

\begin{lem}\label{lem:Finsler} For each basis $\be$ of $V$, the Gram--Schmidt projection $\rho_\be:\cN^\diag(V)\to \A_\be$ is a smooth map, and it satisfies $\rho_\be^\star |\cdot|_\chi\le|\cdot|_\chi$.
\end{lem}
\begin{proof} The smoothness of $\rho_\be$ follows from its description in terms of the Gram--Schmidt orthogonalization process. Pick $\g\in\cN^\diag(V)$, $h\in H(V)$. As in the proof of Lemma~\ref{lem:normdist}, it will be enough to show that $\la:=\la_{\rho_\be(\g)}(d_\g\rho_\be(h))$ is majorized by $\mu:=\la_\g(h)$. Differentiating the identity
$$
\vol(\g,\g')=\vol(\rho_\be(\g),\rho_\be(\g'))
$$
with respect to $\g'$ shows that the trace $\Tr_\g(h)$ satisfies
$$
\Tr_\g(h)=\Tr_{\rho_\be(\g)}(d_\g\rho_\be(h)),
$$
\ie $\la_1+\dots+\la_N=\mu_1+\dots+\mu_N$. Arguing as in Lemma~\ref{lem:normdist}, we apply this fact to the restrictions of $\g$ and $h$ to the span $W$ of $(e_1,\ldots,e_i)$ for a given $i$, and get
$$
\la_1+\dots+\la_i=\Tr_{\g|_W}(h|_W)\le\mu_1\dots+\mu_i
$$
thanks to the min-max principle. 
\end{proof}

\begin{proof}[Proof of Theorem~\ref{thm:Finsler}] Pick a basis $e$, and observe that the differential $d_\la\iota_\be:\R^N\to H(V)$ at $\la\in\R^N$ satisfies for all $\mu\in\R^N$ 
$$
\la_{\iota_\be(\la)}(d_\la\iota_\be(\mu))=\mu\mod\fS_N. 
$$
As a result, $\iota_\be^\star |\cdot|_\chi$ is the constant Finsler norm $\chi$ on $\R^N$, and the $\chi$-length of any smooth path $\g:[0,1]\to \A_\be$ joining $\g=\iota_\be(\la)$ to $\g'=\iota_\be(\la')$ thus satisfies 
$$
\ell_\chi(\g)\ge\chi(\la'-\la)=\dist_\chi(\g,\g'),
$$
with equality when $\g$ is the image of the line segment $[\la,\la']$. If $\g:[0,1]\to\cN^\diag(V)$ is now a smooth path joining $\g$ to $\g'$ in $\cN^\diag(V)$ only, the previous case applies to $\rho_\be\circ\g:[0,1]\to \A_\be$, which combines with Lemma~\ref{lem:Finsler}~(ii) to give $\ell_\chi(\g)\ge\ell_\chi(\rho_\be\circ\g)\ge \dist_\chi(\g,\g')$. 
\end{proof}

%
%

\part{Models and metrics}

%
%
\section{Analytification and models}\label{sec:models}
%
%
This section reviews some well-known facts on Berkovich analytifications and models, with an emphasis on the reduced fiber condition. We provide in particular a direct proof of a version of the Bosch--L\"utkebohmert--Raynaud reduced fiber theorem for models. 

As before, $K$ denotes a complete valued field. All schemes over $K$ (or $K^\circ$, when $K$ is non-Archimedean) considered below are \emph{separated}, unless otherwise specified. 
%
%
\subsection{Analytification}
To any scheme $X$ of finite type over $K$, Berkovich functorially associates in~\cite[\S 3.4]{Ber} an analytification $X^\an$, together with a continuous map $\ker:X^\an\to X$. In the present paper, we mostly only need the underlying topological space, which is Hausdorff and locally compact. As such, the analytifications of $X$ and of the reduced scheme $X_{\red}$ coincide, but it will nevertheless be useful for inductive arguments to allow $X$ to be non-reduced. 

Assume first that $X$ is affine, \ie $X=\Spec A$ with $A$ a finite type $K$-algebra. The topological space $X^\an$ is defined as the set of all multiplicative seminorms on $A$ extending the given absolute value on $K$, endowed with the topology of pointwise convergence. The multiplicative seminorm associated to $x\in X^\an$ is denoted by $f\mapsto |f(x)|$. The set of $f\in A$ with $|f(x)|=0$ is a prime ideal, the \emph{kernel} of $x$, thereby defining a natural continuous map $\ker:X^\an\to X$. We thus have $|f(x)|=0$ if and only if $f$ vanishes at $\xi=\ker(x)$, and $f\mapsto|f(x)|$ defines a norm on the residue field $\kappa(\xi)$. 

\begin{lem}\label{lem:nilpo} A function $f\in A$ satisfies $|f|\equiv 0$ on $X^\an$ if and only if $f$ is nilpotent.
\end{lem}
\begin{proof} For each closed point $\xi\in X$, the absolute value on $K$ (uniquely) extends to the finite field extension $\kappa(\xi)$ of $K$, and image of the kernel map therefore contains the set of closed points of $X$ (in fact, $\ker$ injects $X^\an$ onto the set of closed points of $X$ if $K$ is Archimedean, while $\ker$ maps $X^\an$ onto $X$ when $K$ is non-Archimedean). As a consequence, a function $f\in A$ with $|f|\equiv 0$ vanishes at all closed points of $X$, and hence is nilpotent. 
\end{proof}

Consider now an arbitrary $K$-scheme of finite type $X$, and cover it with finitely many affine open subschemes $U_i$. Since $X$ is separated, each $U_{ij}=U_i\cap U_j$ is affine, and $U_{ij}^\an$ is homeomorphic to the inverse image of $U_{ij}$ in both $U_i^\an$ and $U_j^\an$. We can thus glue $U_i^\an$ and $U_j^\an$ together along their common open subset $U_{ij}^\an$ to define the topological space $X^\an$ with a continuous map $\ker:X^\an\to X$. 

The GAGA theorem~\cite[3.4.8, 3.5.3]{Ber} guarantees that $X^\an$ is Hausdorff (since we always assume $X$ separated), locally compact, and $X^\an$ is compact if and only if $X$ is proper. 

\begin{exam} If $K$ is Archimedean, the Gelfand-Mazur theorem yields the following description of $X^\an$. When $K=\C$, $X^\an$ is the usual analytification of $X$, \ie $X^\an=X(\C)$ with its Euclidian topology. When $K=\R$, $X^\an$ is identified with the set of closed points of $X$, \ie the quotient of $(X\otimes\C)^\an=X(\C)$ by complex conjugation. 
\end{exam}

\begin{exam}\label{exam:semival} When $K$ is non-Archimedean with valuation $v_K=-\log|\cdot|$, $X^\an$ can be seen as a space of semivaluations on $X$, \ie real valuations on the residue fields of points of $X$. More precisely, the bijective map $x\mapsto(\ker(x),v_x)$ with $v_x(f):=-\log|f(x)|$ describes $X^\an$ as the set of pairs $(\xi,v)$ where $\xi\in X$ is a scheme point and $v:\kappa(\xi)^\star \to\R$ is a rank $1$ valuation on the residue field $\kappa(\xi)$ extending $v_K$. 
\end{exam}
%
%
\subsection{Smooth functions}\label{sec:smooth}
In the Archimedean case, every point of $X^\an$ admits a neighborhood $V$ with a closed (analytic) embedding in a polydisc $\DD^r$, and a smooth function $u:V\to\R$ is defined as the restriction of a smooth function on $\DD^r$, the definition being independent of the choice of closed embedding. 

In the non-Archimedean case, Chambert-Loir and Ducros have introduced in~\cite{CLD} a notion of $(p,q)$-form on any Berkovich analytic space (see also~\cite{Gub13} for the case of an analytification). In particular, a sheaf of \emph{smooth functions} on $X^\an$ is defined, basically prescribed by the following two natural requirements: 
\begin{itemize}
\item[(i)] $\log|f|$ is smooth for each invertible analytic function $f$; 
\item[(ii)] if $u_1,\ldots,u_r$ are smooth functions on an open $V\subset X^\an$ and $\chi$ is a smooth function defined near the range of the map $V\to\R^r$ with components $(u_i)$, then $\chi(u_1,\ldots,u_r)$ is smooth on $V$. 
\end{itemize}
More explicitly, a function $u$ on an open subset of $X^\an$ is smooth iff it is locally of the form 
\begin{equation}\label{equ:smooth}
u=\chi\left(\log|f_1|,\ldots,\log|f_r|\right)
\end{equation}
where the $f_i$ are invertible analytic functions and $\chi$ is a smooth function on an appropriate subset of $\R^r$. Note that this description of smooth functions also holds in the Archimedean case (using $\Rea z_i=\log|e^{z_i}|$). 

%
\subsection{Models and reduction}\label{sec:model}
In what follows, $K$ is non-Archimedean (possibly trivially valued). Recall that the valuation ring $K^\circ$ is Noetherian if and only if $K$ is discretely or trivially valued. Let $X$ be a $K$-scheme of finite type. 

\begin{defi} A \emph{model} of $X$ is a (separated) flat, finite type $K^\circ$-scheme $\cX$ together with an identification of $K$-schemes $\cX_K:=\cX\otimes_{K^\circ} K\simeq X$. 
\end{defi}
The \emph{special fiber} of a model $\cX$ is the $\tK$-scheme of finite type $\cX_s:=\cX\otimes_{K^\circ}\tK$. We say that a model $\cX$ is \emph{proper} (resp.~\emph{projective}) if it is proper (resp.~projective) as a $K^\circ$-scheme. This implies of course that $X$ is proper (resp.~projective) as a $K$-scheme. 

Models of $X$ form a category, in which a morphism of models $\mu:\cX'\to\cX$ is a morphism of $K^\circ$-schemes compatible with the given identifications $\cX'_K\simeq X\simeq\cX_K$. If two models $\cX'$, $\cX$ admit a morphism $\mu:\cX'\to\cX$, then $\mu$ is unique, by separatedness, and we then say that $\cX'$ \emph{dominates} $\cX$. We say that $\cX'$ \emph{properly dominates} $\cX$ if $\mu$ is proper. 

In the trivially valued case, models of $X$ are in 1--1 correspondence with automorphisms of $X$, and $\cX=X$ is thus the only model up to isomorphism.
 
\begin{lem}\label{lem:modelfp} Any model $\cX$ of $X$ is automatically finitely presented over $K^\circ$. 
\end{lem}
This follows from a general result of Raynaud--Gruson~\cite[Th\'eor\`eme 3.4.6]{RG}, and goes back to Nagata~\cite[Theorem 3]{Nag}. For the sake of completeness, we reproduce here a simple argument due to Antoine Ducros~\cite{Ducros}, which is basically equivalent to that of Nagata. 

\begin{proof} We claim that it is enough to prove the result when $\cX$ is projective over $K^\circ$. Indeed, arguing locally, we may first assume that $\cX$ is affine, and we get the claim by choosing a closed embedding in an affine space and passing to the schematic closure in the corresponding projective space. Pick a closed embedding $\cX\hookrightarrow\P^N_{K^\circ}$, and denote by $I$ the corresponding homogeneous ideal of $R:=K^\circ[t_0,\ldots,t_N]$. Since both $(R/I)\otimes K$ and $(R/I)\otimes\tK$ are Noetherian, we may choose a finitely generated homogeneous ideal $I'\subset I$ such that $R/I'\to R/I$ becomes an isomorphism after tensoring with either $K$ or $\tK$. This means that the finitely presented closed subscheme $\cX'\subset\P^N_{K^\circ}$ defined by $I'$ has the same special fiber and generic fiber as $\cX$. If we can show that $\cX'$ is flat over $K^\circ$, it will coincide with the schematic closure of its generic fiber, which will prove that $\cX=\cX'$ is finitely presented. 
But $\cX'$ is flat over $K^\circ$ if and only if the finite type $K^\circ$-module $\cV_m:=(R/I')_m$ is free for all $m\in\N$ large enough, which is indeed the case since $\dim_K \cV_m\otimes K=\dim_{\tK} \cV_m\otimes\tK$, by choice of $I'$. 
\end{proof}

A model $\cX$ of $X$ determines a compact subset $\cX^\beth$ \footnote{The letter $\beth$ ('bet') is the second letter of the Hebrew alphabet. The chosen notation follows the lead of~\cite{thu}.} of $X^\an$ and an anticontinuous\footnote{The inverse image of an open is closed.} \emph{reduction map}
$$
\red_\cX:\cX^\beth\to\cX_s,
$$
as follows. If $\cX$ is affine, \ie $\cX=\spec(\cA)$ with $\cA$ a flat (\ie torsion-free) finite type $K^\circ$-algebra, then 
$$
\cX^\beth=\left\{x\in X^\an\mid|f(x)|\le 1\text{ for all }f\in\cA\right\}, 
$$
and the reduction $\red_\cX(x)$ of $x\in\cX^\beth$ is the point of $\cX_s$ induced by the prime ideal $\{f\in\cA\mid|f(x)|<1\}$. In the general case, $\cX$ is covered by finitely many affine open subschemes $\cU_i$, whose generic fibers $U_i$ give an affine open cover of $X$, and 
$$
\cX^\beth=\bigcup_i\cU_i^\beth\subset\bigcup_i U_i^\an=X^\an.
$$
In the language of Example~\ref{exam:semival}, $\cX^\beth$ consists of those semivaluations on $X$ that admit a center on $\cX$, and $\red_\cX$ maps a semivaluation to its center, which is necessarily on $\cX_s$. By the valuative criterion of properness, we thus have:

\begin{lem}\label{lem:doman} If a model $\cX'$ properly dominates $\cX$, then $(\cX')^\beth=\cX^\beth$. If $\cX$ is proper over $K^\circ$ (and hence $X$ is proper over $K$), then $\cX^\beth=X^\an$.
\end{lem}

\begin{exam} In the trivially valued case, $X$ is the only model up to isomorphism, and $X^\beth$ coincides with the construction of~\cite{thu}. In particular, $X^\beth\subset X^\an$, with equality iff $X$ is proper. 
\end{exam}
The compact set $\cX^\beth$ associated to a model $\cX$ can also be understood as (the underlying topological space of) the generic fiber in the sense of~\cite[\S 1]{Ber3} of the formal completion $\hcX$ of $\cX$. As above, it is enough to consider the case where $X=\spec(A)$ and $\cX=\spec(\cA)$ are affine. For any two nonzero $a,a'\in K^{\circ\circ}$, there exists $n\gg 1$ with $a^n\in K^\circ a'$, and the formal completion $\hcA$ of $\cA$ with respect to $a$ is thus independent of the choice of $a\in K^{\circ\circ}$ (set $\hcA=\cA=A$ in the trivially valued case). The $K^\circ$-algebra $\hcA$ is flat and topologically of finite type, and $\hcX=\Spf(\hcA)$ is thus an admissible formal $K^\circ$-scheme, whose generic fiber $\hcX_\eta$ is defined as the set of bounded multiplicative seminorms on the $K$-affinoid algebra $\hA:=\hcA\otimes K$. Composing such a seminorm with the canonical map $A\to\hA$ defines a continuous map $\hcX_\eta\to\cX^\beth$, which is easily see to be bijective by density of the image of $A$ in $\hA$, and hence a homeomorphism, by compactness of $\hcX_\eta$. 
 
The following well-known result holds in fact for the reduction map of any admissible formal scheme (see for instance~\cite[\S 2.13]{GRW}). 
\begin{lem}\label{lem:redmodel} The reduction map $\red_\cX:\cX^\beth\to\cX_s$ of any model $\cX$ is anticontinuous and surjective. Further, the preimage $\Ga(\cX)$ of the set of generic points of $\cX_s$ is a finite set. 
\end{lem}
We shall call $\Ga(\cX)$ the set of \emph{Shilov points}. When $\cX$ (and hence $X$) is affine, $\Ga(\cX)$ is exactly the Shilov boundary of the $K$-affinoid domain $\cX^\beth$ in the sense of~\cite[2.4.4]{Ber}, cf.~\cite[Proposition A.3]{GM}. Covering a general model with affine open subschemes, we get: 

\begin{lem}\label{lem:Shilov} For any model $\cX$ of $X$ and $f\in\cO(X)$, the sup-seminorm on $\cX^\beth$ satisfies 
$$
\|f\|_{\cX^\beth}:=\sup_{\cX^\beth}|f|=\max_{\Ga(\cX)}|f|.
$$
\end{lem}

\begin{exam}\label{exam:shilovtriv} If $K$ is trivially valued and $\eta$ is a generic point of $X=X_s$, the point of $X^\beth$ corresponding to the trivial valuation on $\kappa(\eta)=\cO_{X,\eta}$ is the unique Shilov point mapping to $\eta$. 
\end{exam}

\begin{exam}\label{exam:shilovnormal} Assume that $K$ is nontrivially valued and $\cX$ is normal, \ie integrally closed in $X$ with $X$ normal. For each generic point $\eta$ of $\cX_s$, the local ring $\cO_{\cX,\eta}$ is a rank one valuation ring. This is of course well-known when $\cX$ is Noetherian (\ie $K$ trivially or discretely valued), while the general case is proved in~\cite[Theorem 2.6.1]{Kna}. As observed in~\cite[Proposition 2.3]{GS}, one easily checks that the corresponding point of $\cX^\beth$ is the unique Shilov point of $\cX$ mapping to $\eta$. 
\end{exam}

\begin{exam}\label{exam:shilovintclos} If we merely assume that $\cX$ is integrally closed in $X$, it is still true that there is a unique Shilov point mapping to any given generic point $\eta$ of $\cX_s$. In the discretely valued case, this is proved in~\cite[Lemme 2.1]{CLT}. In the densely valued case, the assumption implies that $\cX_s$ is reduced by Theorem~\ref{thm:redint} below, which also implies that $\red_\cX:\cX^\beth\to\cX_s$ coincides with the affinoid reduction map, and we conclude by~\cite[Proposition 2.4.4]{Ber}. 
\end{exam}

%
%
\subsection{Sup-seminorm, integral closure and reduced fiber}
The goal of this section is to review the relation between sup-seminorm, integral closure and reduced fiber. 

Assume first that $K$ is non-Archimedean and nontrivially valued, $X=\spec(A)$ is affine and $\cX=\spec(\cA)$ is an affine model of $X$. Denote as above by $\hcA$ the formal completion of $\cA$ with respect to any nonzero $a\in K^{\circ\circ}$, by $\hA:=\hcA\otimes K$ the associated $K$-affinoid algebra, and write $f\mapsto\hf$ for the canonical map $A\to\hA$ (which is not injective in general, cf.~Example~\ref{exam:notinj} below). The affinoid algebra $\hA$ is equipped with the \emph{sup-seminorm} $\n_{\sup}$, defined by setting for $g\in\hA$ 
$$
\|g\|_{\sup}:=\sup_{\cX^\beth}|g|=\max_{\Ga(\cX)}|g|,
$$
where the second equality holds by~\cite[2.4.4]{Ber}. The sup-seminorm on $\cX^\beth$ of $f\in A$ as in Lemma~\ref{lem:Shilov} can thus be written as
$$
\|f\|_{\cX^\beth}=\|\hf\|_{\sup}. 
$$
By~\cite[6.2.1/4]{BGR} and~\cite[3.4.3]{Ber}, we have: 

\begin{lem}\label{lem:seminorm} The sup-seminorm on $\hA$ is a norm if and only if $\hA$ is reduced. This holds in particular if $A$ is reduced. 
\end{lem}
While $A\to\hA$ has dense image, it is not injective in general, even when $A$ is reduced:

\begin{exam}\label{exam:notinj} If $K$ is discretely valued, $\cA:=K$ is of finite type of $K^\circ$, and hence a model of $A=K$, for which $\hA=\{0\}$ (thanks to Antoine Ducros for this simple example). 
\end{exam}

\begin{thm}\label{thm:intunit} The unit ball of $\n_{\sup}$ coincides with the integral closure $\hcA'$ of $\hcA$ in $\hA$. Similarly, the unit ball of $\n_{\cX^\beth}$ coincides with the integral closure $\cA'$ of $\cA$ in $A$, and the induced map $\cA'\to\hcA'$ further has dense image.  
\end{thm}

\begin{cor}\label{cor:intunit} A given $f\in A$ is integral over $\cA$ if and only if $\hf$ is integral over $\hcA$, and $\cA$ is integrally closed in $A$ if and only if $\hcA$ is integrally closed in $\hA$. 
\end{cor}
We start with a useful observation. 

\begin{lem}\label{lem:sigma} For each $f\in A$, we have $f\in\cA\Longleftrightarrow\hf\in\hcA$. 
\end{lem}
\begin{proof} We imitate~\cite[Lemma 1.4]{BL2}. Pick a nonzero $a\in K^{\circ\circ}$ such that $a f\in\cA$, and note that the canonical map $\cA\to\hcA$ induces an isomorphism $\cA/a\cA\simeq\hcA/a\hcA$. If $\hf\in\hcA$, we thus have $a f=a g$ for some $g\in\cA$, and hence $f=g\in\cA$ after multiplying by $a^{-1}$ in $A$.  
\end{proof}

\begin{proof}[Proof of Theorem~\ref{thm:intunit}] That the unit ball of $\n_{\sup}$ is the integral closure of $\hcA$ is a reformulation of~\cite[6.3.4/1]{BGR} (and the remark that follows). Before dealing with the unit ball of $\n_{\cX^\beth}$, we first establish the density of the image of $\cA'$ in $\hcA'$. 

Let thus $g\in\hcA'$, so that $g^n+\sum_{i=1}^{n-1}b_i g^{n-i}\in\hcA$ for some $b_i\in\hcA$. Since $\cA\to\hcA$ and $A\to\hA$ have dense images, we can pick sequences $f_j\in A$ and $a_{ij}\in\cA$ with $\hf_j\to g$ and $\ha_{ij}\to b_i$ as $j\to\infty$. As $\hcA$ is open in $\hA$ (for instance by Lemma~\ref{lem:infnorm} below), it follows that $\hf_j^n+\sum_{i=1}^{n-1}\ha_{ij} \hf_j^{n-i}\in\hcA$ for all $j\gg 1$, \ie $f_j^n+\sum_{i=1}^{n-1} a_{ij} f_j^{n-i}\in\cA$, by Lemma~\ref{lem:sigma}. As a result, $f_j$ is integral over $\cA$, \ie $f_j\in\cA'$, which proves that $\cA'\to\hcA'$ has dense image. 

It remains to show that an element $f\in A$ with $\|f\|_{\cX^\beth}\le 1$ is integral over $\cA$. Since $\|\hf\|_{\sup}\le 1$, we already know that $\hf$ belongs to $\hcA'$. Since $\cA'\to\hcA'$ has dense image and $\hcA$ is open in $\hA$, we find $f'\in\cA'$ with $\hf-\hf'\in\hcA$, \ie $f-f'\in\cA$, and hence $f\in\cA'$ (see also~\cite[Theorem 2.10]{CM}) for a direct proof). 
\end{proof}

Besides the sup-seminorm $\n_{\cX^\beth}$, $A$ is also equipped with a 'lattice seminorm' $\n_\cA$, defined by
$$
\|f\|_\cA:=\inf\left\{|a|\mid  a\in K,\,f\in a\cA\right\}. 
$$
By Example~\ref{exam:notinj}, this is again not a norm in general. However, similarly setting for $g\in\hA$ 
$$
\|g\|_{\hcA}:=\inf\left\{|a|\mid  a\in K,\,g\in a\hcA\right\},
$$
does yield a norm on $\hA$:

\begin{lem}\label{lem:infnorm} For each $g\in\hA$, the infimum defining $\|g\|_{\hcA}$ is achieved. In particular, $\hcA$ is the closed unit ball of $\n_{\hcA}$, and $K^{\circ\circ}\hcA$ is its open unit ball. 
\end{lem} 
\begin{proof} The result is true for the polynomial ring $\cB:=K^\circ[t_1,\ldots,t_r]$, since $\n_{\hcB}$ is then the Gauss norm on the Tate algebra $\widehat{B}=K\{t_1,\ldots,t_r\}$. In the general case, choose a surjection $\cB:=K^\circ[t_1,\ldots,t_r]\twoheadrightarrow\cA$ for some $r\ge 1$, and observe that $\n_{\hcA}$ is the quotient seminorm of $\n_{\hcB}$ with respect to the induced surjection $\rho:\widehat{B}\twoheadrightarrow\hA$. The kernel of $\rho$, being an ideal in a Tate algebra, is strictly closed~\cite[5.2.7/8]{BGR}. By definition, this means that for each $g\in\hA$, there exists $h\in \widehat{B}$ such that $\rho(h)=g$ and $\|h\|_{\hcB}=\|g\|_{\hcA}$. Since the desired result holds for $\hcB$, we can then find $ a\in K$ with $|a|=\|h\|_{\hcB}=\|g\|_{\hcA}$, which implies that $h\in a\hcB$, and hence $g=\rho(h)\in a\hcA$. 
\end{proof}

\begin{thm}\label{thm:redint} As above, let $\cX = \Spec \cA$ be an affine model of $X = \Spec A$. We then have $\n_{\sup}\le\n_{\hcA}$ on $\hA$, and hence $\n_{\cX^\beth}\le\n_{\cA}$ on $A$. Consider further the following properties: 
\begin{itemize}
\item[(i)] $\cX_s$ is reduced; 
\item[(ii)] $\n_{\sup}=\n_{\hcA}$ on $\hA$;
\item[(iii)] $\n_{\cX^\beth}=\n_{\cA}$ on $A$;
\item[(iv)] $\hcA$ is integrally closed in $\hA$;
\item[(v)] $\cA$ is integrally closed in $A$.
\end{itemize}
Then $(i)\Longleftrightarrow(ii)\Longleftrightarrow(iii)\Longrightarrow(iv)\Longleftrightarrow(v)$. If $K$ is densely valued, we also have $(v)\Longrightarrow(i)$, and $(i)$---$(v)$ are then equivalent. 
\end{thm}

\begin{proof} Pick a nonzero $g\in\hcA$. By Lemma~\ref{lem:infnorm}, $\|g\|_{\hcA}$ is in the value group $|K^\times|$, and we may thus assume that $\|g\|_{\hcA}=1$ after multiplying $g$ by a nonzero scalar. We then have $g\in\hcA$, hence $|g(x)|\le 1$ for all $x \in\cX^\beth$, which proves that $\|g\|_{\sup}\le 1$, and hence $\n_{\sup}\le\n_{\hcA}$. 

Suppose now that $\cX_s$ is reduced, and assume by contradiction that $g$ as above satisfies $\|g\|_{\sup}<1$. By~\cite[6.2.3/2]{BGR}, $g$ is topologically nilpotent, \ie $g^n\to 0$. For $n\gg 1$, we thus have $\|g^n\|_{\hcA}<1$, \ie $g^n\in K^{\circ\circ}\hcA$; since  $\cA\otimes\tK\simeq\hcA\otimes\tK$ is reduced, this implies $g\in K^{\circ\circ}\hcA$, which contradicts $\|g\|_{\hcA}=1$. We have thus proved (i)$\Longrightarrow$(ii), which trivially implies (iii) by composing with $A\to\hA$. If (iii) holds, then $\n_{\cA}$ is power-multiplicative, \ie $\|f^n\|_\cA=\|f\|_\cA^n$ for each $f\in A$ and $n\in\N$. In particular, $f^n\in K^{\circ\circ}\cA\Longleftrightarrow\|f^n\|_\cA<1\Longleftrightarrow f\in K^{\circ\circ}\cA$, which means that $\cA\otimes\tK\simeq\cA/K^{\circ\circ}\cA$ is reduced. 

Since $\hcA$ (resp. $\cA$) is the unit ball of $\n_{\hcA}$ (resp. $\n_{\cA})$, Theorem~\ref{thm:intunit} shows that (ii) and (iii) respectively imply (iv) and (v), while Corollary~\ref{cor:intunit} shows that (iv) and (v) are equivalent. 

Assume finally that $K$ is densely valued and that (v) holds. To prove (i), we need to show that each $f\in\cA$ such that $f^n\in a\cA$ for some $n\ge 1$ and $a\in K^{\circ\circ}$ actually satisfies $f\in K^{\circ\circ}\cA$. Since $K$ is densely valued, we can find $a'\in K^{\circ\circ}$ with $|a|^{1/n}\le|a'|< 1$, and hence $a/a'^n\in K^\circ$. As a result, $g:=a'^{-1}f\in A$ satisfies $g^n\in\cA$, and hence $g\in\cA$, since $\cA$ is  integrally closed in $A$. We have thus shown as desired that $f=a' g\in K^{\circ\circ}\cA$ (we are grateful to Walter Gubler for his help with this argument). 
\end{proof}

We conclude this section with the following rather special case of the scheme-theoretic version of the Bosch--L\"utkebohmert--Raynaud reduced fiber theorem~\cite[Theorem 2.1']{BLR}. We provide some details for the convenience of the reader (see also~\cite[Th\'eor\`eme 1', p.73]{Ana}). 

\begin{thm}\label{thm:redfiber} Assume that $K$ is non-Archimedean and nontrivially valued, and either discretely valued or algebraically closed. Let $X$ be a reduced $K$-scheme of finite type, and pick a model $\cX$ of $X$. The integral closure $\cX'$ of $\cX$ in $X$ is then finite over $\cX$, and hence a model of $X$ as well. In the algebraically closed case, $\cX'_s$ is further reduced. 
\end{thm}
Note conversely that the existence of a model with reduced special fiber implies that $X$ is reduced, by~\cite[IV,12.1.1]{EGA}. 

\begin{proof}[Proof of Theorem~\ref{thm:redfiber}] If $K$ is discretely valued, then $\cX$ is excellent, which implies that its integral closure in $X$ is finite. Assume that $K$ is algebraically closed. It is then densely valued, and the final point will thus follow from Theorem~\ref{thm:redint}. 

The finiteness of $\cX'$ over $\cX$ being local, we assume that $\cX=\spec(\cA)$ is affine and use the above notation. We will reduce the result to the Grauert--Remmert finiteness theorem, basically arguing as in~\cite[Proposition 1.5]{BL2} and~\cite[Theorem 3.5.5, Step 3]{Tem10}. 

Since $A$ is reduced, $\hA$ is reduced as well by Lemma~\ref{lem:seminorm}, and~\cite[6.4.1/5]{BGR} thus shows that $\hcA'$ is finite over $\hcA$, \ie $\hcA'=\sum_i \hcA g_i$ for a finite set $g_i\in\hA$, in which we include $1$ for convenience. As in the proof of Lemma~\ref{lem:sigma}, we can find for each $i$ some $f_i\in A$ with $g_i-\hf_i\in\hcA$, and hence $\hcA'=\sum_i\hcA \hf_i$. By Corollary~\ref{cor:intunit}, an element $f\in A$ belongs to the integral closure $\cA'$ of $\cA$ in $A$ if and only $\hf$ belongs to $\hcA'=\sum_i\hcA \hf_i$, which is also equivalent to $f\in\sum_i\cA f_i$ by Lemma~\ref{lem:sigma}. We conclude as desired that $\cA'=\sum_i\cA f_i$ is finite over $\cA$. 
\end{proof}

\begin{rmk} Theorem~\ref{thm:redfiber} fails in general over an arbitrary densely valued non-Archimedean field $K$. Let indeed $X=\spec A$ be an affine reduced $K$-scheme, $\cX=\spec\cA$ an affine model, and denote by $\cA'$ the integral closure of $\cA$ in $A$. If $K$ is densely valued and $\cA'$ is finite over $\cA$, then $\cX':=\spec\cA'$ is a model of $X$, and hence $\cX'_s$ is reduced, by Theorem~\ref{thm:redint}. As a result, the sup-seminorm $\n_{\cX^\beth}=\n_{\cX'^\beth}=\n_{\cA'}$ takes values in $|K|$, a condition that is not satisfied in general when the group $|K^\times|$ is not divisible. 
\end{rmk} 

%
%
\section{Fubini--Study metrics and model metrics}
This section introduces Fubini--Study metrics, and compares them with model metrics. The main result is Theorem~\ref{thm:supmodel}, which compares the supnorms and lattice norms induced by a model metric, and relies on the reduced fiber theorem. 

In what follows, $X$ denotes a \emph{projective} $K$-scheme, where $K$ is a complete valued field.
%
%
\subsection{Metrics}
Let $L$ be a line bundle on $X$. A \emph{continuous metric} $\phi$ on $L$ is a family of norms 
$$
|\cdot|_{\phi_x}:L_x:=L\otimes\cH(x)\to[0,+\infty),\,x\in X^\an,
$$
such that for any local section $s$ of $L$ on an open $U\subset X$, the induced function $|s|_\phi$ on $U^\an$ is continuous. We use additive notation for metrics, which amounts to the following rules: 
\begin{itemize}
\item[(i)] if $\phi,\p$ are continuous metrics on line bundles $L,M$, then $\phi\pm\p$ denotes the induced metric on $L\pm M=L\otimes M^{\pm 1}$, \ie
$$
|\cdot|_{\phi\pm\psi}= |\cdot|_\phi \otimes |\cdot|_\psi^{\pm 1}; 
$$
\item[(ii)] a continuous metric $\phi$ on the trivial line bundle $L=\cO_X$ is identified with the continuous function $-\log|1|_\phi$ on $X^\an$. 
\end{itemize}
If $\phi,\p$ are two continuous metrics on the same line bundle $L$, $\phi-\p$ is thus a continuous function on $X^\an$, which means that the space $\cz(L)$ of continuous metrics on $L$ is an affine space modeled on $\cz(X^\an)$. 

A \emph{smooth metric} $\phi$ on $L$ is a continuous metric such that $|s|_\phi$ is smooth for any local trivialization $s$ of $L$. In the non-Archimedean case, smoothness is understood in the sense of~\cite{CLD}, see~\S\ref{sec:smooth}. The set of smooth metrics $C^\infty(L)\subset \cz(L)$ is an affine space modeled on $C^\infty(X)$. 

We occasionally use the notion of a \emph{singular metric} on $L$, by which we mean a metric of the form $\phi=\p+f$ with $\p\in \cz(L)$ and $f:X^\an\to[-\infty,+\infty)$ an arbitrary function. For a local trivialization $s$ of $L$, $|s|_\phi=|s|_\p e^{-f}$ is thus allowed to be $+\infty$ at certain points. 

If $f$ is bounded, we say that $\phi$ is a \emph{bounded metric}, defining an affine space $\cL^\infty(L)\supset \cz(L)$ modeled on the space of bounded functions. 

\begin{exam}\label{exam:singsec} Every section $s\in \Hnot(X,L)$ defines a singular metric $\log|s|$ on $L$, such that $\log|s|=\p+\log|s|_\p$ for any $\p\in \cz(L)$. 
\end{exam}

For each $m\ge 1$, every metric on $L$ is of the form $m\phi$ with $\phi$ a metric on $L$. As a result, all of the above notions immediately extend to $\Q$-line bundles. 

%
%
\subsection{Fubini--Study metrics}\label{sec:FS}
The usual Fubini--Study metric on the tautological ample line bundle $\cO(1)$ over the projective space (also known as the Weil metric in the non-Archimedean context) generalizes in a natural way as follows (see Corollary~\ref{cor:FSdiag} below for a more conceptual characterization). 

\begin{defi}\label{defi:FSmetric} A metric $\phi$ on a line bundle $L$ over $X$ is called a \emph{Fubini--Study metric} if there exists $m\ge 1$, a finite set of sections $(s_i)$ of $mL$ without common zeroes, and constants $\la_i\in\R$ such that
\begin{itemize}
\item[(A)] $\phi=\frac{1}{2m}\log\sum_i|s_i|^2 e^{2\la_i}$ (Archimedean case);
\item[(NA)] $\phi=\frac{1}{m}\max_i\{\log|s_i|+\la_i\}$ (non-Archimedean case). 
\end{itemize}
We say that $\phi$ is \emph{pure} if $\la_i=0$.  
\end{defi}

By definition, $L$ admits a (pure) Fubini--Study metric iff $L$ is \emph{semiample}, \ie $mL$ is globally generated for some $m\ge 1$. 

The notation $|s_i|=e^{\log|s_i|}$ is understood as in Example~\ref{exam:singsec}, and thus means that for any local trivializing section $\tau$ of $L$, the corresponding functions $f_i:=s_i/\tau^m\in\cO_X$ satisfy 
$$
-\log|\tau|_\phi=\frac{1}{2m}\log\sum_i|f_i|^2 e^{-2\la_i}\,\,\,\,\,\,(\text{resp. }\frac{1}{m}\max_i\{\log|f_i|+\la_i\}).
$$
 In the Archimedean case, one can replace $s_i$ with $e^{\la_i}s_i$, and all Fubini--Study metrics are thus pure in that case. 

Given a subgroup $\Ga\subset\R$, we say that a metric $\phi$ as above is a \emph{$\Ga$-Fubini--Study metric} if $\la_i\in\Ga$ for all $i$. We denote by 
$$
\FS_\Ga(L)\subset\cz(L)
$$ 
the set of $\Ga$-Fubini--Study metrics on $L$, and by $\FS(L)=\FS_\R(L)$ the set of all Fubini--Study metrics. The set of pure Fubini--Study metrics is thus $\FS_{\{0\}}(L)$.

\begin{rmk} While Fubini--Study metrics are smooth in the Archimedean case, they are instead piecewise linear when $K$ is non-Archimedean (see \S\ref{sec:QPL} below). But it is anyway easy to explicitly approximate a Fubini--Study metric by smooth metrics, cf.~Theorem~\ref{thm:pshCLD}. 
\end{rmk}

The next result summarized the main properties of Fubini--Study metrics. 
\begin{prop}\label{prop:FS} If $L,L'$ are line bundles on $X$ and $\Ga\subset\R$ is a subgroup, then:
\begin{itemize}
\item[(i)] $\FS_\Ga(L)+\FS_\Ga(L')\subset\FS_{\Ga}(L+L')$; 
\item[(ii)] $\FS_\Ga(L)=\FS_{\Ga'}(L)$ with $\Ga':=\Q(\Ga+\Ga_K)$ the $\Q$-linear subspace of $\R$ spanned by $\Ga$ and $\Ga_K$; 
\item[(iii)] $\FS_\Ga(mL)=m\FS_\Ga(L)$ for all $m\in\Z_{>0}$; 
\item[(iv)] $f^\star\FS_\Ga(L)\subset\FS_\Ga(f^\star L)$ for any projective morphism $f:X'\to X$; 
\item[(v)] if $\Ga\ne\{0\}$, or $K$ is nontrivially valued, then $\FS_\Ga(L)$ is dense in $\FS(L)$ with respect to uniform convergence. 
\end{itemize}
When $K$ is non-Archimedean, we further have: 
\begin{itemize}
\item[(vi)] $\FS_\Ga(L)$ is stable under max.
\end{itemize}
\end{prop}
\begin{proof} Pick $m,m'\ge 1$ and finitely many sections without common $s_i\in \Hnot(mL)$, $s'_j\in \Hnot(m'L')$ such that $\phi=\frac{1}{2m}\log\sum_i|s_i|^2$, $\phi'=\frac{1}{2m'}\log\sum_j|s'_j|^2$ in the Archimedean case, and 
$\phi=\frac{1}{m}\max_i\{\log|s_i|+\la_i\}$, $\phi'=\frac{1}{m}\max_j\{\log|s'_j|+\la'_j\}$ in the non-Archimedean case. In the former case,  
$$
mm'(\phi+\phi')=\log\left(\sum_i|s_i|^2\right)^{m'}\left(\sum_j|s'_j|^2\right)^m, 
$$
which expands out as $mm'(\phi+\phi')=\log\sum_\a |\sigma_\a|^2$ for a finite set of sections $\sigma_ a\in \Hnot(mm'L)$ without common zeroes, proving that $\phi+\phi'$ is Fubini--Study. In the non-Archimedean case,
$$
mm'(\phi+\phi')=\max_i\{\log|s_i^{m'}|+m'\la_i\}+\max_j\{\log|s'^m_j|+m\la'_j\}
$$
$$
=\max_{i,j}\{\log|s_i^{m'}s'^m_j|+m'\la_i+m\la'_j\}, 
$$
which proves (i). The proof of (ii) and (iii) is similar, and (iv), (vi) follow directly from the definitions. To prove (v), note that the assumption implies that $\Ga'=\Q(\Ga+\Ga_K)$ is a nontrivial $\Q$-linear subspace of $\R$. It is thus dense in $\R$, and it is then straightforward to check that $\FS_\Ga(L)=\FS_{\Ga'}(L)$ is dense in $\FS(L)=\FS_\R(L)$, simply by approximating the coefficients $\la_i$ in Definition~\ref{defi:FSmetric}. 
\end{proof}

By Proposition~\ref{prop:FS}~(ii), we can make sense of $\Ga$-Fubini--Study metrics on any $\Q$-line bundle $L$ over $X$. As noted above, such metrics exist iff $L$ is semiample. As we now show, Fubini--Study metrics always descend to an ample $\Q$-line bundle. 

\begin{lem}\label{lem:Stein} If $L$ is a semiample $\Q$-line bundle on $X$, then there exists a surjective morphism $f:X\to Y$ to a projective $K$-scheme with $f_\star\cO_X=\cO_Y$ and an ample $\Q$-line bundle $A$ on $Y$ such that $f^\star A=L$. Further, $f$ and $(Y,A)$ are unique up to isomorphism. 
\end{lem}
We call $A$ the \emph{Stein factorization} of $L$. 
\begin{proof} After passing to a multiple, we may assume that $L$ is a globally generated line bundle. We then have a morphism $h:X\to\P \Hnot(L)$ such that $L=h^\star\cO(1)$. By Stein factorization, we have $h=g\circ f$ with $f:X\to Y$ such that $f_\star\cO_X=\cO_Y$ and $g:Y\to\P \Hnot(L)$ finite. Thus $A:=g^\star \cO(1)$ is ample, and $L=f^\star A$. By the projection formula, we have 
$\Hnot(Y,mA)\simeq \Hnot(X,mL)$ for all $m$. This shows that $(Y,A)$ is recovered from the graded ring $R(X,L)=\bigoplus_{m\in\N} \Hnot(X,mL)$ by the Proj construction, and hence is unique up to isomorphism. 
\end{proof}

\begin{lem}\label{lem:Steinbis} Let $L$ be a semiample $\Q$-line bundle on $X$. Denote by $X_\redu\subset X$ the reduction of $X$, and by $L_\redu$ the restriction of $L$. For all $m$ sufficiently divisible, the restriction map $\Hnot(X,mL)\to \Hnot(X_\redu,mL_\redu)$ is surjective. 
\end{lem}
\begin{proof} Use the notation of Lemma~\ref{lem:Stein}. Since every germ on $\cO_{Y_\redu}$ locally lifts to $\cO_Y$, we have $f_\star\cO_{X_{\redu}}=\cO_{Y_\redu}$, hence $\Hnot(X_\redu,mL_\redu)\simeq \Hnot(Y_\redu,mA_\redu)$, by the projection formula. We are thus reduced to the case where $L$ is ample, which follows from Serre vanishing. 
\end{proof}

\begin{cor}\label{cor:Stein} Let $L$ be a semiample $\Q$-line bundle on $X$. Denote by $A$ the Stein factorization of $L$ as in Lemma~\ref{lem:Stein}, and by $L_\redu$ the restriction of $L$ to $X_\redu$. Then 
$$
\FS_\Ga(A)\simeq\FS_\Ga(L)\simeq\FS_\Ga(L_\redu).
$$
\end{cor}

%
%
\subsection{Model metrics}\label{sec:modelmetric}
In this section, $K$ is non-Archimedean (possibly trivially valued). A \emph{model} of a line bundle $L$ on the projective $K$-scheme $X$ is a line bundle $\cL$ on a \emph{projective} model $\cX$ of $X$, together with an isomorphism $\cL|_{\cX_K}\simeq L$ compatible with the given isomorphism $\cX_K\simeq X$. When $L=\cO_X$, each model of $L$ is of the form $\cO_\cX(D)$ where $D$ is a \emph{vertical} Cartier divisor on a projective model $\cX$, \ie $\supp D\subset\cX_s$. 

\begin{lem}\label{lem:modelexist} Let $L$ be a line bundle on $X$, and $\cX$ a projective model of $X$. Then we can find a projective model $\cX'$ dominating $\cX$ and a model $\cL$ of $L$ determined on $\cX'$. 
\end{lem}
\begin{proof} Pick very ample line bundles $A_1,A_2$ on $X$ such that $L=A_1-A_2$. Sections of $A_i$ yields an embedding $X\hookrightarrow\P^{N_i}_K$ such that $A_i=\cO(1)|_X$. The scheme theoretic closure of $X$ in $\P^{N_i}_{K^\circ}$ is thus a projective model $\cX_i$ of $X$, and $\cA_i:=\cO(1)|_{\cX_i}$ is a model of $A_i$ determined on $\cX_i$. It remains to choose a projective model $\cX'$ dominating $\cX$, $\cX_1$ and $\cX_2$, and to define $\cL$ as the difference of the pullbacks to $\cX'$ of $\cA_1$ and $\cA_2$.  
\end{proof}

A model $\cL$ of $L$ defines a continuous metric $\phi_\cL$ on $L$, as follows. Cover $\cX$ with finitely many open subschemes $\cU_i$ with a trivializing section $\tau_i$ of $\cL$. Since $\cX$ is in particular proper over $K^\circ$, $X^\an=\cX^\beth$ is covered by the compact sets $\cU_i^\beth$. We may thus define a continuous metric $\phi_\cL$ on $L^\an$ by requiring that $|\tau_i|_{\phi_\cL}=1$ on $\cU_i^\beth$. This is indeed well-defined, since any other trivializing section of $\cL$ on $\cU_i$ is of the form $u_i\tau_i$ with $u_i\in\cO^\times(\cU_i)$ a unit, and hence $|u_i|\equiv 1$ on $\cU_i^\beth$. 

\begin{exam}\label{exam:trivialmodel} If $K$ is trivially valued, the model metric defined by the unique model $(X,L)$ is called the \emph{trivial metric} of $L$. 
\end{exam}

\begin{lem}\label{lem:func} Let $\cL$ be a model of $L$, with corresponding metric $\phi_\cL$. 
\begin{itemize}
\item[(i)] For each $m\in\Z$ we have $\phi_{m\cL}=m\phi_\cL$.
\item[(ii)] If a model $\cX'$ dominates $\cX$, then the pull-back $\cL'$ of $\cL$ to $\cX'$ satisfies $\phi_{\cL'}=\phi_{\cL}$. 
\end{itemize}
\end{lem}
\begin{proof} If $\tau$ is trivializing section of $\cL$ on an open set $\cU$, then $\tau^m$ is a trivializing section of $m\cL$, and the pull-back of $\tau$ is a local trivialization of $\cL'$ on the inverse image $\cU'$, which satisfies $\cU'^\beth=\cU^\beth$ since $\cU'$ is proper over $\cU$ (cf.~Lemma~\ref{lem:doman}). 
\end{proof}

If $L$ is now a $\Q$-line bundle on $X$, a \emph{$\Q$-model} of $L$ is defined as a $\Q$-line bundle $\cL$ on a projective model $\cX$ of $X$ such that $m\cL$ is a model of $mL$ for some $m\ge 1$. By Lemma~\ref{lem:func}~(i), the metric $\phi_\cL:=m^{-1}\phi_{m\cL}$ on $L$ is independent of the choice of $m$.
\begin{defi} A \emph{model metric} on a $\Q$-line bundle $L$ is a metric of the form $\phi=\phi_\cL$, where $\cL$ is a $\Q$-model of $L$. A \emph{model function} is a model metric $\phi$ on $\cO_X$, identified with the continuous function $-\log|1|_\phi$ on $X^\an$. 
\end{defi}
A model function $f$ is thus determined by a vertical $\Q$-Cartier divisor $D$ on some projective model $\cX$ of $X$. Every line bundle $L$ on $X$ admits a model metric $\phi$, and any other model metric on $L$ is then of the form $\phi+f$ with $f$ a model function.

\begin{thm}\label{thm:modelmet} Assume that $K$ is either discretely (or trivially) valued, or algebraically closed. Let $\cL,\cL'$ be two models of a the same line bundle $L$, determined on models $\cX,\cX'$ of $X$. Then the model metrics $\phi_\cL$, $\phi_{\cL'}$ on $L$ coincide iff the pullbacks of $\cL,\cL'$ to some model $\cX''$ dominating both $\cX$ and $\cX$' are equal. 
\end{thm}

\begin{proof} If the pullbacks of $\cL,\cL'$ to some high enough model agree, then $\phi_\cL=\phi_{\cL'}$, by Lemma~\ref{lem:func}. Assume conversely that $\phi_{\cL}=\phi_{\cL'}$. By Theorem~\ref{thm:redfiber}, $\cX$ and $\cX'$ are dominated by a model $\cX''$ which is integrally closed in $X$. After pulling back $\cL$ and $\cL'$ to $\cX''$, we may thus assume that $\cL$ and $\cL'$ are determined on $\cX$ integrally closed in $X$, and we then claim that $\cL=\cL'$. To see thus, let $\cU=\spec(\cA)$ be an affine open subscheme of $\cX$ with trivializing sections $\tau\in \Hnot(\cU,\cL)$, $\tau'\in \Hnot(\cU,\cL')$. Since $\cL,\cL'$ are both models of the same line bundle $L$, the restrictions of $\tau,\tau'$ to the generic fiber $U=\spec(\cA_K)$ of $\cU$ satisfy $\tau'|_U=u\tau|_U$ with $u\in\cA_K$ a unit. As $\phi_{\cL}$ and $\phi_{\cL'}$ coincide, the definition of model metrics yields $|u|\equiv 1$ on $\cU^\beth$.  By Theorem~\ref{thm:redint}, $\cA$ coincides with the unit ball of $\n_{\sup}$ on $\cA_K$. As $\|u\|_{\sup}=\|u^{-1}\|_{\sup}=1$, it follows that $u$ and $u^{-1}$ belong to $\cA$, \ie $u$ is a unit in $\cA$, and we conclude that $\cL=\cL'$. 
\end{proof}

The next result describes the behavior of model metrics under pullback.
\begin{prop}\label{prop:modelpullback} Let $L$ be a $\Q$-line bundle on $X$, and $\phi_\cL$ be the model metric determined by a $\Q$-model $\cL$ of $L$ on a model $\cX$ of $X$. For any projective morphism $f:Y\to X$, the induced metric $f^\star \phi$ on $f^\star  L$ is also a model metric. More precisely, the family of projective models $\cY$ of $Y$ such that $f$ extends to a $K^\circ$-morphism $f:\cY\to\cX$ is cofinal in all projective models, and we have $f^\star \phi_\cL=\phi_{f^\star \cL}$ for any such model $\cY$. 

If $f\colon Y\to X$ is flat, then $\cY$ can further be assumed to be flat over $\cX$. 
\end{prop}
\begin{proof} Pick a projective model $\cY'$ of $Y$, and denote by $\cY$ the schematic (or flat) closure in $\cY'\times_{K^\circ}\cX$ of the graph of $f$. Then $\cY$ dominates $\cY'$, and the projection $\cY\to\cX$ extends $f$. To see the last point, we may assume that $\cL$ is a line bundle. Let $(\cU_i)$ be a finite open cover of $\cY$ with trivializing sections $\tau_i\in \Hnot(\cU_i,\cL)$. Then $(f^{-1}(\cU_i))$ is an open cover of $\cY$ with trivializing sections $f^\star \tau_i$ for $f^\star \cL$, and the result easily follows. \\
For the last point, by the Raynaud--Gruson flattening theorem~\cite{RG}, we can blow up $\cX$ and take the proper transform of $\cY$ to obtain a flat morphism $f: \cY' \to \cX'$ of models. By Lemma~\ref{lem:func} this does not change the metrics induced by the models.
\end{proof}

We are now in a position to compare model metrics and Fubini--Study metrics (see~\cite[Proposition 3.8]{CM} for a related result). 

\begin{thm}\label{thm:puremodel} For a continuous metric $\phi$ on a $\Q$-line bundle $L$ over $X$, the following are equivalent:
\begin{itemize}
\item[(i)] $\phi$ is a pure Fubini--Study metric;
\item[(ii)] $\phi$ is a model metric determined by a semiample $\Q$-model $\cL$ of $L$, \ie $m\cL$ is a globally generated line bundle when $m\in\N$ is sufficiently divisible. 
\end{itemize}
\end{thm}
 
 \begin{proof} The key observation is that the Weil metric $\log\max_i|z_i|$ on $\P^r_K$ with homogeneous coordinates $[z_0:\dots:z_r]$ coincides with the model metric $\phi_{\cO(1)}$ determined by the canonical model of $(\P^r,\cO(1))$ over $K^\circ$. 

Let first $\phi$ be a pure Fubini--Study metric on $L$. After replacing $L$ by a multiple, we may assume that $L$ is globally generated and $\phi=\log\max_i|s_i|$ with $(s_i)$ a basis of $\Hnot(L)$. These sections induce a morphism of $K$-schemes $f:X\to\P^r_K$ with an identification  $L=f^\star \cO(1)$ such that $\phi=f^\star \phi_{\cO(1)}$. By Proposition~\ref{prop:modelpullback}, $X$ admits a projective model $\cX$ such that $f$ extends to a morphism $f:\cX\to\P^r_{K^\circ}$, and $\phi$ is the model metric defined by $(\cX,f^\star \cO(1))$. Since $f^\star \cO(1)$ is globally generated, this proves (i)$\Longrightarrow$(ii). Conversely, let $\cL$ be a semiample $\Q$-model of $L$. After passing to a multiple, we may assume that $\cL$ is a globally generated line bundle. Choosing a $K^\circ$-basis $(s_i)$ of the lattice $\Hnot(\cL)$ yields a morphism $f:\cX\to\P^r_{K^\circ}$ with $f^\star \cO(1)=\cL$, and hence $\phi_\cL=f^\star \phi_{\cO(1)}=\log\max_i|s_i|$, by Proposition~\ref{prop:modelpullback}.
\end{proof}

%
%
\subsection{PL metrics}\label{sec:QPL}
In this section, $K$ is non-Archimedean, and $\Ga\subset\R$ is an additive subgroup. 

\begin{defi}\label{defi:QPL} We say that a continuous metric $\phi$ on a $\Q$-line bundle $L$ is \emph{$\Ga$-piecewise linear}, or \emph{$\Ga$-PL} for short, if there exists (semiample) $\Q$-line bundles $M,M'$ and $\Ga$-Fubini--Study metrics $\p,\p'$ on $M,M'$ such that $L=M-M'$ and $\phi=\p-\p'$. 
\end{defi}
We denote by $\PL_\Ga(L)$ the set of $\Ga$-PL metrics on $L$. A $\Ga$-PL metric on $\cO_X$ is called a $\Ga$-PL function, defining a subset $\PL_\Ga(X^\an)\subset \cz(X^\an)$. 

When $\Ga=\{0\}$, we speak of \emph{pure PL metrics} and \emph{pure PL functions}, respectively. 

\begin{prop}\label{prop:QPL} Let $L,L'$ be $\Q$-line bundles on $X$. Then:
\begin{itemize}
\item[(i)] $\PL_\Ga(L)+\PL_\Ga(L')\subset\PL_\Ga(L+L')$; 
\item[(ii)] $\PL_\Ga(L)$ is stable under max and min; 
\item[(iii)] $\PL_\Ga(mL)=m\PL_\Ga(L)$ for all nonzero $m\in\Z$; 
\item[(iv)] $\PL_\Ga(L)=\PL_{\Ga'}(L)$ with $\Ga'=\Q(\Ga+\Ga_K)$; 
\item[(v)] $\PL_\Ga(X^\an)$ is a $\Q$-linear subspace of $\cz(X^\an)$, and $\PL_\Ga(L)$ is a $\Q$-affine space modeled on $\PL_\Ga(X^\an)$; 
\item[(vi)] pure PL metrics on $L$ coincide with model metrics. 
\end{itemize}
\end{prop}
We refer to~\cite[Theorem 1.1]{GM} for an alternative description of (pure) PL metrics. 

\begin{proof} (i)---(v) are direct consequences of Proposition~\ref{prop:FS}.

To prove (vi), pick a $\Q$-model $\cL$ of $L$, determined on a projective model $\cX$ of $X$. Pick an ample line bundle $\cA$ on $\cX$, and denote by $A$ its restriction to $X$. For $a\gg 1$, $\cL+a\cA$ is then ample on $\cX$. By Theorem~\ref{thm:puremodel}, $\phi_{\cL+a\cA}$ and $\phi_{a\cA}$ are pure Fubini--Study metrics, and $\phi_\cL=\phi_{\cL+a\cA}-\phi_{a\cA}$ is thus a pure PL metric. 

Conversely, Theorem~\ref{thm:puremodel} implies that any pure Fubini--Study metric is a model metric. Every pure PL metric is thus also a model metric, which concludes the proof of (vi). 
\end{proof}

For later use, we also note: 
\begin{lem}\label{lem:QPL} Let $L$ be a $\Q$-line bundle, and $\phi$ be a $\Ga$-PL metric on $L$. Then we can find $\Ga$-Fubini--Study metrics $\p,\p'$ on ample $\Q$-line bundles $A,A'$ such that $L=A-A'$ and $\phi=\p-\p'$. 
\end{lem}
\begin{proof} By definition, we can find $\Ga$-Fubini--Study metrics $\phi,\phi'$ on $\Q$-line bundles $M,M'$ such that $L=M-M'$ and $\phi=\p-\p'$. Pick an ample line bundle $H$, and $a\gg 1$ such that $aH-M$ and $aH-L$ are both ample, and pick $\tau\in\FS_\Ga(aH-M)$. By Proposition~\ref{prop:FS}, $\p+\tau$, and $\p'+\tau$ are $\Ga$-Fubini--Study metrics on the ample line bundles $A:=aH$, $A':=aH-L$, and $\phi=(\p+\tau)-(\p'+\tau)$, which proves the result. 
\end{proof}

The following result generalizes the well-known density of model functions~\cite{GM} in the nontrivially valued case. 

\begin{thm}\label{thm:PLdense} If $\Ga\ne\{0\}$, or if $K$ is nontrivially valued, then $\PL_\Ga(X^\an)$ is dense in $\cz(X^\an)$. 
\end{thm}

\begin{proof} The assumption guarantees that $\Ga'=\Q(\Ga+\Ga_K)$ is nontrivial; replacing $\Ga$ with $\Ga'$, we may thus assume in all cases that $\Ga\ne \{0\}$, by Proposition~\ref{prop:QPL}~(iv). By Proposition~\ref{prop:QPL}, $\PL_\Ga(X^\an)$ is a $\Q$-linear subspace of $\cz(X^\an)$, stable under max. By the 'lattice version' of the Stone--Weierstrass theorem, it will thus be enough to show that $\PL_\Ga(X^\an)$ separates the points of $X^\an$. 

To see this, fix $\gamma\in\Ga\cap\R_{>0}$, $\rho\in\FS_\Ga(L)$, and pick $x_1\ne x_2\in X^\an$ and a very ample line bundle $L$. After replacing $L$ with a multiple, we can find nonzero sections $s,s'\in \Hnot(L)$ such that $s'$ does not vanish at $x_1,x_2$ and $|s/s'|(x_1)\ne|s/s'|(x_2)$. For each $m\in\N$ set 
$$
\phi_m=\max\{\log|s|,\rho-m\gamma\},\quad\phi'_m:=\max\{\log|s'|,\rho-m\gamma\},\quad f_m:=\phi_m-\phi'_m.
$$
Then $\phi_m,\phi'_m\in\FS_\Ga(L)$, and hence $f_m\in\PL_\Ga(X^\an)$. Since $\log|s'|$ is finite at $x_1,x_2$, $\phi'_m=
\log|s'|$ at $x_1,x_2$ for all $m$ large enough. Thus 
$$
f_m(x_i)=\max\{\log|s/s'|(x_i),\rho(x_i)-m\gamma\}\to \log|s/s'|(x_i)
$$
as $m\to\infty$, and $|s/s'|(x_1)\ne|s/s'|(x_2)$ thus yields $f_m(x_1)\ne f_m(x_2)$ for $m\gg 1$. This proves as desired that $\PL_\Ga(X^\an)$ separates the points of $X^\an$. 
\end{proof}
%
%
\section{The supnorm of a metric}\label{sec:supnorm}
In this section, $K$ is an arbitrary complete valued field, $X$ is a projective $K$-scheme, and $L$ is a line bundle on $X$. 

%
%
%
\subsection{Supnorms vs.~lattice norms}
 
The data of a bounded metric $\phi$ on $L$ defines a sup-seminorm $\n_\phi$ on $\Hnot(L)$ by setting 
$$
\|s\|_\phi:=\sup_{X^\an}|s|_\phi
$$
for each $s\in \Hnot(L)$. If $X$ is reduced, then $\n_\phi$ is a norm, by Lemma~\ref{lem:nilpo}.

\begin{lem}\label{lem:supground} Let $F/K$ be a complete field extension, and assume that both $X$ and $X_F$ are reduced. Let $\phi$ be a bounded metric on $L$, and denote by $\phi_F$ the pullback of $\phi$ to the base change $L_F$. Then $\n_{\phi_F}=\n_\phi$ on $\Hnot(L)\hookrightarrow \Hnot(L_F)=\Hnot(L)_F$, and $\n_{\phi_F}\le(\n_\phi)_F$ on $\Hnot(L_F)$. 

If $K$ is Archimedean, we further have 
$$
\tfrac 12(\n_\phi)_F\le\n_{\phi_F}\le(\n_\phi)_F.
$$
\end{lem}
\begin{proof} Denoting by $p:X_F^\an\to X^\an$ the surjective projection map, we plainly have for $s\in \Hnot(L)$
$$
\|s\|_{\phi_F}=\sup_{X_F^\an}|s|_\phi\circ p=\sup_{X^\an}|s|_\phi=\|s\|_\phi. 
$$
By Proposition~\ref{prop:ground}, $(\n_\phi)_F$ is the largest norm on $\Hnot(L)_F$ that coincides with $\n_\phi$ on $\Hnot(L)$, thus $\n_{\phi_F}\le(\n_\phi)_F$. If $K$ is Archimedean, the only nontrivial case is $K=\R$, $F=\C$. The metric $\phi_\C$ is conjugation invariant, thus $\n_{\phi_\C}$ is conjugation invariant as well, and hence $\n_{\phi_\C}\ge\tfrac 12(\n_{\phi})_\C$, by Proposition~\ref{prop:ground} again. 
\end{proof}

The inequality $\n_{\phi_F}\le(\n_\phi)_F$ of Lemma~\ref{lem:supground} is strict in general. One has indeed: 

\begin{lem}\label{lem:supgroundgalois} Assume that $K$ is discretely (nontrivially) valued, and let $F/K$ be a finite Galois extension. The following are equivalent:
\begin{itemize}
\item[(i)] $F/K$ is tamely ramified, \ie the residue field extension $\tF/\tK$ is separable and the ramification order $\left[|F^\times |:|K^\times|\right]$ is prime to the residue characteristic; 
\item[(ii)] for each bounded metric $\phi$ on a line bundle $L$ over a reduced projective $K$-scheme $X$ such that $X_F$ is reduced, we have $\n_{\phi_F}=(\n_\phi)_F$.
\end{itemize}
\end{lem}

\begin{proof} In the situation of (ii), $\n_{\phi_F}$ is a Galois invariant norm on $\Hnot(L)_F$. If $F/K$ is tamely ramified, the descent result of~\cite[Proposition 5.1.1]{Rou} (see also~\cite{Pra}) implies that every Galois invariant norm is obtained by ground field extension, and hence (i)$\Longrightarrow$(ii). 

Conversely, for $X:=\spec F$, $L=\cO_X$ and $\phi=0$, the supnorm $\n_{\phi_F}$ is the spectral norm on $F\otimes_K F$, and~\cite[\S 5.1]{RTW} thus says that $\n_{\phi_F}=(\n_{\phi})_F$ if and only if $F$ is tamely ramified. 
\end{proof} 

In the remainder on this section, we assume that $K$ is non-Archimedean. 

\begin{lem}\label{lem:supmodel} $\cL$ be a model of $L$ determined on a projective model $\cX$ of $X$, with associated set of Shilov points $\Ga(\cX)\subset X^\an$. Denote by $\phi=\phi_\cL$ the induced model metric on $L$, and recall that $\n_{\Hnot(\cL)}$ denotes the norm on $\Hnot(L)$ determined by the lattice $\Hnot(\cL)$. Then: 
\begin{itemize}
\item[(i)] for each $s\in \Hnot(L)$, we have $\|s\|_\phi=\max_{\Ga(\cX)}|s|_\phi$;
\item[(ii)] $\n_\phi\le\n_{\Hnot(\cL)}$; 
\item[(iii)] if $\cX_s$ is further reduced, then $\n_\phi=\n_{\Hnot(\cL)}$. 
\end{itemize}
\end{lem}

\begin{proof} Cover $\cX$ with finitely many open subschemes $\cU_i$ with trivializing sections $\tau_i\in \Hnot(\cU_i,\cL)$. Denoting by $U_i$ the generic fiber of $\cU_i$, we have $s|_{U_i}=f_i\tau_i$ with $f_i\in\cO(U_i)$, and $|s|_\phi=|f_i|$ on $\cU_i^\beth$, by definition of a model metric. By Lemma~\ref{lem:Shilov}, it follows that $\sup_{\cU_i^\beth}|s|_\phi=\max_{\Ga(\cU_i)}|s|_\phi$, and (i) follows since $X^\an=\bigcup_i\cU_i^\beth$ and $\Ga(\cX)=\bigcup_i\Ga(\cU_i)$. 

In order to prove (ii) and (iii), let $s\in \Hnot(L)$ be a nonzero section. Since $\|s\|_{\Hnot(\cL)}$ belongs to value group $|K^\times|$, we may then multiply $s$ be a scalar and assume that $\|s\|_{\Hnot(\cL)}=1$, \ie $s\in \Hnot(\cL)$ but $s\notin K^{\circ\circ} \Hnot(\cL)$. Choose as above a finite cover $\cX$ by affine open subscheme $\cU_i=\spec(\cA_i)$ with a trivializing section $\tau_i$ of $\cL$, and write $s|_{U_i}=f_i\tau_i$ with $f_i\in\cA_i$. On $\cU_i^\beth$, we have $|s|_\phi=|f_i|\le 1$, and hence $\|s\|_\phi\le 1$, which proves (ii). 

Finally, suppose that $\cX_s$ is reduced. To prove (iv), it is enough to show that a section $s\in \Hnot(L)$ with $\|s\|_{\Hnot(\cL)}=1$ satisfies $\|s\|_\phi=1$. Assume by contradiction that a nonzero section $\|s\|_\phi<1$. For each $i$, we then have $\sup_{\cU_i^\beth}|f_i|<1$, and hence $a_i^{-1}f_i\in\cA_i$ for some nonzero $a_i\in K^{\circ\circ}$, by Theorem~\ref{thm:redint}. Setting $a:=a_{i_0}$ for an index $i_0$ achieving $\max_i|a_i|$, we infer that $a^{-1}f_i\in\cA_i$ for all $i$, and hence $a^{-1}s\in \Hnot(\cL)$, a contradiction.  
\end{proof} 

The next result will be a key ingredient in the proof of Theorem~\ref{thm:Euler} below. 

\begin{thm}\label{thm:supmodel} Assume that $X$ is geometrically reduced, let $\cL$ be a model of $L$ determined on a projective model $\cX$ of $X$, and denote by $\phi=\phi_\cL$ the induced model metric on $L$. Then the supnorm $\n_{m\phi}$ and the lattice norm $\n_{\Hnot\left(m\cL\right)}$ on $\Hnot\left(mL\right)$ satisfy 
$$
\n_{m\phi}\le\n_{\Hnot\left(m\cL\right)}\le C\n_{m\phi}.
$$ 
for a constant $C>0$ independent of $m$. 
\end{thm}
\begin{proof} Lemma~\ref{lem:supmodel} yields the left-hand inequality. Fix an algebraically closed complete field extension $F/K$, and denote by $(\cX_F,\cL_F)$ the base change of $(\cX,\cL)$ to $F^\circ$, which is thus a model of $(X_F,L_F)$. The pulled back metric $\phi_F$ is the model metric determined by $\cL_F$, and the ground field extension of $\n_{\Hnot(\cL)}$ to $\Hnot(L_F)=\Hnot(L)_F$ coincides with the lattice norm  $\n_{\Hnot(\cL_F)}$, by Lemma~\ref{lem:latticedense}. By Lemma~\ref{lem:supground} and Lemma~\ref{lem:supmodel}, we infer
$$
\n_{m\phi_F}\le\left(\n_{m\phi}\right)_F\le\n_{\Hnot(m\cL_F)}. 
$$
Since $X$ is geometrically reduced, $X_F$ is reduced, and Theorem~\ref{thm:redfiber} thus yields a model $\cX'$ of $X_F$ with a finite morphism $\mu:\cX'\to\cX_F$ such that $\cX'_s$ is reduced. By Lemma~\ref{lem:func}, we have $\phi_F=\phi_{\cL_F}=\phi_{\mu^\star \cL_F}$, and hence $\n_{\Hnot(m\mu^\star \cL_F)}=\n_{m\phi_F}$, by Lemma~\ref{lem:supmodel}. All in all, we get
$$
\n_{\Hnot(m\mu^\star \cL_F)}=\n_{m\phi_F}\le\left(\n_{m\phi}\right)_F\le(\n_{\Hnot(m\cL)})_F=\n_{\Hnot(m\cL_F)}, 
$$
and it will thus be enough to show that $\n_{\Hnot(m\cL_F)}\le C\n_{\Hnot(m\mu^\star \cL_F)}$ for a uniform constant $C>0$. By the projection formula, we have an injection of $K^\circ$-modules
$$
\Hnot\left(\cX',m\mu^\star \cL_F\right)/\Hnot(m\cL_F)\hookrightarrow \Hnot\left(\cX_F,\cF(m\cL_F)\right),
$$
where $\cF:=\left(\mu_\star \cO_{\cX'}\right)/\cO_{\cX_F}$ is a coherent module on $\cX_F$ supported in the special fiber (cf.~Theorem~\ref{thm:kiehl} for coherence), and hence $a$-torsion for some nonzero $a\in K'^{\circ\circ}$. We thus have $a \Hnot(m\mu^\star \cL_F)\subset \Hnot(m\cL_F)$, and hence $\n_{\Hnot(m\cL_F)}\le |a|^{-1}\n_{\Hnot(m\mu^\star \cL_F)}$, which yields the desired result. 
\end{proof}

For later use, we also provide the following technical generalization of Theorem~\ref{thm:supmodel}. A bounded metric on a line bundle $L$ over $X$ induces for each $r\ge 1$ a bounded metric $\phi^{\boxtimes r}$ on the external tensor product $L^{\boxtimes r}$ over $X^r:=X\times_K\dots\times_K X$ ($r$ times). If $\phi$ is the model metric determined by a line bundle $\cL$ over $\cX$, then $\phi^{\boxtimes r}$ is the model metric determined by $\cL^{\boxtimes r}$ over $\cX^r:=\cX\times_{K^\circ}\dots\times_{K^\circ}\cX$. 

\begin{thm}\label{thm:supmodelbis} Assume that $X$ is geometrically reduced, and let $\phi_i$ be a finite family of model metrics on line bundles $L_i$, with $\phi_i$ determined by a line bundle $\cL_i$ on a given model $\cX$ of $X$. We can then find a constant $C \geq 1$ such that for all integers $r,m_i\ge 1$, we have 
$$
\n_{(\sum_im_i\phi_i)^{\boxtimes r}}\le\n_{\Hnot\left(\cX^r,(\sum_i m_i\cL_i)^{\boxtimes r}\right)}\le C^r\n_{(\sum_i m_i\phi_i)^{\boxtimes r}}
$$
as norms on $\Hnot(X^r,(\sum_i m_i L_i)^{\boxtimes r})$. 
\end{thm}

Note that $X^r$ is reduced for each $r$, since $X$ is geometrically reduced, so that $\n_{(\sum_i m_i\phi_i)^{\boxtimes r}}$ is indeed a norm. 

\begin{proof} Use the notation in the proof of Theorem~\ref{thm:supmodel}. Since $F$ is algebraically closed, so is its residue field $\tF$; the $\tF$-scheme $\cX'_s$ is thus geometrically reduced, and 
$$
(\cX'^r)_s=(\cX'\times_{F^\circ}\dots\times_{F^\circ}\cX')_s=\cX'_s\times_{\tF}\dots\times_{\tF}\cX'_s
$$
is therefore reduced. As above, we have 
$$
\n_{\Hnot\left(\cX'^r,(\mu^r)^\star (\sum_i m_i\cL_{i,F})^{\boxtimes r}\right)}=\n_{(\sum_i m_i\phi_{i,F})^{\boxtimes r}}\le\left(\n_{(\sum_i m_i\phi_i)^{\boxtimes r}}\right)_F
$$
$$
\le\left(\n_{\Hnot\left(\cX^r,(\sum_i m_i\cL_i)^{\boxtimes r}\right)}\right)_F=\n_{\Hnot\left(\cX_F^r,(\sum_i m_i\cL_F)^{\boxtimes r}\right)}, 
$$
and 
$$
\Hnot\left(\cX'^r,(\mu^r)^\star (\sum_i m_i\cL_F)^{\boxtimes r}\right)/\Hnot\left(\cX_F^r,(\sum_i m_i\cL_{i,F})^{\boxtimes r}\right)\hookrightarrow \Hnot\left(\cX_F^r,\cF_r(\sum_i m_i\cL_{i,F})^{\boxtimes r}\right)
$$
with $\cF_r:=\left((\mu^r)_\star \cO_{\cX'^r}\right)/\cO_{\cX_F^r}$. It will thus be enough to show that $\cF_r$ is $a^r$-torsion, where $a\in F^{\circ\circ}$ annihilates $\cF=\left(\mu_\star \cO_{\cX'}\right)/\cO_{\cX_F}$ as above. Let $\cU=\spec(\cA)$ be an affine open subset of $\cX_F$. Since $\mu$ is finite, $\mu^{-1}(\cU)=\spec(\cA')$ is also affine. Since $a\cA'\subset\cA$, we get $a^r\cA'\otimes_{F^\circ}\dots\otimes_{F^\circ}\cA'\subset\cA\otimes_{F^\circ}\dots\otimes_{F^\circ}\cA'$, which implies as desired $a^r(\mu^r)_\star \cO_{\cX'^r}\subset\cO_{\cX_F^r}$.  
\end{proof} 
%
\subsection{Supnorms and ground field extension}
Even though supnorms are generally not compatible with ground field extension (see Lemma~\ref{lem:supgroundgalois}), the following result shows that it becomes asymptotically true for large powers of a line bundle (thanks to Walter Gubler for his help with the argument). 

\begin{thm}\label{thm:supground} Assume that $X$ is geometrically reduced. Let $F/K$ be an arbitrary complete field extension, $\phi$ a continuous metric on $L$, and denote by $\phi_F$ its pull-back to $L_F$. Then 
$$
\dGI\left(\n_{m\phi_F},(\n_{m\phi})_F\right)=o(m).
$$
If $K$ is Archimedean, or if $\phi$ is a model metric, then $\dGI\left(\n_{m\phi_F},(\n_{m\phi})_F\right)=O(1)$. 
\end{thm}
\begin{proof} The Archimedean case follows directly from Lemma~\ref{lem:supground}. From now on, $K$ is non-Archimedean. Assume that $\phi$ is a model metric, and pick a model $(\cX,\cL)$ of $(X,aL)$,  $a\ge 1$, such that $\phi=a^{-1}\phi_\cL$. After passing to a higher model, we can also assume that $L$ has a model $\cM$ on $\cX$ (Lemma~\ref{lem:modelexist}). Pick $m\ge 1$ and write $m=qa+r$ with $q\in\N$ and $0\le r<a$. Denote by $\cX_F$, $\cL_F$ and $\cM_F$ the base change of $\cX$, $\cL$ and $\cM$ to $F^{\circ}$. Since $m\phi=\phi_{q\cL}+r\phi$ and $\phi-\phi_\cM$ is bounded, 
$$
\dGI\left(\n_{m\phi_F},\n_{\phi_{q\cL_F+r\cM_F}}\right)=O(1)\text{  and  }\dGI\left((\n_{m\phi})_F,(\n_{\phi_{q\cL+r\cM}})_F\right)=O(1),
$$
using again Proposition~\ref{prop:ground}. By Theorem~\ref{thm:supmodel}, 
$$
\dGI\left(\n_{\phi_{q\cL_F+r\cM_F}},\n_{\Hnot(q\cL_F+r\cM_F)}\right)=O(1)\text{ and }\dGI\left((\n_{\phi_{q\cL+r\cM}})_F,(\n_{\Hnot(q\cL+r\cM)})_F\right)=O(1). 
$$
By Lemma~\ref{lem:latticenorm}, $\n_{\Hnot(q\cL_F+r\cM_F)}=(\n_{\Hnot(q\cL+r\cM)})_F$, and it follows that 
$$
\dGI((\n_{m\phi})_F,\n_{m\phi_F})\le \dGI\left((\n_{m\phi})_F,(\n_{\phi_{q\cL+r\cM}})_F\right)+\dGI\left((\n_{\phi_{q\cL+r\cM}})_F,(\n_{\Hnot(q\cL+r\cM)})_F\right)
$$
$$
+\dGI\left(\n_{\Hnot(q\cL_F+r\cM_F)},\n_{\phi_{q\cL_F+r\cM_F}}\right)+\dGI\left(\n_{\phi_{q\cL_F+r\cM_F}},\n_{m\phi_F}\right),
$$
which is thus bounded. 

Let now $\phi$ be an arbitrary continuous metric. Assume first that $K$ is nontrivially valued, and pick $\e>0$. By Theorem~\ref{thm:PLdense}, we can find a model metric $\p$ on $L$ such that $\sup|\phi-\p|\le\e$. 
By invariance of $\dGI$ under ground field extension (Proposition~\ref{prop:ground}), we infer 
$$
\dGI\left(\n_{m\phi_F},(\n_{m\phi})_F\right)\le\dGI\left(\n_{m\phi_F},\n_{m\p_F}\right)+\dGI\left(\n_{m\p_F},(\n_{m\p})_F\right)
$$
$$+\dGI\left((\n_{m\p})_F,(\n_{m\phi})_F\right)\le C+2m\e. 
$$
Thus $\limsup_{m\to\infty} m^{-1}\dGI((\n_{m\phi})_F,\n_{m\phi_F})\le\e$ for all $\e>0$, and hence 
$$
\dGI((\n_{m\phi})_F,\n_{m\phi_F})=o(m).
$$
Assume next that $K$ is trivially valued, and that $F$ is nontrivially valued. By Lemma~\ref{lem:CM}, there exists a nontrivially valued extension $K'/K$ such that for each $m\in\N$, the ground field extension $(\n_{m\phi})_{K'}$ to $\Hnot(mL)_{K'}$ is the unique norm that coincides with $\n_{m\phi}$ on $\Hnot(mL)$, and hence $(\n_{m\phi})_{K'}=\n_{m\phi_{K'}}$, by Lemma~\ref{lem:supground}. 

Choose a common non-Archimedean extension $F'$ of both $K'$ and $F$, and note that $(\n_{m\phi})_{F'}=(\n_{m\phi_{K'}})_{F'}$. By Proposition~\ref{prop:ground} and the triangle inequality, we get 
$$
\dGI\left((\n_{m\phi})_F,\n_{m\phi_F}\right)=\dGI\left((\n_{m\phi})_{F'},(\n_{m\phi_F})_{F'}\right)
$$
$$
=\dGI\left((\n_{m\phi_{K'}})_{F'},(\n_{m\phi_F})_{F'}\right)
$$
$$
\le\dGI\left(\n_{m\phi_{K'}})_{F'},\n_{m\phi_{F'}}\right)+\dGI\left(\n_{m\phi_{F'}},(\n_{m\phi_F})_{F'}\right),
$$
which is $o(m)$ by the first part of the proof applied to the nontrivially valued fields $F$ and $K'$.

Assume finally that both $K$ and $F$ are trivially valued, and chose a nontrivially valued extension $F'$ of $F$. We similarly get
$$
\dGI\left(\n_{m\phi_F},(\n_{m\phi})_F\right)=\dGI\left((\n_{m\phi_F})_{F'},(\n_{m\phi})_{F'}\right)
$$
$$
\le \dGI\left((\n_{m\phi_F})_{F'},\n_{m\phi_{F'}}\right)+\dGI\left(\n_{m\phi_{F'}},(\n_{m\phi})_{F'}\right)
$$
Since $F'$ is nontrivially valued, the previous step shows that the last two terms are $o(m)$, and we are done. 
\end{proof}

%
%
\section{Plurisubharmonic metrics and envelopes} 
Following~\cite{siminag,trivval}, we introduce a general notion of plurisubharmonic metric, and compare it with other existing notions. In what follows, $X$ denotes a projective scheme over a complete valued field $K$. 
 %
%
\subsection{Plurisubharmonic metrics}\label{sec:limFS}
Assume first that $K$ is Archimedean (the case $K=\R$ merely consisting in working with conjugation-invariant objects). A (possibly singular) metric $\phi$ on a line bundle $L$ is \emph{plurisubharmonic} (\emph{psh} for short) if, for each local trivialization $\tau$ of $L$, the function $-\log|\tau|_\phi$ is psh, \ie locally the restriction of a psh function in a polydisc. Note that $\phi$ is psh iff its pullback to $L_\redu$ is psh. For any $m\ge 1$, $\phi$ is psh iff $m\phi$ is psh, and psh metrics thus make sense on $\Q$-line bundles. 

\begin{thm}\label{thm:dem} Assume that $K$ is Archimedean, and let $L$ be a semiample $\Q$-line bundle. Then every psh metric $\phi$ on $L$ is the limit of a decreasing sequence $(\phi_j)$ of Fubini--Study metrics on $L$. 
\end{thm}
\begin{proof} When $X$ is smooth and $L$ is ample, the result is due to Demailly, and is based on the deep Ohsawa--Takegoshi $L^2$-extension theorem (see for instance~\cite[Theorem 8.1]{GZ} and its proof). In the general case, le $A$ be the Stein factoration of $(X,L)$ as in Lemma~\ref{lem:Stein}, \ie $A$ is an ample $\Q$-line bundle on a projective scheme $Y$, and $L=f^\star A$ with $f:X\to Y$ surjective and such that $f_\star\cO_X=\cO_Y$. By Zariski's main theorem, $f:X(\C)\to Y(\C)$ has connected fibers, and $\phi$ thus descends to a (singular) usc metric on $A$. Relying on the deep Fornaess--Narasimhan theorem, one checks as in the final part of the proof of \cite[Th\'eor\`eme 1.7]{DemSMF} that the induced metric on $A$ is psh. As a result, we may assume wlog that $L$ was ample to begin with. Passing to a multiple, we can even assume that $L$ is a very ample line bundle, and hence $L=\cO(1)|_X$ for an embedding of $X$ in a projective space $\P$. By~\cite[Theorem B']{CGZ}, every psh metric $\phi$ on $L$ is then the restriction of a psh metric $\p$ on $\cO(1)$. Since $\P$ is smooth, Demailly's result implies that $\p$ is the decreasing limit of a sequence $\p_j$ of Fubini--Study metrics on $\cO(1)$. The restriction of each $\p_j$ to $X$ is then a Fubini--Study metric on $L$ (Proposition~\ref{prop:FS}~(iv)), and the result follows. 
\end{proof}

In view of Theorem~\ref{thm:dem}, it is natural to introduce as in~\cite{trivval}: 

\begin{defi}\label{defi:psh} Assume that $K$ is non-Archimedean, and let $L$ be a semiample $\Q$-line bundle. We say that a (possibly singular) metric $\phi$ on $L$ is \emph{plurisubharmonic} (\emph{psh} for short) if $\phi$ can be written as the pointwise limit of a decreasing net $\phi_i\in\FS(L)$ of Fubini--Study metrics on $L$. We denote by $\PSH(L)$ the set of psh metrics on $L$. 
\end{defi}
Given a subgroup $\Ga\subset\R$, nontrivial if $K$ is trivially valued, $\FS_\Ga(L)$ is dense in $\FS(L)$ with respect to uniform convergence, by Proposition~\ref{prop:FS}~(v); it would thus be equivalent to require $\phi_i\in\FS_\Ga(L)$ in the above definition. 

As a direct consequence of Corollary~\ref{cor:Stein}, we have:
\begin{lem} Let $A$ be the Stein factorization of $L$. Then 
$$
\PSH(A)\simeq\PSH(L)\simeq\PSH(L_\redu).
$$
\end{lem}

The following general result, borrowed from~\cite[Proposition 5.6]{trivval}, ensures that $\PSH(L)$ is closed under decreasing limits. 

\begin{lem}\label{lem:deacr} Let $\cF\subset  \cz(L)$ be a family of continuous metrics on a $\Q$-line bundle $L$, closed under addition of a constant, and define $\widetilde{\cF}$ as the set of possibly singular metrics on $L$ that can be written as decreasing limits of nets in $\cF$. Then $\widetilde{\cF}$ is closed under decreasing limits. 
\end{lem}
\begin{proof} Let $(\phi_i)_{i\in I}$ be a decreasing net in $\widetilde{\cF}$. For each $i\in I$ choose a decreasing net $(\phi_{i,j})_{j\in J_i}$ in $\cF$ converging to $\phi_i$. Let $A$ be the set of triples $\a=(i,j,m)$ with $i\in I$, $j\in J_i$ and $m\in\Z_{>0}$. For each $\a=(i,j,m)$ in $A$, set 
$$
\p_\a:=\phi_{i,j}+m^{-1}\in\cF,
$$
and define a partial order on $A$ by 
$$
\a\ge \a'\Longleftrightarrow\p_\a\le\p_{\a'}.
$$
We claim that $A$ is directed, \ie any two elements $ a_1=(i_1,j_1,m_1)$, $ a_2=(i_2,j_2,m_2)$ in $A$ can be dominated by a third. To see this, first pick $i\ge i_1,i_2$, so that $\phi_i\le\phi_{i_1}$ and $\phi_i\le\phi_{i_2}$, and set $m:=\max\{m_1,m_2\}+1$. A simple variant of Dini's lemma shows that $\phi_{i,j}+m^{-1}\le\phi_{i_k,j_k}+m_k^{-1}=\p_{ a_k}$ on $X^\an$ for $k=1,2$ and all $j\in J_i$ large enough, which proves the claim. By construction, $(\p_\a)_{\a\in A}$ is a decreasing net in $\cF$ such that $\p_\a\to\phi$ pointwise on $X^\an$, and hence $\phi\in\widetilde{\cF}$. 
\end{proof}

\begin{cor}\label{cor:psh} The space $\PSH(L)$ is stable under finite max and decreasing limits. 
\end{cor}

\begin{cor} \label{cor:contpsh} Pick a subgroup $\Ga\subset\R$, nontrivial if $K$ is trivially valued. For any semiample $\Q$-line bundle $L$, the set $\cz\cap\PSH(L)$ of continuous psh metrics on $L$ is the closure of $\FS_\Ga(L)$ with respect to uniform convergence. 
\end{cor}
\begin{proof} Let $\phi$ be a continuous psh metric, and pick a decreasing net $(\phi_i)$ in $\FS_\Ga(L)$ converging pointwise to $\phi$. By Dini's lemma, $\phi_i\to\phi$ uniformly on $X^\an$, and $\phi$ is thus in the closure of $\FS_\Ga(L)$. Conversely, let $(\phi_i)$ be a sequence in $\FS_\Ga(L)$ converging unifornly to some (continuous) metric $\phi$. We can inductively construct a sequence of constants $c_i\to 0$ such that $\phi_i+c_i$ is decreasing, and it follows that $\phi=\lim_i(\phi_i+c_i)$ is psh.
\end{proof}

%
%
\subsection{Nef metrics}
Assume that $K$ is non-Archimedean. A model metric $\phi$ on a $\Q$-line bundle $L$ is \emph{nef} if $\phi=\phi_\cL$ for some $\Q$-model $\cL$ of $L$ which is \emph{nef}, \ie $\cL\cdot C\ge 0$ for each $\tK$-projective curve $C\subset\cX_s$. By~\cite[Proposition 3.5]{GM}, any $\Q$-model $\cL'$ of $L$ such that $\phi=\phi_{\cL'}$ is then also nef. 

\begin{defi}\label{defi:nefZhang} A \emph{nef metric} on a $\Q$-line bundle $L$ is a continuous metric $\phi$ that can be written as a uniform limit of nef model metrics. 
\end{defi} 
This is usually known as a semipositive metric in the sense of S.W.~Zhang, cf.~\cite{GM} for a thorough discussion. When $K$ is trivially valued, the trivial metric is the only nef metric. 

\begin{thm}\label{thm:pshZhang} Assume that $K$ is nontrivially valued, and let $L$ be an ample $\Q$-line bundle. Then a continuous metric on $L$ is nef iff it is psh (in the sense of Definition~\ref{defi:psh}). 
\end{thm}
A proof of this result is also sketched, in a slightly different language, in~\cite[Remark 3.18]{CM}; we nevertheless provide a proof below, for completeness. By~\cite[Corollary 1.4]{GM}, we get:

\begin{cor}\label{cor:pshZhang} Let $L$ be an ample $\Q$-line bundle, and pick a $\Q$-model $\cL$ of $L$. Then $\phi_\cL$ is psh iff $\cL$ is nef. 
\end{cor}

\begin{proof}[Proof of Theorem~\ref{thm:pshZhang}] By Theorem~\ref{thm:puremodel}, pure Fubini--Study metrics on $L$ coincide with model metrics associated to $\Q$-models $\cL$ of $L$ that are semiample, and hence nef. Since $K$ is nontrivially valued, Corollary~\ref{cor:contpsh} implies that every continuous psh metric is a uniform limit of pure Fubini--Study metrics, and hence is nef. 

To prove the converse, let $\cL$ be a nef $\Q$-model of $L$, determined on a model $\cX$ of $X$. Using again Theorem~\ref{thm:puremodel}, it will be enough to show that $\phi_\cL$ is a uniform limit of model metrics associated to semiample $\Q$-models. Since $L$ is ample, it admits an ample $\Q$-model $\cA$ on some projective model $\cX'$ dominating $\cX$ (cf.~\cite[Lemma 4.12]{GM}). The pull-back $\cL'$ of $\cL$ to $\cX'$ is nef. For each $\e \in (0,1) \cap \Q$, $\cL_\e:=(1-\e)\cL'+\e\cH$ is thus ample on $\cX_s$, and hence on $\cX$, by~\cite[IV.9.6.4]{EGA}. In particular, $\cL_\e$ is a semiample model of $L$. 
Since $\phi_\cL=\phi_{\cL'}$ (Lemma~\ref{lem:func}), $\phi_{\cL_\e}-\phi_\cL=\e(\phi_\cH-\phi_{\cL})$ converges to $0$ uniformly as $\e>0$, and we are done. 
\end{proof}

\begin{rmk} We expect Theorem~\ref{thm:pshZhang} and Corollary~\ref{cor:pshZhang} to hold when $L$ is merely semiample. For this, it would be enough to show that any nef model metric on $L$ descends to the Stein factorization $A$ as in Lemma~\ref{lem:Stein}; this would for instance follow from a relative version of the local Hodge index theorem proved in~\cite[Theorem 2.1]{YZ}. 
\end{rmk}

 %
%
\subsection{Psh-regularizable metrics}\label{sec:pshreg}
The \emph{curvature form} of a smooth metric $\phi$ on a line bundle $L$ is a closed $(1,1)$-form, which we denote by $dd^c\phi$ (see~\cite[\S 6.4.1]{CLD} for the non-Archimedean case). For each trivializing section $\tau$ of $L$ over an open $U\subset X$, the function $\log|\tau|_\phi$ is smooth, and $dd^c\phi=-dd^c\log|\tau|_\phi$ on $U^\an$. Again, this makes sense for $\Q$-line bundles as well. 

\begin{defi} A smooth metric $\phi$ on a $\Q$-line bundle $L$ is \emph{semipositive} if its curvature form $dd^c\phi$ is semipositive (see~\cite[\S 5.4]{CLD} for the non-Archimedean case).
\end{defi}

\begin{exam}\label{exam:smooth} Let $L$ be a line bundle, $s_1,\dots s_r\in \Hnot(L)$ be sections without common zeroes, $\la_1,\ldots,\la_r\in\R$, and $\max_\e:\R^r\to\R$ a regularized max function. Then
$$
\phi:=\mathrm{max}_\e\left\{\log|s_1|+\la_1,\ldots,\log|s_r|+\la_r\right\}
$$
is a smooth semipositive metric on $L$, see~\cite[6.3.2]{CLD}. 
\end{exam}

\begin{defi} A (possibly singular) metric $\phi$ on a line bundle $L$ is \emph{psh-regularizable} if $\phi$ can be written as the limit of a decreasing net $\phi_i$ of smooth semipositive metrics. 
\end{defi}
If $\phi$ is continuous then $\phi_i\to\phi$ uniformly, by Dini's lemma. 

\begin{thm}\label{thm:pshCLD} Every psh metric $\phi$ on a semiample $\Q$-line bundle $L$ is psh-regularizable. 
\end{thm}
\begin{proof} By Lemma~\ref{lem:deacr}, it is enough to prove the result when $\phi$ is a Fubini--Study metric. After replacing $L$ with a multiple, there exists a basis $(s_1,\ldots,s_N)$ of $\Hnot(L)$ without common zeroes and constants $\la_i\in\R$ such that $\phi=\max\left\{\log|s_1|+\la_1\dots,\log|s_N|+\la_N\right\}$, and $\phi$ is thus the uniform limit as $\e\to 0$ of the smooth semipositive metrics 
$$
\phi_\e=\mathrm{max}_\e\left\{\log|s_1|+\la_1\dots,\log|s_N|+\la_N\right\}
$$
as in Example~\ref{exam:smooth}. 
\end{proof}
%
%
\subsection{The Fubini--Study operator}\label{sec:projop}
The following discussion is closely related to the point of view developed by Chen and Moriwaki in~\cite[\S 3.2]{CM}. 

\begin{defi}\label{defi:FSmetr} Let $L$ be a globally generated line bundle. We say that a subspace $V\subset \Hnot(L)$ is \emph{basepoint free} if sections in $V$ generate $L$ at all points. To each norm $\n$ on $V$, we then associate the metric $\FS(\n)$ on $L$ defined at each $x\in X^\an$ as the quotient norm
$$
\FS(\n)(x):=\log\sup_{s\in V\setminus\{0\}}\frac{|s(x)|}{\|s\|}. 
$$
\end{defi}
Recall that $\log|s|$ denotes the singular metric attached to $s\in \Hnot(L)$, cf.~ Example~\ref{exam:singsec}. 

\begin{thm}\label{thm:FScont} For any norm $\n$ on a basepoint free subspace $V\subset \Hnot(L)$, the metric $\FS(\n)$ is continuous and psh. 
\end{thm}

\begin{lem}\label{lem:FSdiag} Let $s_1,\ldots,s_r$ be finitely many sections of $L$ without common zeroes, and $\la_1,\ldots,\la_r\in\R$. Let $V\subset \Hnot(L)$ be the basepoint free subspace generated by the $s_i$. Let $\n_\la$ be the norm on $K^r$ which is diagonal in the canonical basis $(e_i)$ and such that $\|e_i\|_\la=e^{-\la_i}$, and denote by $\n$ the induced quotient norm on $V$ with respect to the surjective map $K^r\to V$ that sends $(e_i)$ to $(s_i)$. Then: 
\begin{itemize}
\item[(i)] $\FS(\n)=\tfrac 12\log\sum_i |s_i|^2 e^{2\la_i}$ when $K$ is Archimedean; 
\item[(ii)] $\FS(\n)=\max_i(\log|s_i|+\la_i)$ when $K$ is non-Archimedean. 
\end{itemize}
\end{lem}
\begin{proof} For each $s\in V$, we have $\|s\|=\inf\left\{\|a\|_\la\mid a\in K^r,\,s=\sum_i a_i s_i\right\}$. Using this, one immediately checks that
$$
\FS(\n)=\log\sup_{a\in K^r\setminus\{0\}}\frac{|\sum_i a_i s_i|}{\|a\|_\la}
$$
pointwise on $X^\an$. In the Archimedean case, the Cauchy--Schwarz inequality yields
$$
\sup_{a\in K^r\setminus\{0\}}\frac{|\sum_i a_i s_i|^2}{\|a\|_\la^2}=\sum_i\frac{|s_i|^2}{\|e_i\|_\la^2}=\sum_i |s_i|^2 e^{2\la_i}. 
$$
In the non-Archimedean case, Lemma~\ref{lem:supdiag} yields
$$
\sup_{a\in K^r\setminus\{0\}}\frac{|\sum_i a_i s_i|}{\|a\|_\la}=\max_i\frac{|s_i|}{\|e_i\|_\la}=\max_i |s_i| e^{\la_i}. 
$$
The result follows. 
\end{proof}

\begin{cor}\label{cor:FSdiag} For a continuous metric $\phi$ on a $\Q$-line bundle $L$, the following are equivalent: 
\begin{itemize}
\item[(i)] $\phi$ is a Fubini--Study metric; 
\item[(ii)] for $m\in\Z_{>0}$ divisible enough, there exists a basepoint free subspace $V\subset \Hnot(mL)$ and a diagonalizable norm $\n$ on $V$ such that $\phi=m^{-1}\FS(\n)$. 
\end{itemize}
\end{cor}

\begin{proof} That (ii)$\Longrightarrow$(i) follows directly from Lemma~\ref{lem:FSdiag}. Assume conversely that $\phi$ is Fubini--Study. For $m$ divisible enough we can find sections $s_1,\ldots,s_r$ of $mL$ without common zeroes such that 
$\phi=\tfrac{1}{2m}\log\sum_i |s_i|^2$ in the Archimedean case, and $\phi=\tfrac{1}{m}\max_i(\log|s_i|+\la_i)
$ with $\la_i\in\R$ in the non-Archimedean case. By Lemma~\ref{lem:FSdiag}, $\phi=m^{-1}\FS(\n)$ with $\n$ the quotient norm on $V=\Vect(s_i)$ of some diagonalizable norm. By Lemma~\ref{lem:quotientdiag}, $\n$ is diagonalizable, and we are done. 
\end{proof}

\begin{exam} Assume that $K$ is non-Archimedean, and let $\cL$ be a semiample $\Q$-model of a (semiample) $\Q$-line bundle $L$. For $m$ divisible enough, $m\cL$ is a globally generated line bundle, and the lattice norm $\n_{\Hnot(m\cL)}$ on $\Hnot(mL)$ satisfies
$$
\phi_\cL=m^{-1}\FS\left(\n_{\Hnot(m\cL)}\right)
$$ 
This follows indeed from Theorem~\ref{thm:puremodel} (and its proof). 
\end{exam}

\begin{proof}[Proof of Theorem~\ref{thm:FScont}] Let $V\subset \Hnot(L)$ be a basepoint free subspace, and note first that for any two norms $\n,\n'$ on $V$, we trivially have
\begin{equation}\label{equ:FSlip}
\FS(\n)-\dGI(\n,\n')\le\FS(\n')\le\FS(\n)+\dGI(\n,\n'). 
\end{equation}
When $K$ is non-Archimedean, Theorem~\ref{thm:diag} shows that, for each norm $\n$ on $V$, there exists a sequence of diagonalizable norms $\n_i$ such that $\dGI(\n_i,\n)\to 0$. By Lemma~\ref{lem:FSdiag}, $\FS(\n_i)$ is a Fubini--Study metric for all $i$, and $\FS(\n_i)$ converges uniformly to $\FS(\n)$, by~\eqref{equ:FSlip}. This proves that $\FS(\n)$ is continuous and psh. 

In the Archimedean case, the unit ball $B\subset V$ of $\n$ is compact, and 
$$
\FS(\n)=\sup_{s\in B}\log|s|
$$
pointwise. Using this, it is easy to see that $\FS(\n)$ is continuous and psh. 
\end{proof}

\begin{lem}\label{lem:pshbase} Let $\n$ be a norm on a basepoint free subspace $V\subset \Hnot(L)$. Let $F/K$ be a complete field extension, and denote by $\n_F$ the ground field extension of $\n$ to $V_F\subset \Hnot(L)_F=\Hnot(L)_F$. Then $\FS(\n_F)$ is the pullback to $L_F$ of $\FS(\n)$. 
\end{lem}
\begin{proof} In the Archimedean case, the only nontrivial case is $K=\R$, $F=\C$. Denote by $B\subset V$ and $B_\C\subset V_\C$ the unit balls of $\n$, $\n_\C$. We need to show that $\sup_{s\in B}\log|s|=\sup_{s\in B_\C}\log|s|$, which is an easy consequence of Proposition~\ref{prop:ground}. Assume now that $K$ is non-Archimedean. As above, it is enough to prove the result when $\n$ is diagonalizable. Let $(s_i)$ be an orthogonal basis of $\Hnot(L)$. The pullback $(s'_i)$ of $(s_i)$ to $\Hnot(L_F)$ is orthogonal for $\n_F$ with $\|s'_i\|_F=\|s_i\|$,  and Lemma~\ref{lem:FSdiag} thus yields
$$
\FS(\n)=\max_i\{\log|s_i|-\log\|s_i\|\},\,\,\,\,\FS(\n_F)=\max_i\{\log|s'_i|-\log\|s'_i\|_F\},
$$
which proves the result. 
\end{proof}

\begin{defi} Let $L$ be a globally generated line bundle, and denote by $L_\redu$ its restriction to $X_\redu$. For each bounded metric $\phi$, we set
$\FS(\phi):=\FS(\n_\phi)$, where $\n_\phi$ is the supnorm induced by $\phi$ on $\Hnot(L_\redu)$. 
\end{defi}

\begin{thm}\label{thm:projFS} If $\phi$ is a Fubini--Study metric on a semiample $\Q$-line bundle $L$, then $\phi=m^{-1}\FS(m\phi)$ for all $m$ sufficiently divisible.
\end{thm}

\begin{lem}\label{lem:projFS} For each bounded metric $\phi$ on a globally generated line bundle $L$, we have
\begin{itemize}
\item[(i)] $\FS(\phi)\le\phi$; 
\item[(ii)] if $\phi=\FS(\n)$ for some norm $\n$ on a basepoint free subspace $V\subset \Hnot(L)$, then $\FS(\phi)=\phi$; 
\item[(iii)] $\n_{\FS(\phi)}=\n_{\phi}$ as norms on $\Hnot(L_\redu)$; 
\item[(iv)] for any other bounded metric $\phi'$ on $L$ we have 
\begin{equation}\label{equ:FSphilip}
\sup\left|\FS(\phi)-\FS(\phi')\right|\le\sup|\phi-\phi'|. 
\end{equation}
\end{itemize}
\end{lem}
\begin{proof} For any nonzero $s\in \Hnot(L_\redu)$, the definitional inequality $|s|_\phi\le\|s\|_\phi$ yields $\log|s|-\phi\le\log\|s\|_\phi$, and 
$$
\FS(\phi)=\sup_{s\in \Hnot(L_\redu)\setminus\{0\}}\left(\log|s|-\log\|s\|_\phi\right)\le\phi,
$$
proving (i). Assume $\phi=\FS(\n)$ for a norm $\n$ on a basepoint free subspace $V\subset \Hnot(L)$. For each $s\in V$, we have $\log|s|-\log\|s\|\le\phi$, \ie $\|s\|_\phi\le\|s\|$. Thus
$$
\FS(\n_\phi)\ge\sup_{s\in V\setminus\{0\}}\left(\log|s|-\log\|s\|\right)=\phi,
$$
and hence (ii). Since $\FS(\phi)\le\phi$, each nonzero $s\in \Hnot(L_\redu)$ satisfies $\|s\|_\phi\le\|s\|_{\FS(\phi)}$. On the other hand, $\log|s|-\log\|s\|_\phi\le\FS(\phi)$, hence $|s|_{\FS(\phi)}\le\|s\|_\phi$, which proves (iii). 

Finally, we trivially have $\phi\le\phi'\Longrightarrow\FS(\phi)\le\FS(\phi')$ and $\FS(\phi+c)=\FS(\phi)+c$ for $c\in\R$, which formally imply~\eqref{equ:FSphilip}.  
\end{proof}

\begin{proof}[Proof of Theorem~\ref{thm:projFS}] By Corollary~\ref{cor:FSdiag}, for all $m$ sufficiently divisible we have $\phi=m^{-1}\FS(\n)$ with $\n$ a norm on a basepoint free subspace $V\subset \Hnot(mL)$, and Lemma~\ref{lem:projFS}~(ii) thus yields $\phi=m^{-1}\FS(m\phi)$.  
\end{proof}

%
%
\subsection{Envelopes and ground field invariance} 
In what follows, $L$ is a semiample $\Q$-line bundle on $X$. 

\begin{defi}\label{defi:env} Let $\phi$ be a bounded metric on $L$. 
\begin{itemize}
\item[(i)] The \emph{psh envelope} of $\phi$ is defined as the pointwise upper envelope
\begin{equation}\label{equ:env}
\env(\phi):=\sup\{\p\in\PSH(L)\mid\p\le\phi\}. 
\end{equation}
\item[(ii)] The \emph{regular psh envelope} of $\phi$ is 
\begin{equation}\label{equ:Q}
\qq(\phi):=\sup\{\p\in \FS(L)\mid\p\le\phi\}.
\end{equation}
\end{itemize}
\end{defi}
Since any continuous psh metric on $L$ is a uniform limit of Fubini--Study metrics, it would be equivalent to require $\p$ continuous and psh in (ii). Trivially, $\qq(\phi)$ is lsc, $\qq(\phi)\le\env(\phi)\le\phi$, and the operators $\env$ and $\qq$ are monotone increasing and $1$-Lipschitz with respect to the supnorm. They are related as follows.  

\begin{prop}\label{prop:QP} For any bounded metric $\phi$ on $L$ we have $\qq(\phi)=\env(\phi_\star)$, with $\phi_\star\le\phi$ denoting the lsc regularization of $\phi$. 
\end{prop}
\begin{proof} A continuous metric $\p$ satisfies $\p\le\phi$ iff $\p\le\phi_\star$, and hence $\qq(\phi)=\qq(\phi_\star)$. We may thus assume from the start that $\phi$ is lsc, and we then need to show that any psh metric $\p$ on $L$ such that $\p\le\phi$ satisfies $\p\le\qq(\phi)$. By Definition~\ref{defi:psh}, there exists a decreasing net $(\p_j)$ of Fubini--Study metrics on $L$ with $\p_j\to\phi$. Pick $\e>0$. By lower semicontinuity of $\phi-\p_j$ and a simple compactness argument (Dini's lemma), we thus have $\p_j\le\phi+\e$ on $X^\an$ for $j$ large enough. Thus $\p_j\le \qq(\phi)+\e$, and hence $\p\le \qq(\phi)+\e$, which yields the result.
\end{proof}

To $L$ we attach the semigroup 
$$
\N(L)=\left\{m\in\N\mid mL\text{ is a globally generated line bundle}\right\}.
$$

\begin{thm}\label{thm:FSQ} For any bounded metric $\phi$ on $L$ we have
\begin{equation}\label{equ:FSQ}
\lim_{\N(L)\ni m,\,m\to\infty} m^{-1}\FS(m\phi)=\qq(\phi),
\end{equation}
and the convergence is uniform iff $\qq(\phi)$ is continuous. 
\end{thm}

\begin{proof} Pick $m,m'\in\N(L)$. For any two $s\in \Hnot(mL_\redu)$, $s'\in \Hnot(m'L_\redu)$, we have 
\begin{equation}\label{equ:supgraded}
\|s\cdot s'\|_{(m+m')\phi}\le\|s\|_{m\phi}\|s'\|_{m'\phi}.
\end{equation}
This implies $\FS((m+m')\phi)\ge\FS(m\phi)+\FS(m'\phi)$, and hence 
$$
\lim_{\N(L)\ni m,\,m\to\infty} m^{-1}\FS(m\phi)=\sup_{m\in\N(L)_{>0}}m^{-1}\FS(m\phi),
$$
by Fekete's lemma. For each $m\in\N(L)_{>0}$, $m^{-1}\FS(m\phi)$ is continuous and psh, by Theorem~\ref{thm:FScont}. By Lemma~\ref{lem:projFS}, $m^{-1}\FS(m\phi)\le\phi$, and hence $m^{-1}\FS(m\phi)\le\qq(\phi)$, which proves that
$$
\sup_{m\in\N(L)_{>0}}m^{-1}\FS(m\phi)\le\qq(\phi).
$$
Conversely, pick $\p\in\FS(L)$ such that $\p\le\phi$. By Theorem~\ref{thm:projFS}, we find $m\in\N(L)_{>0}$ such that $\p=m^{-1}\FS(m\p)$. Thus $\p\le\sup_{m\in\N(L)_{>0}}m^{-1}\FS(m\phi)$, and hence 
$$
\qq(\phi)\le\sup_{m\in\N(L)_{>0}}m^{-1}\FS(m\phi).
$$
\end{proof}

\begin{cor}\label{cor:QP} Let $\phi$ be a bounded metric on $L$, and pick $m\in\N$ such $mL$ is a line bundle. Then 
$$
\n_\phi=\n_{\qq(\phi)}=\n_{\env(\phi)}
$$
as norms on $\Hnot(mL_\redu)$.  
\end{cor}
\begin{proof} For each $s\in \Hnot(mL)$ and $k\in\N$, we have $\|s^k\|_{km\phi}=\|s\|_{m\phi}^k$. After replacing $m$ with a multiple, we may thus assume that $mL$ is globally generated. Then 
$$
m^{-1}\FS(m\phi)\le \qq(\phi)\le\env(\phi)\le\phi,
$$ 
and we are done since $\n_{\FS(m\phi)}=\n_{m\phi}$ on $\Hnot(mL_\redu)$, by Lemma~\ref{lem:projFS}. 
\end{proof}

\begin{defi} Let $L$ be a semiample $\Q$-line bundle. We say that \emph{continuity of envelopes} holds for $L$ if for any continuous metric $\phi$ on $L$, $\env(\phi)=\qq(\phi)$ is continuous. 
\end{defi}
By Theorem~\ref{thm:FSQ} and Dini's lemma, $\env(\phi)$ is continuous iff $m^{-1}\FS(m\phi)$ converges uniformly as $m\to\infty$. 

\begin{lem}\label{lem:cont} Continuity of envelopes is equivalent to the following property: for any family $(\phi_\a)$ of psh metrics on $L$, uniformly bounded above, the usc upper envelope $\left(\sup_\a\phi_\a\right)^\star$ is psh.
\end{lem}
\begin{proof} Suppose first that continuity of envelopes holds, and let $\phi$ be a continuous metric on $L$. Since each $\FS(m\phi)$ is continuous, $\env(\phi)=\sup_m m^{-1}\FS(m\phi)$ is lsc, and its usc regularization $\env(\phi)^\star $ is the decreasing limit of a net of asymptotically Fubini--Study metrics $\p_j$. Since $\env(\phi)\le\phi$ and $\phi$ is continuous, $\lim_j\p_j=\env(\phi)^\star \le\phi$, and (a small variant of) Dini's lemma therefore yields $\p_j\le\phi+\e_j$ for some constants $\e_j\to 0$. It follows that $\p_j=\env(\p_j)\le \env(\phi)+\e_j$, and hence $\env(\phi)^\star \le \env(\phi)$ in the limit, which proves that $\env(\phi)$ is usc. 

Conversely, asssume that $\env(\phi)$ is continuous for each continuous metric $\phi$, and let $(\phi_\a)$ be a family of psh metrics, uniformly bounded above. The metric $\p:=\left(\sup_\a\phi_\a\right)^\star $, being usc, can be written as the limit of a decreasing net of continuous metrics $\tau_j$. For each $\a,j$, we have $\phi_\a\le\tau_j$, and hence $\phi_\a\le \env(\tau_j)$, which in turn yields $\p\le \env(\tau_j)\le\tau_j$. We have thus written $\p$ as the limit of the decreasing net of psh metrics $\env(\tau_j)$, which shows that $\p$ is psh. 
\end{proof}

In the Archimedean case, the equivalent formulation of Lemma~\ref{lem:cont} is a classical property, assuming $X$ to be normal. It is thus natural to conjecture: 

\begin{conj}\label{conj:env} Continuity of envelopes holds for any semiample $\Q$-line bundle $L$ over a \emph{normal} projective variety $X$. 
\end{conj}

As of this writing, continuity of envelopes in the non-Archimedean case has been established when $X$ is smooth and one of the following is satisfied:
\begin{itemize}
\item $X$ is a curve, as a consequence of A.~Thuillier's work~\cite{Thu} (see~\cite{GJKM});
\item $K$ discretely or trivially valued, of residue characteristic $0$~\cite{siminag,trivval}, building on
multiplier ideals and the Nadel vanishing theorem;
\item $K$ is discretely valued of characteristic $p$, $(X,L)$ is defined over a function field of
transcendence degree $d$, and resolution of singularities is assumed in dimension $n+d$~\cite{GJKM}, replacing multiplier ideals with test ideals.
\end{itemize}

We conclude this section with the following application of Theorem~\ref{thm:FSQ}. 

\begin{thm}\label{thm:pshbase} Assume that $X$ is geometrically reduced. Let $\phi$ be a continuous metric on a semiample $\Q$-line bundle $L$ over $X$. Pick a complete field extension $F/K$, and denote $\phi_F$ the pullback of $\phi$ to $L_F$. Then $\phi$ is psh iff $\phi_F$ is psh. 
\end{thm}

\begin{proof} By definition, the pullback of a Fubini--Study metric on $L$ is a Fubini--Study metric on $L_F$, and hence $\phi$ psh implies $\phi_F$ psh. Assume conversely that $\phi_F$ is psh. By Theorem~\ref{thm:FSQ}, 
$$
\e_m:=\sup_{X_F^\an}\left|m^{-1}\FS(m\phi_F)-\phi_F\right|
$$ 
tends to $0$ as $m\in\N(L)$ tends to $\infty$; we need to show that $\phi_m:=m^{-1}\FS(m\phi)$ converges uniformly to $\phi$. By Lemma~\ref{lem:pshbase}, we have
$$
(\phi_m)_F=m^{-1}\FS\left(\left(\n_{m\phi}\right)_F\right).
$$
Thus
$$
\sup_{X^\an}\left|\phi_m-\phi\right|=\sup_{X^\an_F}\left|(\phi_m)_F-\phi_F\right|\le\sup_{X^\an_F}\left|(\phi_m)_F-m^{-1}\FS(m\phi_F)\right|+\e_m
$$
$$
=m^{-1}\sup_{X^\an_F}\left|\FS\left(\left(\n_{m\phi}\right)_F\right)-\FS\left(\n_{m\phi_F}\right)\right|+\e_m\le m^{-1}\dGI(\n_{m\phi_F},(\n_{m\phi})_F)+\e_m,
$$
by~\eqref{equ:FSlip}. By Theorem~\ref{thm:supground}, $m^{-1}\dGI(\n_{m\phi_F},(\n_{m\phi})_F)\to 0$, and are done. 
\end{proof}

\part{Asymptotics of relative volumes}
%
%
\section{Monge--Amp\`ere measures and Deligne pairings}
In this section, we review some basic properties of Monge--Amp\`ere operators, and use them to define metrics on Deligne pairings (Theorem~\ref{thm:delmetr}), following Deligne's program~\cite{Del}. In what follows, $X$ is as before an $n$-dimensional projective scheme over a complete valued field $K$.
%
%
\subsection{Mixed Monge--Amp\`ere measures}\label{sec:MA}
The (algebraic) fundamental class of $X$ is the $n$-dimensional cycle 
$$
[X]=\sum_\a m_\a[X_\a],
$$
where the $X_\a$ denote the $n$-dimensional irreducible components of $X$ (with their reduced structure) and $m_\a$ is the length of the local ring of $X$ at the generic point of $X_\a$. The intersection number of $\Q$-line bundles $L_1,\ldots,L_n$ on $X$ is defined as
$$
(L_1\cdot\ldots\cdot L_n):=\deg\pi_\star \left(c_1(L_1)\cdot\ldots\cdot c_1(L_n)\cdot [X]\right),
$$ 
with $\pi:X\to\Spec K$ the structure morphism.

Now let $\phi_1,\ldots,\phi_n$ be smooth metrics on $L_1,\ldots,L_n$, and recall that $dd^c\phi_i$ denotes the curvature form of $\phi_i$, a smooth, closed $(1,1)$-form on $X^\an$ (cf.~\cite[\S 6.4.1]{CLD} for the non-Archimedean case). The smooth $(n,n)$-form $dd^c\phi_1\wedge\ldots\wedge dd^c\phi_n$ determines a Radon measure
\begin{equation}\label{equ:mixedMA}
dd^c\phi_1\wedge\ldots\wedge dd^c\phi_n\wedge\d_X=\sum_\a m_\a\,dd^c\phi_1\wedge\ldots\wedge dd^c\phi_n\wedge\d_{X_\a}
\end{equation}
on the topological space $X^\an$, the \emph{mixed Monge--Amp\`ere measure} of the $\phi_i$, with 
$\d_X=\sum_\a m_\a\d_{X_\a}$ denoting the analytic fundamental class of $X$ (see~\cite[\S 3.7]{CLD} for the non-Archimedean case). 

\begin{rmk}\label{rmk:MAarch} When $K=\R$, we define $dd^c\phi_1\wedge\ldots\wedge dd^c\phi_n\wedge\d_X$ as the image of the corresponding measure on $X(\C)=X_\C^\an$ by the projection map $X_\C^\an\to X^\an$. 
\end{rmk}

\begin{exam}\label{exam:dX} Assume that $n=0$, \ie $X=\spec A$ with $A$ a finite $K$-algebra. The previous construction produces a positive measure $\d_X$ on the finite set $X^\an$, which is described as follows. We have a product decomposition $A=\prod_i A_i$ into local finite $K$-algebras $A_i$ corresponding to the connected components of $X$. The (reduced) irreducible components of $X$ are given by $X_i=\spec K_i$, where the residue field $K_i$ of $A_i$ is a finite extension of $K$, and the unique extension of the absolute value of $K$ to $K_i$ defines a point $x_i\in X^\an$. The current $\d_X=\sum_i m_i\d_{X_i}$ is a measure on $X^\an$, and the requirement that $\d_{X_i}$ has total mass $\deg\pi_\star [X_i]=[K_i:K]$ yields $\d_{X_i}=[K_i:K]\d_{x_i}$. As $m_i[K_i:K]=\dim_K A_i$, we conclude that
$$
\d_X=\sum_i(\dim_K A_i)\d_{x_i}. 
$$
\end{exam}

\begin{prop}\label{prop:mixedMA} For each tuple $\phi_1,\ldots,\phi_n$ of smooth metrics on $\Q$-line bundles $L_1,\ldots,L_n$, the mixed Monge--Amp\`ere measure $dd^c\phi_1\wedge\ldots\wedge dd^c\phi_n\wedge\d_X$ satisfies the following properties: 
\begin{itemize}
\item[(i)] it is a symmetric, multilinear function of the $\phi_i$, and a positive Radon measure if the $\phi_i$ are semipositive, \ie $dd^c\phi_i\ge 0$; 
\item[(ii)] $\int dd^c\phi_1\wedge\ldots\wedge dd^c\phi_n\wedge\d_X=(L_1\cdot\ldots\cdot L_n)$; 
\item[(iii)] if $L_0=L_1=\cO_X$ (and hence $\phi_0,\phi_1$ are smooth functions on $X^\an$), then 
\begin{equation}\label{equ:intpart}
\int\phi_0\,dd^c\phi_1\wedge dd^c\phi_2\wedge\ldots\wedge dd^c\phi_n\wedge\d_X=\int \phi_1\,dd^c\phi_0\wedge dd^c\phi_2\wedge\ldots\wedge dd^c\phi_n\wedge\d_X. 
\end{equation}
\item[(iv)] for any complete field extension $F/K$, we have
$$
p_\star\left(dd^c\phi_{1,F}\wedge\ldots\wedge dd^c\phi_{n,F}\wedge\d_{X_F}\right)=dd^c\phi_1\wedge\ldots\wedge dd^c\phi_n\wedge\d_X,
$$
where $\phi_{i,F}$ is the pullback of $\phi_i$ to $L_{i,F}$ and $p:X_F^\an\to X^\an$ is the canonical projection.
\end{itemize}
\end{prop}
\begin{proof} (i) is straightforward. (ii) is classical in the Archimedean case, and proved in~\cite[Proposition 6.4.3]{CLD} in the non-Archimedean case. (iii) is a consequence of the Stokes formula (see~\cite[Th\'eor\`eme 3.12.2]{CLD} in the non-Archimedean case). In the Archimedean case, (iv) is nontrivial only when $K=\R$, $F=\C$, in which case it holds by definition (see Remark~\ref{rmk:MAarch}). In the non-Archimedean case, it is implicit in~\cite{CLD}, and can be checked as follows. Pick a smooth, compactly supported function $f$ on $X^\an$, and consider the smooth, compactly supported $(n,n)$-form $\a:=f\,dd^c\phi_1\wedge\ldots\wedge dd^c\phi_n$. According to~\cite[Proposition 5.13]{Gub13}, there exists a Zariski open subset $U\subset X$ with a closed embedding $U\hookrightarrow\G_{m,K}^r$ and a compactly supported smooth $(n,n)$-superform on $\R^r$ such that $\a$ is obtained by pulling-back $\eta$ by the tropicalization map $\trop:U^\an\to\R^r$, and
$$
\int f\,dd^c\phi_1\wedge\dots\wedge dd^c\phi_n\wedge\d_X=\int_{X}\a
$$ 
is defined as the integral of $\eta$ on the tropical cycle $\trop(U)$. Unravelling the definitions, it is clear that the pull-back $\a_F$ of $\a$ to $X_F^\an$ is simply the pull-back of $\eta$ by the tropicalization map $\trop_F:U_F^\an\to\R^r$. The construction of the tropical cycle of an algebraic variety being invariant under ground field extension, we have $\trop_F(U_F)=\trop(U)$ as tropical cycles, and we conclude as desired that $\int_{X}\a=\int_{X_F}\a_F$.  
\end{proof}

Recall from \S\ref{sec:pshreg} that a continuous metric on a $\Q$-line bundle $L$ is psh-regularizable iff it is a uniform limit of smooth semipositive metrics on $L$. If $L$ is semiample, then any continuous psh metric (in the sense of Definition~\ref{defi:psh}, \ie a uniform limit of Fubini--Study metrics) is psh-regularizable, by Theorem~\ref{thm:pshCLD}.

\begin{thm}\label{thm:BT} Let $L_1,\ldots,L_n$ be $\Q$-line bundles on $X$. The measure-valued operator 
$$
(\phi_1,\ldots,\phi_n)\mapsto dd^c\phi_1\wedge\ldots\wedge dd^c\phi_n\wedge\d_X,
$$ 
defined so far on smooth metrics, admits a unique continuous extension to tuples of continuous, psh-regularizable metrics on the $L_i$, with respect to uniform convergence for tuples of metrics, and the weak topology on Radon measures. 
\end{thm}

As a direct consequence of Proposition~\ref{prop:mixedMA}, we get:
\begin{prop} For each tuple of continuous, psh-regularizable metrics $(\phi_i)$, 
$$
dd^c\phi_1\wedge\ldots\wedge dd^c\phi_n\wedge\d_X=\sum_\a m_\a\, dd^c\phi_1\wedge\ldots\wedge dd^c\phi_n\wedge\d_{X_\a}
$$
is a positive Radon measure on $X^\an$, of total mass $(L_1\cdot\ldots\cdot L_n)$. It is further symmetric and multiadditive as a function of the $\phi_ i$. 
\end{prop} 
Theorem~\ref{thm:BT} follows from the classical Bedford--Taylor theory~\cite{BT76} in the Archimedean case, and from its analogue in the non-Archimedean case~\cite[Corollaire 5.6.5]{CLD}. In the present global setting, one can however provide a simple direct argument, based on the following global version of the classical Chern--Levine--Nirenberg inequality that will be put to use again a bit later. 

\begin{lem}\label{lem:CLN} Let $L_0,\ldots,L_n$ be line bundles on $X$, and suppose given a pair of smooth, semipositive metrics $\phi_i,\phi'_i$ on each $L_i$. Then 
$$
\left|\int(\phi_0-\phi'_0)\,dd^c\phi_1\wedge\ldots\wedge dd^c\phi_n\wedge\d_X-\int (\phi_0-\phi'_0)\,dd^c\phi'_1\wedge\ldots\wedge dd^c\phi'_n\wedge\d_X\right|
$$
$$
\le C\sum_{i=1}^n\sup|\phi_i-\phi'_i|
$$
with $C:=2\max_{1\le i\le n}(L_0\cdot L_1\cdot\ldots\cdot\widehat{L_i}\cdot\ldots\cdot L_n)$. 
\end{lem}
\begin{proof} For brevity of notation, set $T:=dd^c\phi_2\wedge\ldots\wedge dd^c\phi_n\wedge\d_X$. By~\eqref{equ:intpart}, 
$$
\int (\phi_0-\phi'_0)\,dd^c\phi_1\wedge T-\int(\phi_0-\phi'_0)\,dd^c\phi'_1\wedge T=\int (\phi_0-\phi'_0)\,dd^c(\phi_1-\phi'_1)\wedge T
$$
$$
=\int(\phi_1-\phi_1')\,dd^c(\phi_0-\phi'_0)\wedge T=\int(\phi_1-\phi_1')\,dd^c\phi_0\wedge T+\int(\phi_1-\phi'_1)\,dd^c\phi'_0\wedge T. 
$$
Since $dd^c\phi_0\wedge T$ and $dd^c\phi'_0\wedge T$ are both positive measures of mass $(L_0\cdot L_2\cdot\ldots\cdot L_n)$, we infer
$$
\left|\int (\phi_0-\phi'_0)\,dd^c\phi_1\wedge T-\int (\phi_0-\phi'_0)\,dd^c\phi'_1\wedge T\right|\le 2(L_0\cdot L_2\cdot\ldots\cdot L_n)\sup|\phi_1-\phi'_1|.
$$
We conclude by symmetry and multilinearity. 
\end{proof}

\begin{proof}[Proof of Theorem~\ref{thm:BT}] Assume first that $K$ is non-Archimedean, and pick $f\in\PL_\R(X^\an)$. By Lemma~\ref{lem:QPL}, we can find a pair of Fubini--Study metrics on an ample line bundle $L_0$ such that $f=\phi_0-\phi'_0$. By Theorem~\ref{thm:pshCLD}, $\phi_0$ and $\phi'_0$ are uniform limits of smooth, semipositive metrics on $L_0$. Lemma~\ref{lem:CLN} thus shows that 
$$
(\phi_1,\ldots,\phi_n)\mapsto\int f\,dd^c\phi_1\wedge\ldots\wedge dd^c\phi_n\wedge\d_X
$$
is Lipschitz continuous with respect to the supnorm on tuples of smooth, semipositive metrics (with Lipschitz constant depending on $f$), and hence extends continuously to tuples of continuous, psh-regularizable metrics. For each tuple of continuous, psh-regularizable metrics $(\phi_i)$, the linear form on $\PL_\R(X^\an)$ $f\mapsto\int f\,dd^c\phi_1\wedge\ldots\wedge dd^c\phi_n\wedge\d_X$ is positive, hence continuous, and it thus extends to a positive linear form on $\cz(X)$, by density of $\PL_\R(X^\an)$. 

In the Archimedean case, the argument is similar, arguing with a smooth function $f$ instead. It is in fact even simpler, since $f$ can directly be written as a difference of smooth psh metrics on some ample line bundle. 
\end{proof}

In order to exploit the multilinearity of the previous construction, it is convenient to introduce the following terminology.

\begin{defi}\label{defi:dpsh} We say that a continuous metric $\phi$ on a line bundle $L$ is  \emph{dpsh} if it can be written as a difference of continuous, psh-regularizable metrics.
\end{defi}
A dpsh metric is thus the same as an \emph{approachable metric} in the sense of~\cite[\S 6.3]{CLD}. 

\begin{lem}\label{lem:dpsh} Every line bundle $L$ admits a smooth, dpsh metric.
\end{lem}
\begin{proof} By Theorem~\ref{thm:pshCLD}, any ample line bundle admits a smooth, semipositive metric, and 
the result follows since $L$ can be written as a difference of ample line bundles.  
\end{proof}

\begin{exam} If $K$ is Archimedean, every smooth metric is dpsh.
\end{exam}
While this is probably true in the non-Archimedean case as well, it is less obvious, since strictly positive $(1,1)$-forms do not exist globally in this context.

\begin{exam} If $K$ is non-Archimedean, every $\R$-PL metric is dpsh, by Lemma~\ref{lem:QPL}. 
\end{exam} 

The operator $(\phi_1,\ldots,\phi_n)\mapsto dd^c\phi_1\wedge\ldots\wedge dd^c\phi_n\wedge\d_X$, being symmetric and multiadditive on tuples of continuous, psh-regularizable metrics, extends to a symmetric, multilinear operator on continuous dpsh metrics, with values in signed Radon measure. Note that Proposition~\ref{prop:mixedMA} still holds for such metrics. 

\begin{exam}\label{exam:MAshilov} If $K$ is non-Archimedean and $\phi_1,\dots\phi_n$ are $\R$-PL metrics, then $dd^c\phi_1\wedge\ldots\wedge dd^c\phi_n\wedge\d_X$ has finite support. By ground field invariance, we may indeed assume that $K$ is algebraically closed, and that its value group is large enough to ensure that the $\phi_i$ become pure after base change. By Proposition~\ref{prop:QPL}, $\phi_1,\ldots,\phi_n$ are then induced by $\Q$-models $\cL_1,\ldots,\cL_n$ of $L_1,\ldots,L_n$ determined on a projective model $\cX$ of $X$. By linearity with respect to the fundamental class $[X]$, we may further assume that $X$ is (geometrically) irreducible, which allows us to assume that $\cX_s$ is reduced, by Theorem~\ref{thm:redfiber}. Each irreducible component $Y$ of $\cX_s$ then determines a unique Shilov point $x_Y\in\Ga(\cX)$, and 
\begin{equation}\label{equ:MAredfib}
dd^c\phi_1\wedge\ldots\wedge dd^c\phi_n\wedge\d_X=\sum_Y(\cL_1|_Y\cdot\ldots\cdot\cL_n|_Y)\d_{x_Y},
\end{equation}
by~\cite[Proposition 6.9.2]{CLD} (compare~\cite[Proposition 3.11]{Gub07}). 
\end{exam}

%
%
\subsection{A Poincar\'e--Lelong formula}
We start with the following integrability result. 

\begin{thm}\label{thm:logs} Let $\phi_1,\ldots,\phi_n$ be continuous, psh-regularizable metrics on line bundles $L_1,\ldots,L_n$, and let $s\in \Hnot(X,L)$ be a regular section of a line bundle $L$, equiped with any continuous metric $\phi$. Then $\log|s|_\phi$ is integrable with respect to $dd^c\phi_1\wedge\ldots\wedge dd^c\phi_n\wedge\d_X$, and 
$$
\int\log|s|_\phi\,dd^c\phi_1\wedge\ldots\wedge dd^c\phi_n\wedge\d_X
$$ 
depends continuously on the $\phi_i$ with respect to uniform convergence. 
\end{thm}

\begin{lem}\label{lem:logs} We can find two ample line bundles $A,A'$, a decreasing sequence $(\p_j)$ of smooth, semipositive metrics on $A$, and a smooth, semipositive metric $\p'$ on $A'$ such that $L=A-A'$ and $\p_j-\p'\to\log|s|$ pointwise. 
\end{lem}
\begin{proof} Pick a very ample line bundle $A'$ such that $A:=L+A'$ is also very ample. Let $(s_\a)$ and $(s'_\b)$ be bases of $\Hnot(A)$ and $\Hnot(A')$, respectively, and set
$$
\p_j:=\tfrac{1}{2}\log\left(\sum_\b |s\cdot s'_\b|^2+e^{-j}\sum_\a |s_\a|^2\right),\,\,\,\,\,\,\p':=\tfrac12\log\left(\sum_\b |s'_\b|^2\right).
$$
These metrics are smooth and semipositive (both in the Archimedean and non-Archimedean cases), and
$$
\log|s|+\p'=\tfrac 12\log\left(\sum_\b |s\cdot s'_\b|^2\right)
$$
is the decreasing limit of $\p_j$.
\end{proof}

\begin{proof}[Proof of Theorem~\ref{thm:logs}] Given any other metric $\phi'$ on $L$, $\log|s|_\phi-\log|s|_{\phi'}=\phi'-\phi$ is a continuous function on $X^\an$, and Theorem~\ref{thm:logs} for $\phi$ or $\phi'$ are thus equivalent, by Theorem~\ref{thm:BT}. By Lemma~\ref{lem:dpsh}, we may therefore assume that $\phi$ is smooth dpsh.

Write $\log|s|$ as the decreasing limit of a sequence of smooth dpsh metrics $\p_j-\p'$ as in Lemma~\ref{lem:logs}, so that 
$$
f_j:=\p_j-\p'-\phi\to f:=\log|s|_\phi,
$$
where $\p_j$ and $\p'+\phi$ are metrics on a fixed (ample) line bundle $A$. By Lemma~\ref{lem:CLN}, there exists $C>0$ only depending on the $L_i$ and $A$ such that for any other tuple $\phi'_1,\ldots,\phi'_n$ of continuous, psh-regularizable metrics on $L_1,\ldots,L_n$ we have 
\begin{equation}\label{equ:logslip}
\left|\int f_j\,dd^c\phi_1\wedge\ldots\wedge dd^c\phi_n\wedge\d_X-\int f_j\,dd^c\phi'_1\wedge\ldots\wedge dd^c\phi'_n\wedge\d_X\right|\le C\sum_{i=1}^n\sup|\phi_i-\phi'_i|. 
\end{equation}
By monotone convergence, we infer
\begin{align*}
-C\sum_{i=1}^n\sup|\phi_i-\phi'_i|+\int f\,dd^c\phi_1\wedge\ldots\wedge dd^c\phi_n\wedge\d_X\le\int f\,dd^c\phi'_1\wedge\ldots\wedge dd^c\phi'_n\wedge\d_X\\
\le\int f\,dd^c\phi_1\wedge\ldots\wedge dd^c\phi_n\wedge\d_X+C\sum_{i=1}^n\sup|\phi_i-\phi'_i|
\end{align*}
in $[-\infty,+\infty)$. It thus remains to show that $f=\log|s|_\phi$ is integrable with respect to $dd^c\phi_1\wedge\ldots\wedge dd^c\phi_n\wedge\d_X$ for \emph{some} choice of continuous, psh-regularizable metrics $(\phi_i)$, which holds as soon as the $\phi_i$ are smooth (see~\cite[4.6.2]{CLD} for the non-Archimedean case, where we could alternatively take the $(\phi_i)$ to be model metrics, since 
$dd^c\phi_1\wedge\ldots\wedge dd^c\phi_n\wedge\d_X$ is then supported in a finite set of Shilov points, by Example~\ref{exam:MAshilov}). 
\end{proof}

We are now in a position to state the following version of the Poincar\'e--Lelong formula, which plays a key role in the proof of Theorem A. 
\begin{thm}\label{thm:PLformula} Let $\phi_2,\ldots,\phi_n$ be continuous dpsh metrics on line bundles $L_2,\ldots,L_n$, and set for brevity $T:=dd^c\phi_2\wedge\ldots\wedge dd^c\phi_n$. Suppose also given a regular section $s\in \Hnot(X,L)$ of a line bundle $L$, with divisor $Z$, a continuous dpsh metric $\phi$ on $L$, and a continuous dpsh function $f$. Then 
\begin{equation}\label{equ:PL}
\int\log|s|_{\phi}\,dd^c f\wedge T\wedge\d_X=\int f\,T\wedge\d_Z-\int f\,dd^c\phi\wedge T\wedge\d_X.
\end{equation}
\end{thm}
Note that the first integral is well-defined, by Theorem~\ref{thm:logs} and multilinearity. 

\begin{proof} Given any other continuous dpsh metric $\phi'$ on $L$, we have $\log|s|_{\phi}-\log|s|_{\phi'}=\phi'-\phi$, hence
$$
\int\log|s|_{\phi}\,dd^c f\wedge T\wedge\d_X-\int\log|s|_{\phi'}\,dd^c f\wedge T\wedge\d_X=\int(\phi'-\phi)\,dd^c f\wedge T\wedge\d_X
$$
$$
=\int f\,dd^c\phi'\wedge T\wedge\d_X-\int f\,dd^c\phi\wedge T\wedge\d_X,
$$
by integration by parts. As a result, the desired formula for $\phi$ is equivalent to that for $\phi'$, and we may thus assume from the start that $\phi$ is smooth, by Lemma~\ref{lem:dpsh}. 

Assume next that $f$ and $\phi_2,\ldots,\phi_n$ are also smooth. The Poincar\'e--Lelong formula (cf.~\cite[Theorem 4.6.5]{CLD} in the non-Archimedean case) implies that 
$$
dd^c\log|s|_{\phi}\wedge\d_X=\d_Z-dd^c\phi\wedge\d_X
$$ 
in the sense of currents, which yields~\eqref{equ:PL}. Assume finally that $f$ and $\phi_2,\ldots,\phi_n$ are merely continuous dpsh. We can then find line bundles $M,M_2,M'_2,\ldots,M_n,M'_n$ and continuous, psh-approchable metrics $\tau,\tau'$ on $M$, $\p_i,\p'_i$ on $M_i$, $M'_i$ such that $f=\p-\p'$, $L_i=M_i-M'_i$ and $\phi_i=\p_i-\p'_i$. Choose sequences $(\tau_j)$, $(\tau'_j)$, $(\p_{i,j})$ and $(\p'_{i,j})$ of smooth, semipositive metrics on $M$, $M_i$ and $M'_i$ such that $\tau_j\to\tau$, $\tau'_j\to\tau_j$, $\p_{i,j}\to\p_i$ and $\p'_{i,j}\to\p'_i$ uniformly. Set $f_j:=\tau_j-\tau'_j$, $\phi_{i,j}:=\p_{i,j}-\p'_{i,j}$, and $T_j:=dd^c\phi_{2,j}\wedge\ldots\wedge dd^c\phi_{n,j}$. Theorem~\ref{thm:logs} and multilinearity yield
$$
\lim_j\int\log|s|_{\phi}\,dd^c f_j\wedge T_j\wedge\d_X=\int\log|s|_{\phi}\,dd^c f\wedge T\wedge\d_X
$$
By Theorem~\ref{thm:BT}, we have weak convergence of measures
$$
T_j\wedge\d_Z\to T\wedge\d_Z,\,\,\,\,dd^c\phi\wedge T_j\wedge\d_X\to dd^c\phi\wedge T\wedge\d_X,
$$
and hence 
$$
\int f_j T_j\to\int_Z f T\,\,\,\text{   and   }\,\,\,\int f_j\,dd^c\phi\wedge T_j\wedge\d_X\to\int f\,dd^c\phi\wedge T\wedge\d_X,
$$
by uniform convergence $f_j\to f$. Thus~\eqref{equ:PL} in the smooth case implies the general case. 
\end{proof}
%
%
\subsection{Metrics on Deligne pairings}\label{sec:delmetric}
As a special case of a general construction discussed extensively in Appendix~\ref{sec:Deligne}, the \emph{Deligne pairing} associates to an $(n+1)$-tuple of line bundles $L_0,\ldots,L_n$ on an $n$-dimensional projective $K$-scheme $X$ a line bundle $\langle L_0,\ldots,L_n\rangle$ on $\spec K$, \ie a one-dimensional $K$-vector space. 

Following F.~Ducrot's approach~\cite{Duc}, we view the Deligne pairing as the $(n+1)$-st iterated difference of the determinant of cohomology. In our situation, the determinant of cohomology of a line bundle $L$ on $X$ can simply be described as
$$
\la(L)=\sum_i(-1)^i\det H^i(L),
$$
and
$$
\langle L_0,\ldots,L_n\rangle=\sum_{I\subset\{0,\ldots,n\}}(-1)^{n+1-|I|}\la\left(\sum_{i\in I} L_i\right).
$$
This pairing is functorial, symmetric, multilinear, and the data of a regular section $s\in \Hnot(X,L_0)$ with zero divisor $Z$ induces a canonical isomorphism 
\begin{equation}\label{equ:deldiv}
\langle L_0,\ldots,L_n\rangle=\langle L_1|_{Z},\ldots,L_n|_{Z}\rangle.
\end{equation}

\begin{exam}\label{exam:delnorm} When $n=0$, a regular section $s$ of $L_0$ is a trivializing section of $L_0$, and~\eqref{equ:deldiv} thus yields a generator $\langle s_0\rangle\in\langle L_0\rangle$. Any other trivializing section of $L_0$ is of the form $u s_0$ with $u\in A^\times$ a unit, and 
\begin{equation}\label{equ:normsection}
\langle us_0\rangle=N_{A/K}(u)\langle s_0\rangle
\end{equation}
where $N_{A/K}(u)\in K^\times$ is the \emph{norm} of $u$, \ie the determinant of the endomorphism of the $K$-vector space $A$ given by multiplication by $u$. In other words, the Deligne pairing coincides with norm functor when $n=0$.  
\end{exam}


In line with Deligne's program~\cite{Del} and the work of Elkik~\cite{Elkik}, we prove: 

\begin{thm}\label{thm:delmetr} There exists a unique way to associate to each tuple of continuous dpsh metrics $\phi_0,\ldots,\phi_n$ on line bundles $L_0,\ldots,L_n$ over an $n$-dimensional projective $K$-scheme $X$ a metric $\langle\phi_0,\ldots,\phi_n\rangle$ on $\langle L_0,\ldots,L_n\rangle$, such that the following holds: 
\begin{itemize}
\item[(i)] the pairing $(\phi_0,\ldots,\phi_n)\mapsto\langle\phi_0,\ldots,\phi_n\rangle$ is symmetric and multilinear; 
\item[(ii)] if $s\in \Hnot(L_0)$ is a regular section with zero divisor $Z$, the following \emph{restriction formula} holds:
\begin{equation}\label{equ:delmetres}
\langle\phi_0,\ldots,\phi_n\rangle=\langle\phi_0|_Z,\ldots,\phi_n|_Z\rangle-\int\log|s|_{\phi_0}dd^c\phi_1\wedge\ldots\wedge dd^c\phi_n\wedge\d_X
\end{equation}
as metrics on $\langle L_0,\ldots,L_n\rangle=\langle L_1|_Z,\ldots,L_n|_Z\rangle$.
\end{itemize}
\end{thm}
Note that (ii) implies the \emph{change of metric formula} 
\begin{equation}\label{equ:changemetr}
\langle\phi_0,\phi_1,\ldots,\phi_n\rangle-\langle\phi'_0,\phi_1,\ldots,\phi_n\rangle=\int(\phi_0-\phi'_0) dd^c\phi_1\wedge\ldots\wedge dd^c\phi_n\wedge\d_X
\end{equation}
for all continuous dpsh metrics $\phi_0,\phi'_0$ on $L_0$ and $\phi_1,\ldots,\phi_n$ on $L_1,\ldots,L_n$. This follows indeed by applying~\eqref{equ:delmetres} to the dpsh metrics $\phi_0-\phi'_0,\phi_1,\ldots,\phi_n$ on $\cO_X,L_1,\ldots,L_n$ and the regular section $s=1\in \Hnot(X,\cO_X)$. 

%

\begin{lem}\label{lem:delmetr} Let $\phi_0,\ldots,\phi_n$ be continuous dpsh metrics on line bundles $L_0,\ldots,L_n$. Let $s_0\in \Hnot(L_0)$, $s_1\in \Hnot(L_1)$ be regular sections with divisors $Z_0,Z_1$, and assume that $s_1|_{Z_0}$ and $s_0|_{Z_1}$ are also regular. Then:
\begin{itemize}
\item[(i)] the induced isomorphisms 
$$
\langle L_0,L_1,L_2,\ldots,L_n\rangle=\langle L_1|_{Z_0},L_2|_{Z_0},\ldots,L_n|_{Z_0}\rangle=\langle L_2|_{Z_0\cap Z_1},\ldots,L_n|_{Z_0\cap Z_1}\rangle
$$
and
$$
\langle L_0,L_1,L_2,\ldots,L_n\rangle=\langle L_1,L_0,L_2,\ldots,L_n\rangle
$$
$$
=\langle L_0|_{Z_1},L_2|_{Z_1},\ldots,L_n|_{Z_1}\rangle=\langle L_2|_{Z_0\cap Z_1},\ldots,L_n|_{Z_0\cap Z_1}\rangle
$$
coincide; 
\item[(ii)] setting for brevity $T:=dd^c\phi_2\wedge\ldots\wedge dd^c\phi_n$, we have 
$$
\int\log|s_0|_{\phi_0}\,T\wedge\d_{Z_1}-\int\log|s_0|_{\phi_0}\,dd^c\phi_1\wedge T\wedge\d_X
$$
$$
=\int\log|s_1|_{\phi_1}\,T\wedge\d_{Z_0}-\int\log|s_1|_{\phi_1}\,dd^c\phi_0\wedge T\wedge\d_X. 
$$
\end{itemize}
\end{lem}
When $X$, $Z_0$ and $Z_1$ are smooth over an Archimedean field $K$ and the metrics are smooth, the result is a special case~\cite[Lemma I.1.2]{Elkik}. The proof we propose here follows standard monotone regularization arguments in pluripotential theory (compare for instance~\cite[Proposition III.4.9]{Dem}). 

\begin{proof} To prove (i), recall from Theorem~\ref{thm:del} that the restriction isomorphism 
$$
\langle L_0,L_1,L_2,\ldots,L_n\rangle=\langle L_1|_{Z_0},L_2|_{Z_0},\ldots,L_n|_{Z_0}\rangle
$$
is induced by~\eqref{equ:deldiv} and the restriction isomorphism 
$$
\la(L)-\la(L-Z_0)=\la_{Z_0}(L|_{Z_0}),
$$
valid for any line bundle $L$ on $X$. Thus
$$
\langle L_0,L_1,L_2,\ldots,L_n\rangle=\langle L_2|_{Z_0\cap Z_1},L_2|_{Z_0\cap Z_1},\ldots,L_n|_{Z_0\cap Z_1}\rangle
$$
is induced by 
$$
\la(L)-\la(L-Z_0)-\la(L-Z_1)+\la(L-Z_0-Z_1)=\la_{Z_0\cap Z_1}(L),
$$
whose symmetry with respect to $Z_0$ and $Z_1$ implies (i). 

We turn to (ii). By Lemma~\ref{lem:dpsh} and multilinearity, we may assume that $\phi_2,\ldots,\phi_n$ are smooth and semipositive, and also that $\phi_0,\phi_1$ are smooth dpsh. 

For $i=0,1$, set $f_i:=\log|s_i|_{\phi_i}$. As in the proof of Theorem~\ref{thm:logs}, we can choose an (ample) line bundle $H_i$, a decreasing sequences of smooth, semipositive metrics $(\p_{i,j})_j$ on $H_i$ and a smooth, semipositive metric $\p'_i$ such that $f_{i,j}:=\p_{i,j}-\p'_i$ converges to $f_i$ as $j\to\infty$. By integration by parts, 
$$
\int f_{0,j}\,dd^c f_{1,j}\wedge T\wedge\d_X=\int f_{1,j}\,dd^cf_{0,j}\wedge T\wedge\d_X,
$$
and we will thus be done if we prove that
\begin{equation}\label{equ:delsym}
\int f_{0,j}\,dd^cf_{1,j}\wedge T\wedge\d_X\to\int f_0\,T\wedge\d_{Z_1}-\int f_0\,dd^c\phi_1\wedge T\wedge\d_X,
\end{equation}
and 
$$
\int f_{1,j}dd^cf_{0,j}\wedge T\wedge\d_X\to\int f_1\,T\wedge\d_{Z_0}-\int f_1\,dd^c\phi_0\wedge T\wedge\d_X, 
$$
the first of these being enough, by symmetry. Pick indices $j\ge k$. Then $f_{0,j}\le f_{0,k}$, and hence 
$$
\int f_{0,j}\,dd^cf_{1,j}\wedge T\wedge\d_X=\int f_{0,j}dd^c\p_{1,j}\wedge T\wedge\d_X-\int f_{0,j}\,dd^c\p'_1\wedge T\wedge\d_X
$$
$$
\le\int f_{0,k}\,dd^c\p_{1,j}\wedge T\wedge\d_X-\int f_{0,j}\,dd^c\p'_1\wedge T\wedge\d_X
$$
$$
=\int f_{0,k}\,dd^c f_{1,j}\wedge T\wedge\d_X+\int (f_{0,k}-f_{0,j})\,dd^c\p'_1\wedge T\wedge\d_X. 
$$
Letting first $j\to\infty$, we infer
$$
\limsup_j\int f_{0,j}\,dd^cf_{1,j}\wedge T\wedge\d_X\le\int f_{0,k}\,dd^c f_1\wedge T\wedge\d_X+\int (f_{0,k}-f_{0})\,dd^c\p'_1\wedge T\wedge\d_X
$$
$$
=\int f_{0,k}\,T\wedge\d_{Z_1}-\int f_{0,k}\,dd^c\phi_1\wedge T\wedge\d_X+\int (f_{0,k}-f_{0})\,dd^c\p'_1\wedge T\wedge\d_X,
$$
and letting next $k\to\infty$ yields 
$$
\limsup_j\int f_{0,j}\,dd^cf_{1,j}\wedge T\wedge\d_X\le\int f_0\,T\wedge\d_{Z_1}-\int f_0\,dd^c\phi_1\wedge T\wedge\d_X, 
$$
by monotone convergence. For the converse, pick now indices $j\le k$. Then
$$
\int f_{0,j}\,dd^c f_{1,j}\wedge T\wedge\d_X=\int f_{1,j}\,dd^cf_{0,j}\wedge T\wedge\d_X=\int f_{1,j}\,dd^c\p_{0,j}\wedge T\wedge\d_X-\int f_{1,j}\,dd^c\p'_0\wedge T\wedge\d_X
$$
$$
\ge\int f_{1,k}\,dd^c\p_{0,j}\wedge T\wedge\d_X-\int f_{1,j}\,dd^c\p'_0\wedge T\wedge\d_X=\int f_{1,k}\,dd^c f_{0,j}\wedge T\wedge\d_X+\int(f_{1,k}-f_{1,j})\,dd^c\p'_0\wedge T\wedge\d_X
$$
$$
=\int f_{0,j}\,dd^c f_{1,k}\wedge T\wedge\d_X+\int(f_{1,k}-f_{1,j})\,dd^c\p'_0\wedge T\wedge\d_X.
$$
As $k\to\infty$, $dd^c f_{1,k}\wedge T\wedge\d_X\to dd^c f_1\wedge T\wedge\d_X=T\wedge\d_{Z_1}+dd^c\phi_1\wedge T\wedge\d_X
$ in the sense of currents, while $\int f_{1,k}\,dd^c\p'_0\wedge T\wedge\d_X\to\int f_1\,dd^c\p'_0\wedge T\wedge\d_X$, by monotone convergence. Thus
$$
\int f_{0,j}dd^cf_{1,j}\wedge T\wedge\d_X\ge\int f_{0,j} T\wedge\d_{Z_1}-\int f_{0,j}\,dd^c\phi_1\wedge T\wedge\d_X+\int(f_1-f_{1,j})\,dd^c\p'_0\wedge T\wedge\d_X,
$$
and hence 
$$
\liminf_{j\to\infty}\int f_{0,j}dd^cf_{1,j}\wedge T\wedge\d_X\ge \int f_0\, T\wedge\d_{Z_1}-\int f_0\,dd^c\phi_1\wedge T\wedge\d_X,
$$
using again monotone convergence. 
\end{proof}

\begin{proof}[Proof of Theorem~\ref{thm:delmetr}] We argue by induction on $n$. Assume $n=0$, \ie $X=\spec A$ with $A$ finite, flat over $K$. We can then pick a trivializing section $s\in \Hnot(X,L_0)$, and~\eqref{equ:delmetres} yields
\begin{equation}\label{equ:normmet}
\log|\langle s\rangle|_{\langle\phi_0\rangle}=\int\log|s|_{\phi_0}\d_X.
\end{equation}
In view of~\eqref{equ:normsection}, \eqref{equ:normmet} defines a metric $\langle\phi_0\rangle$ on $\langle L_0\rangle$ iff
\begin{equation}\label{equ:logu}
\int\log|u|\,\d_X=\log|N_{A/K}(u)|
\end{equation}
To see this, we may assume that $A$ is local, by Example~\ref{exam:dX}. Then
$$
\int\log|u|\,\d_X=(\dim_K A)\log|u(x)|,
$$
with $x$ is the unique point of $X^\an$, corresponding to the absolute value on the residue field $K'$ of $A$. A standard computation gives $N_{A/K}(u)=  N_{K'/K}(u|_{K'})^m$ with $m=\dim_{K'} A$ the length of $A$ (see for instance~\cite[Lemma 1.16.2]{Mor}). It is also well-known that $|N_{K'/K}(u)|=|u|_{K'}^{[K':K]}=|u(x)|^{[K':K]}$, and~\eqref{equ:logu} follows since $\dim_K A=m[K':K]$. 

We now consider a tuple of continuous dpsh metrics on line bundles $L_0,\ldots,L_n$ on a projective $K$-scheme $X$ of dimension $n$, and argue basically as in the proof of~\cite[Th\'eor\`eme I.1.1]{Elkik}. Given a regular section $s_0\in \Hnot(L_0)$ with divisor $Z_0$, we force the restriction formula by setting temporarily
\begin{equation}\label{equ:resdel}
\langle(\phi_0,s_0),\phi_1,\ldots,\phi_n\rangle:=\langle\phi_1|_{Z_0},\ldots,\phi_n|_{Z_0}\rangle-\int\log|s_0|_{\phi_0}\,dd^c\phi_1\wedge\ldots\wedge dd^c\phi_n\wedge\d_X,
\end{equation}
as metrics on $\langle L_0,\ldots,L_n\rangle=\langle L_1|_{Z_0},\ldots,L_n|_{Z_0}\rangle$, where the right-hand side is well-defined by induction and Theorem~\ref{thm:logs}. By induction, this pairing is multilinear and symmetric in $(\phi_1,\ldots,\phi_n)$. 

Assume next given a regular section $s_1\in \Hnot(L_1)$, with divisor $Z_1$, such that $s_1|_{Z_0}$ is also regular. We are going to show that 
\begin{equation}\label{equ:regsecdel}
\langle(\phi_0,s_0),\phi_1,\ldots,\phi_n\rangle=\langle(\phi_1,s_1),\phi_0,\phi_2,\ldots,\phi_n\rangle
\end{equation}
as metrics on $\langle L_0,L_1,L_2,\ldots,L_n\rangle=\langle L_1,L_0,L_2,\ldots,L_n\rangle$. We first claim that $s_0|_{Z_1}$ is regular as well. Indeed, the sequence $(s_0,s_1)$ is regular at each point of $Z_0\cap Z_1$, and hence so is $(s_1,s_0)$, by general properties of regular sequences in a Noetherian local ring. This means that $s_0|_{Z_1}$ is regular at each point of $Z_0$, and also trivially at each point not in $Z_0=\div(s_0)$, proving the claim. By induction, the restriction formula applies to $Z_0$, and yields 
$$
\langle\phi_1|_{Z_0},\ldots,\phi_n|_{Z_0}\rangle=\langle\phi_2|_{Z_0\cap Z_1},\ldots,\phi_n|_{Z_0\cap Z_1}\rangle-\int\log|s_1|_{\phi_1}\,T\wedge\d_{Z_0}
$$
with $T:=dd^c\phi_2\wedge\ldots\wedge dd^c\phi_n$. Using the commutativity property of Lemma~\ref{lem:delmetr}, we infer
$$
\langle(\phi_0,s_0),\phi_1,\ldots,\phi_n\rangle
$$
$$
=\langle\phi_2|_{Z_0\cap Z_1},\ldots,\phi_n|_{Z_0\cap Z_1}\rangle-\int\log|s_1|_{\phi_1}\,T\wedge\d_{Z_0}-\int\log|s_0|_{\phi_0}\,dd^c\phi_1\wedge T\wedge\d_X
$$
$$
=\langle\phi_2|_{Z_0\cap Z_1},\ldots,\phi_n|_{Z_0\cap Z_1}\rangle-\int\log|s_0|_{\phi_0}\,T\wedge\d_{Z_1}-\int\log|s_1|_{\phi_1}\,dd^c\phi_0\wedge T\wedge\d_X
$$
$$
=\langle\phi_0|_{Z_1},\phi_2|_{Z_1},\ldots,\phi_n|_{Z_1}\rangle-\int\log|s_1|_{\phi_1}\,dd^c\phi_0\wedge T\wedge\d_X
$$
$$
=\langle(\phi_1,s_1),\phi_0,\phi_2,\ldots,\phi_n\rangle
$$
as metrics on 
$$
\langle L_0,\ldots,L_n\rangle=\langle L_2|_{Z_0\cap Z_1},\ldots,L_n|_{Z_0\cap Z_1}\rangle=\langle L_0|_{Z_1},\ldots,L_n|_{Z_1}\rangle=\langle L_1,L_0,L_2,\ldots,L_n\rangle,
$$
using again induction, which proves~\eqref{equ:regsecdel}. 

Assume now that $L_0$ admits two regular sections $s_0,s'_0\in \Hnot(L_0)$, with divisors $Z_0,Z'_0$. We claim that 
$$
\langle (\phi_0,s_0),\phi_1,\ldots,\phi_n\rangle=\langle (\phi_0,s'_0),\phi_1,\ldots,\phi_n\rangle
$$
as metrics on $\langle L_0.\dots,L_n\rangle$. By linearity with respect to $\phi_1$, we may assume that $L_1$ is ample. For $m\gg 1$, we then find a regular section $s_1\in \Hnot(m L_1)$ such that $s_1|_{Z_0}$ and $s_1|_{Z'_0}$ are also regular, since this amounts to requiring that $s_1$ is nonzero at each of the finitely many associated points of $X$, $Z_0$ and $Z'_0$. Applying~\eqref{equ:regsecdel} twice, we infer
$$
\langle (\phi_0,s_0),\phi_1,\ldots,\phi_n\rangle=m^{-1}\langle (m\phi_1,s_1),\phi_0,\phi_2,\ldots,\phi_n\rangle=\langle (\phi_0,s'_0),\phi_1,\ldots,\phi_n\rangle
$$
which proves the claim. Given a tuple of continuous dpsh metrics $\phi_0,\ldots,\phi_n$ on line bundles $L_0,\ldots,L_n$ such that $L_0$ admits a regular section, we can thus set
$$
\langle \phi_0,\ldots,\phi_n\rangle:=\langle (\phi_0,s_0),\phi_1,\ldots,\phi_n\rangle
$$
for any choice of regular section $s_0\in \Hnot(L_0)$. If $\phi'_0$ is a continuous dpsh metric on another line bundle $L'_0$ admitting a regular section, then $L_0+L'_0$ also admits a regular section, and~\eqref{equ:regsecdel} shows that
$$
\langle \phi_0+\phi'_0,\phi_1,\ldots,\phi_n\rangle=\langle \phi_0,\phi_1,\ldots,\phi_n\rangle+\langle \phi'_0,\phi_1,\ldots,\phi_n\rangle.
$$
Consider finally a continuous dpsh metric $\phi_0$ on an arbitrary line bundle $L_0$. After adding to $L_0$ a sufficiently large multiple of an ample line bundle, we can write $L_0=M_0-M'_0$ with $M_0,M'_0$ both admitting a regular section. Since $M'_0$ admits a continuous dpsh metric $\p'_0$, we get $\phi_0=\p_0-\p'_0$ with $\p_0:=\phi_0+\p'_0$ continuous dpsh on $M_0$, and we set 
$$
\langle\phi_0,\phi_1,\ldots,\phi_n\rangle:=\langle\p_0,\phi_1,\ldots,\phi_n\rangle-\langle\p'_0,\phi_1,\ldots,\phi_n\rangle.
$$
By the previous additivity property, this is independent of the choice of $\p_0,\p'_0$, and multiadditive with respect to $\phi_0$. By~\eqref{equ:regsecdel}, the pairing $\langle\phi_0,\phi_1,\ldots,\phi_n\rangle$ is symmetric with respect to $\phi_0,\phi_1$, and hence with respect to any permutation of $\phi_0,\ldots,\phi_n$, being already symmetric with respect to $\phi_1,\ldots,\phi_n$ by induction. Finally, the pairing satisfies by construction the restriction formula, and we are done. 
\end{proof}
%
%
\subsection{The non-Archimedean case}
We assume in this section that $K$ is non-Archimedean, possibly trivially valued. By the change of metric formula~\eqref{equ:changemetr}, the Deligne pairing is basically determined by its values on model metrics, which are described by the following result. 

\begin{thm}\label{thm:delmetrNA} Let $L_0,\ldots,L_n$ be line bundles on a projective $K$-scheme $X$ of dimension $n$. Let $\cL_0,\ldots,\cL_n$ be models of $L_0,\ldots,L_n$, determined on a projective model $\cX$ of $X$, and denote by $\phi_0,\ldots,\phi_n$ the corresponding model metrics. Then 
\begin{equation}\label{equ:delmodel}
\langle\phi_0,\ldots,\phi_n\rangle=\phi_{\langle\cL_0,\ldots,\cL_n\rangle}.
\end{equation}
\end{thm}
Here $\langle\cL_0,\ldots,\cL_n\rangle$ is the Deligne pairing with respect to the flat projective morphism $\cX\to\Spec K^\circ$, cf.~Appendix~\ref{sec:Deligne}. Since Deligne pairings commute with base change, $\langle\cL_0,\ldots,\cL_n\rangle$ is a model of $\langle L_0,\ldots,L_n\rangle$, and $\phi_{\langle\cL_0,\ldots,\cL_n\rangle}$ denotes the corresponding model metric. 

\begin{proof} We argue by induction on $n$. Assume first $n=0$, and hence $\cX=\spec\cA$ with $\cA$ finite free over $K^\circ$. As recalled in Lemma~\ref{lem:norm}, any line bundle on $\cX$ is trivial in a neighborhood of the special fiber of $\cX$, and hence trivial on $\cX$. A trivializing section $s\in \Hnot(\cL_0)$ induces a trivializing section $\langle s\rangle=N_{\cX/K^\circ}(s)$ of $\langle\cL_0\rangle=N_{\cX/K^\circ}(\cL_0)$, as well as a trivializing section $s\in \Hnot(X,L_0)$, such that $|s|_{\phi_0}=1$ on $X^\an$. By~\eqref{equ:normmet}, we infer
$$
\log|\langle s\rangle|_{\langle\phi_0\rangle}=\int\log|s|_{\phi_0}\d_X=0, 
$$
which is equivalent to $\langle\phi_0\rangle=\phi_{\langle\cL_0\rangle}$. 

Let now $L_0,\ldots,L_n$ be line bundle on a projective $K$-scheme $X$ of dimension $n$, and $\cL_0,\ldots,\cL_n$ be models of $L_0,\ldots,L_n$ determined on a projective model $\cX$ of $X$. Since any line bundle on $\cX$ can be written as a difference of ample line bundles, we may assume that $\cL_0$ is ample, by linearity. After passing to a large enough multiple, we assume that there exists a relatively regular section $s\in \Hnot(\cL_0)$, by Proposition~\ref{prop:reg}. Denote by $\cZ$ its divisor, with generic fiber $Z$. By the restriction property of Deligne pairings on models, $s$ induces an isomorphism 
$$
\langle\cL_0,\cdots,\cL_n\rangle=\langle\cL_1|_\cZ,\ldots,\cL_n|_\cZ\rangle
$$
of models of $\langle L_0,\ldots,L_n\rangle=\langle L_1|_Z,\ldots,L_n|_Z\rangle$. By induction, we get
$$
\phi_{\langle\cL_0,\cdots,\cL_n\rangle}=\phi_{\langle\cL_1|_\cZ,\ldots,\cL_n|_\cZ\rangle}=\langle\phi_1|_Z,\ldots,\phi_n|_Z\rangle
$$
$$
=\langle\phi_0,\ldots,\phi_n\rangle+\int\log|s|_{\phi_0}\,dd^c\phi_1\wedge\ldots\wedge dd^c\phi_n\wedge\d_X,
$$
using the restriction property of Theorem~\ref{thm:delmetr}. It remains to show that the right-hand integral vanishes, which follows from the fact that $dd^c\phi_1\wedge\ldots\wedge dd^c\phi_n\wedge\d_X$ is supported on $\Ga(\cX)$ (cf.~Example~\ref{exam:MAshilov}) together with Lemma~\ref{lem:normregsec} below. 
\end{proof}

\begin{lem}\label{lem:normregsec} If $s$ is a relatively regular section of a line bundle $\cL$ on a model $\cX$, then $|s|_{\phi_\cL}=1$ on $\Ga(\cX)$. 
\end{lem}
\begin{proof} Since $s$ doesn't vanish at the associated points of $\cX_s$, $s$ is nonzero on the reduction $\xi:=\red_\cX(x)$ of any Shilov point $x\in X^\an$ defined by $\cX$. In other words, $s$ defines a trivializing section of $\cL$ on an open neighborhood $\cU$ of $\xi$, and we infer $|s|_{\phi_\cL}\equiv 1$ on $\cU^\beth$, by definition of the model metric $\phi_{\cL}$. 
\end{proof}
Using Theorem~\ref{thm:delmetrNA}, we relate Deligne pairings and Gubler's intersection theory on models~\cite{Gub98}. 

\begin{defi} If $D$ is a vertical Cartier divisor and $\cL_1,\ldots,\cL_n$ are line bundles on a projective model $\cX$ of $X$, we define their \emph{intersection number} as the real number
$$
(D\cdot\cL_1\cdot\ldots\cdot\cL_n):=\phi_{\langle\cO_\cX(D),\cL_1,\ldots,\cL_n\rangle}. 
$$
\end{defi} 
Here the right-hand side is a model function on $\spec K$, identified with its value on the unique point of $\spec K$. Equivalently, $(D\cdot\cL_1\ldots\cdot\cL_n)=v_K(f)$ for any choice of generator $f$ of the free $K^\circ$-module 
$$
\langle\cO(D),\cL_1,\ldots,\cL_n\rangle\hookrightarrow K.
$$
In particular, $(D\cdot\cL_1\ldots\cdot\cL_n)$ belongs to the (additive) value group $\Ga_K\subset\R$. 

As a special case of Theorem~\ref{thm:delmetrNA}, we have
\begin{equation}\label{equ:intnumb}
(D\cdot\cL_1\ldots\cdot\cL_n)=\int\phi_D\,dd^c\phi_{\cL_1}\wedge\ldots\wedge dd^c\phi_{\cL_n}\wedge\d_X.
\end{equation}
On the other hand, in~\cite[\S 3]{Gub98} Gubler associates to every vertical Cartier divisor $D$ on $\cX$ an $n$-dimensional cycle $\cyc(D)$ with coefficients in $\Ga_K$ on the finite type $\tilde K$-scheme $\cX_s$, such that the following holds:
\begin{itemize}
\item[(i)] if $\mu:\cX'\to\cX$ is a morphism of projective models, then $\cyc(D)=\mu_\star\cyc(\mu^\star D)$;
\item[(ii)] if $F/K$ is a non-Archimedean extension and $D_F$ is the pull-back of $D$ to the base change of $\cX$ to $F^\circ$, then $\cyc(D_F)=\cyc(D)_{\widetilde F}$;
\item[(iii)] if $\cX_s$ is reduced, every irreducible component $Y$ of $\cX_s$ determines a unique Shilov point $x_Y\in\Ga(\cX)$, and 
$$
\cyc(D)=\sum_Y\phi_D(x_Y)[Y].
$$
\end{itemize}

\begin{cor} Let $\cL_1,\ldots,\cL_n$ be line bundles on a model $\cX$ of $X$, and $D$ be a vertical Cartier divisor on $\cX$. Then 
$$
\left(D\cdot\cL_1\cdot \ldots\cdot\cL_n\right)=\left(\cyc(D)\cdot\cL_1|_{\cX_s}\cdot\ldots\cdot\cL_n|_{\cX_s}\right).
$$
\end{cor}
\begin{proof} By~\eqref{equ:intnumb}, $\left(D\cdot\cL_1\cdot \ldots\cdot\cL_n\right)$ is invariant under ground field extension, and we may thus assume that $K$ is algebraically closed and nontrivially valued. By multilinearity, we may further assume that $X$ is (geometrically) reduced. After replacing $\cX$ with a higher model, we can arrange for  $\cX_s$ to be reduced, by Theorem~\ref{thm:redfiber}. By~\eqref{equ:MAredfib}, we then have 
$$
\left(D\cdot\cL_1\cdot\ldots\cdot\cL_n\right)=\int\phi_D\,dd^c\phi_1\wedge\dots\wedge dd^c\phi_n
$$
$$
=\sum_Y\phi_D(x_Y)\left(\cL_1|_Y\cdot\ldots\cdot\cL_n|_Y\right)=\left(\cyc(D)\cdot\cL_1|_{\cX_s}\cdot\ldots\cdot\cL_n|_{\cX_s}\right).
$$
\end{proof}
%
%
\subsection{General metrics on the norm functor}
We extend the previous discussion on metrics on Deligne pairings to the setting of general finite flat morphisms $f: Y \to X$. In this setting, we will write abreviate $N = N_{X/Y}$ when there is no risk for ambiguity. 

The formula \eqref{equ:normmet} provides a general definition of a metric $N(\phi)$ on $N(L)$, provided we are given a metric $\phi$ on $L$. Namely, for $x\in X^\an$, and a section $s$ of $H^{0}(L|_{f^{-1}(U)})$ for some open neighborhood $U$ of $X$ such that $U^{\an}$ is an open neighborhood of $x$, we define
\begin{equation}\label{defi:normmetric}
\log |N(s)|_{N(\phi)} (x)= \int_{X} \log |s|_{\phi} \delta_{[f^{-1}(x)]}.
\end{equation}
Here $[f^{-1}(x)]$ denotes the fundamental cycle of $f^{-1}(x)$, defined in a way analogous to Example \ref{exam:dX}. The definition extends in a natural way to $\mathbb{Q}$-line bundles.

We have the following proposition. 
\begin{prop}\label{prop:descnormpsh}
Suppose $f: Y \to X$ is a finite flat morphism of projective geometrically reduced $K$-schemes, and $L$ an ample $\mathbb{Q}$-line bundle on $Y$. If $\phi$ is a continuous psh metric on $L$, then so is $N(\phi)$ on $N(L)$.
\end{prop}
\begin{proof} Arguing on the components of $X$ and $Y$, we may assume wlog that $f$ has constant degree $e>1$. The claim is standard in the Archimedean setting since it is straightforward to verify that the curvature of $N(L)$ is the direct image under $f$ of the curvature of $L$. 

Suppose now that $K$ is non-trivially valued and non-Archimedean. As it follows from \eqref{defi:normmetric}, the map $\phi\mapsto N(\phi)$ is Lipschitz continuous. By Theorem \ref{thm:pshZhang} we can hence restrict ourselves to the case when $\phi = \phi_\cL$, for a nef model $\cL$ of $L$ on a model $\cY$. By multiplicativity of the norm we may assume they are actual line bundles. By Proposition~\ref{prop:modelpullback} (and its proof) we can assume $f$ extends to a proper and flat morphism $f: \cY \to \cX$ of models. Flatness together with properness implies that the fiber dimension is locally constant, and hence $f$ is quasi-finite. As $f$ is automatically finitely presented these facts taken together imply it is also finite by \cite[IV.8.11.1]{EGA}. We can thus consider the norm of line bundles $\cL$ which are models of $N(L).$

 An application of the Riemann--Roch theorem for singular curves shows that $N(\cL)$ is nef if $\cL$ is. We claim that $N(\phi_{\cL}) = \phi_{N(\cL)}$ (compare with the first part of the proof of Theorem \ref{thm:delmetrNA}), in which case we conclude that $N(\phi_{\cL})$ is a nef metric. To prove the claim, pick a point $x \in \cX$, and an open neighborhood $\cU$ of $x$, and a trivialization $\tau$ of $\cL$ on $f^{-1}(\cU)$. Since $|\tau|_{\phi_\cL}(y')=1$ for $y' \in f^{-1}(x')$ for $x'\in \cU^\beth$, it follows from definition \eqref{defi:normmetric} that $N(\phi_\cL)$ is the model metric induced by $N(\cL)$. 

The trivially valued case can be deduced from the non-trivially valued case, in a straightforward application of Lemma \ref{lem:CM} and Theorem \ref{thm:pshbase}.
\end{proof}
%
%
\begin{rmk}
It is possible to prove that the norm functor maps continuous metrics to continuous metrics. 
\end{rmk}

\begin{prop}\label{lem:commutewithP} Let $f:Y \to X$ be a finite, flat morphism of geometrically reduced projective $K$-schemes. Suppose we are given an ample line bundle $L$ on $Y$ with a continuous metric $\phi$. Then $P(f^\star  \phi )= f^\star P(\phi)$. In particular, if $P(\phi)$ is continuous, so is $P(f^\star \phi)$. 
\end{prop}

\begin{proof} In the Archimedean case, this is a special case of \cite[Proposition 2.9]{BB}. The statement in the trivially valued case follows from the non-Archimedean non-trivially valued one by comparing with a non-trivially valued field extension as in the previous lemma. We will henceforth suppose $K$ is non-Archimedean non-trivially valued. 

It follows from Proposition~\ref{prop:modelpullback} and Lipschitz continuity of $f^\star$ that $f^\star \left( C^0\cap \PSH(L)\right) \subseteq C^0\cap\PSH(f^\star L)$, from which we infer the inequality $f^\star P(\phi) \leq P(f^\star \phi)$. 

To show the reverse inequality we show that any continuous psh metric $\psi$ on $f^\star L$ with $\psi \leq f^\star \phi$ admits a majoration $\psi \leq f^\star \psi_\infty \leq f^\star \phi$, where $\psi_\infty$ is continuous psh. Since $P$ is non-decreasing we conclude that 
$$\psi \leq f^\star \psi_\infty \leq f^\star P(\phi) \leq P(f^\star \phi)$$
which implies the statement by taking supremum over all continuous psh metrics $\psi.$

\end{proof}
\begin{lem}
Suppose that $f:Y \to X$ is a finite flat morphism of connected projective geometrically reduced $K$-schemes, of degree $e$, and $\psi$ is a  continuous psh metric on $f^\star L $. Inductively define $\psi_0 = \psi$ and $\psi_k = \max \left\{\frac{1}{e}f^\ast \phi_{N({\psi_{k-1}})}, \psi\right\}$. Then all $\psi_k$ are continuous psh, and converge uniformly to $f^\ast \psi_\infty$, where 
$$\psi_\infty(x)= \sup_{y' \in f^{-1}(x)} \psi(y')$$ 
is a continuous psh metric on $\frac{1}{e}N(L)$.
\end{lem}
\begin{proof}
The fact that all the metrics $\psi_k$ are continuous psh follows from Proposition \ref{prop:descnormpsh} together with the statement that continuous psh metrics are closed under max. Without loss of generality we suppose that $e > 1$. Here we can define 
$$
f^\star \psi_\infty(x)= \sup_{y' \in f^{-1}(x) }\left( \psi-f^\star \phi \right)(x)+\phi(f(x))
$$ 
for an auxiliary metric $\phi$ on $L$. It is independent of $\phi$, and $\psi_\infty$ defines a metric on $\frac{1}{e}N(L)$.

By construction we have 
$$
0 \leq \left(f^\star \psi_\infty-\psi_k\right)(y) \leq  f^\star \psi_\infty(y)- \left[\frac{1}{e}  \sum_{y' \in f^{-1}(f(y))}   {\psi_{k-1}}(y') \right ].
$$  Here the sum is computed with taking the multiplicities of the points into account. For any $x \in X^\an$, there is always at least one $y \in f^{-1}(x)$ such that $\psi_k(y) = \operatorname{max}\{\psi_{k-1}, \psi\}(y) = f^\star\psi_\infty(y)$. It follows that an upper bound $M_k$ of $f^\star \psi_\infty-\psi_k$ can be chosen so that $M_k \leq \left( 1- \frac{1}{e} \right) M_{k-1} \leq \left(1-\frac{1}{e}\right)^{k} M_0$. We conclude that $\psi_k \to f^\star \psi_\infty$ uniformly so that $f^\star  \psi_\infty$  is continuous psh, and hence so must $\psi_\infty$ be. 
\end{proof}
\section{Asymptotics of relative volumes}\label{sec:Euler}
This section introduces graded norms and their relative volumes, reviews the basic properties of the Monge--Amp\`ere energy functional, and then proves Theorem A, the main result of this paper.

As before, $K$ denotes a complete valued field. Throughout this section, $X$ is a \emph{geometrically reduced}\footnote{If $K$ is perfect (\eg of characteristic $0$), this is equivalent to being reduced.} projective $K$-scheme, and $n:=\dim X$. 
%
%
%
\subsection{Relative volumes of graded norms}
Let $L$ be any line bundle on $X$, and consider the graded $K$-algebra $R=R(X,L)$ with graded pieces $R_m:=\Hnot(X,mL)$. Set also 
$$
N_m:=\dim R_m=\hnot(X,mL).
$$

\begin{defi} A \emph{graded norm} $\n_\bullet$ on the graded algebra $R$ is defined as a sequence of norms $\n_m$ on the graded pieces $\Hnot(mL)$, $m\in\N$, which is
\begin{itemize}
\item[(i)] \emph{submultiplicative}, \ie
$$
\|s\cdot s'\|_{m+m'}\le\|s\|_m\cdot\|s'\|_{m'}
$$
for all $m,m'\in\N$, $s\in R_m$, $s'\in R_{m'}$;
\item[(ii)] \emph{bounded}, \ie
$$
\dGI\left(\n_m,\n_{m\phi}\right)=O(m). 
$$
for some (hence any) bounded metric $\phi$ on $L$. 
\end{itemize}
\end{defi}

\begin{exam}\label{exam:gradedsup} Each bounded metric $\phi$ on $L$ determines a \emph{graded supnorm} $\n_{\bullet\phi}$, with graded pieces $\n_{m\phi}$. 
\end{exam}

\begin{exam}\label{exam:gradedmodel} If $\cL$ is a model of $L$, then the lattice norms $\n_{\Hnot(m\cL)}$ form a graded sequence, denoted by $\n_{\Hnot(\bullet\cL)}$. Boundedness is a consequence of Theorem~\ref{thm:supmodel}.
\end{exam}

Graded norms are preserved under ground field extension: 

\begin{lem}\label{lem:chenmac} Let $F/K$ be a complete field extension. For any graded norm $\n_\bullet$ on $R=R(X,L)$, the sequence of ground field extensions $(\n_m)_F$ forms a graded norm on $R_F=R(X_F,L_F)$. Further, 
\begin{equation}\label{equ:volgroundbis}
\vol\left((\n_m)_F,(\n'_m)_F\right)=\vol(\n_m,\n'_m)+o(m^{n+1}). 
\end{equation}
\end{lem}
\begin{proof} Submultiplicativity of the sequence $(\n_m)_F$ follows from directly from Definition~\ref{defi:ground}. To see that the sequence is bounded, pick any continuous metric $\phi$ on $L$. By Proposition~\ref{prop:ground}, 
$$
\dGI((\n_m)_F,(\n_{m\phi})_F)=\dGI(\n_m,\n_{m\phi})=O(m),
$$
and it is thus enough to show that $\dGI\left((\n_{m\phi})_F,\n_{m\phi_F}\right)=O(m)$, which follows from Theorem~\ref{thm:supground}. Finally, \eqref{equ:volgroundbis} is a consequence of Proposition~\ref{prop:vol} combined with the estimate
\begin{equation}\label{equ:estimate}
N_m\log N_m=O(m^n\log m)=o(m^{n+1}),
\end{equation}

\end{proof}

The next result is basically due to Chen and Maclean~\cite{CMac}, following the strategy of Witt Nystr\"om~\cite{WN}, itself finding its roots in the work of Zaharjuta~\cite{Zah}. For the convenience of the reader, we review the argument below. 

\begin{thm}\label{thm:chenmac} Assume that $X$ is geometrically integral. Let $L$ be a line bundle on $X$, and pick two graded norms $\n_\bullet,\n'_\bullet$ on $R=R(X,L)$. Then $m^{-(n+1)}\vol(\n_m,\n'_m)$ admits a limit in $\R$. 
\end{thm}

\begin{defi} The \emph{relative volume} of the graded norms $\n_\bullet$, $\n'_\bullet$ is defined as
$$
\vol(\n_\bullet, \n'_\bullet):=\lim_{m\to\infty}\frac{n!}{m^{n+1}}\vol (\n_m, \n'_m).
$$
\end{defi}

\begin{proof}[Proof of Theorem~\ref{thm:chenmac}] Observe first that the result is trivial when $L$ is not big. Set indeed $N_m:=\hnot(mL)$. By the Lipschitz property of relative volumes (Proposition~\ref{prop:vol}), 
$$
\left|\vol(\n_m,\n'_m)\right|\le N_m\dGI(\n_m,\n'_m)=O(m N_m).
$$
If $L$ is not big, then $N_m=o(m^n)$, and hence $m^{-(n+1)}\vol(\n_m,\n'_m)\to 0$. 

We henceforth assume that $L$ is big. By Lemma~\ref{lem:chenmac}, we may further pass to a complete field extension and assume that $K$ is algebraically closed. 

We now follow the approach of~\cite{WN,CMac}, which combines the strategy of Zaharjuta~\cite{Zah}, relying on a multivariate version of Fekete's Lemma on subadditive sequences, with the Okounkov body construction~\cite{LM,KK}. Pick a regular point $p\in X(K)$, a regular sequence $(z_1,\ldots,z_n)$ in $\cO_{X,p}$, and consider the valuation $\nu:K(X)^\times\to(\Z^n,\mathrm{lex})$ defined as follows: by Cohen's structure theorem, every $f\in\cO_{X,p}$ admits a formal power series expansion $f=\sum_{\a\in\N^n} f_\a z^\a$, $f_\a\in K$, and we set
$$
\nu(f):=\min\{\a\in\N^n\mid f_\a\ne 0\}, 
$$
where the min is understood with respect to the lexicographic order. Note that $\nu$ is trivial on $K$, and has center $p$ on $X$. 

The valuation $\nu$ can be naturally evaluated on sections $s\in R_m=\Hnot(mL)$, by setting $\nu(s):=\nu(f)$ with $f\in\cO_{X,p}$ the function corresponding to $s$ in some choice of local trivialization of $L$ at $p$ (the definition being independent of the choice of trivialization). This induces an $\N^n$-filtration on $R_m$, with the key property that its graded pieces 
$$
\gr^\a R_m=\frac{\{s\in R_m\mid\nu(s)\ge\a\}}{\{s\in R_m\mid\nu(s)>\a\}}
$$
have dimension at most $1$ for all $\a\in\N^n$. Thus
$$
\Ga_m:=\nu\left(R_m\setminus\{0\}\right)\subset\N^n
$$
is finite, of cardinality $N_m=\dim R_m$, and $\gr^\a R_m$ is one-dimensional for each $\a\in\Ga_m$ (see for instance~\cite[Lemme 2.11]{Bourbaki}). The fact that $L$ is big implies that the semigroup 
$$
\Ga:=\bigcup_{m\in\N}\left(\{m\}\times\Ga_m\right)\subset\N^{n+1}
$$
generates $\Z^{n+1}$ as a group (cf.~\cite[Proposition 3.3]{Bourbaki}). Further, $\Ga_m$ grows linearly with $m$, \ie there exists $C>0$ such that 
\begin{equation}\label{equ:Gam}
|\a|:=\sum_i|\a_i|\le Cm\text{ for all }\a\in\Ga_m. 
\end{equation}
Indeed, the closure $\bigcup_{m\ge 1}m^{-1}\Ga_m$ in $\R^n$ is a convex body $\D(L)$, the \emph{Okounkov body} of $L$ with respect to $\nu$. 

Next fix a local trivialization $\tau$ of $L$ at $p$. For each $\a\in\Ga_m$, one can find a section $s\in\Hnot(mL)$ with Taylor expansion $s=z^\a+\text{higher order terms}$ (in the trivialization $\tau^m$). The class of $s$ in $\gr^\a R_m$ is independent of the choice of $s$, and hence a canonical generator $s_{m,\a}$ of $\gr^\a R_m$. The norm $\n_m$ induces a subquotient norm on $\gr^\a\Hnot(mL)$, characterized by
\begin{equation}\label{equ:grnm}
\|s_{m,\a}\|_m=\inf\{\|s\|_m\mid s\in \Hnot(mL),\,s=z^\a+\text{higher order terms}\},
\end{equation}
This gives rise to a superadditive function $\Phi:\Ga\to\R$, defined by 
$$
\Phi(m,\a):=-\log\|s_{m,\a}\|_m.
$$
We similarly have a superadditive function $\Phi':\Ga\to\R$ attached to $\n'_\bullet$, and a repeated application of Proposition~\ref{prop:vol}~ (vi), combined with the estimate $N_m\log N_m=o(m^{n+1})$, yields
$$
\vol(\n_m,\n'_m)=\sum_{\a\in\Ga_m}\left(\Phi(m,\a)-\Phi'(m,\a)\right)+o(m^{n+1}), 
$$
and hence
$$
\frac{1}{m^{n+1}}\vol(\n_m,\n'_m)=\frac{1}{m^n}\sum_{\a\in\Ga_m}\left(m^{-1}\Phi(m,\a)-m^{-1}\Phi'(m,\a)\right)+o(1).
$$
Now $\left|m^{-1}\Phi(m,\a)-m^{-1}\Phi'(m,\a)\right|\le m^{-1}\dGI(\n_m,\n'_m)=O(1)$ and $|\Ga_m|=N_m\sim\frac{m^n}{n!}\vol(L)$. By the general convergence result of~\cite[Theorem 4.3]{CMac} (see also~\cite[Theorem 1.3]{WN}), it remains to check that
$$
\sup_{m\ge 1,\,\a\in\Ga_m}m^{-1}\Phi(m,\a)<\infty,
$$
which will then also hold for $\Phi'$, by symmetry. By~\eqref{equ:Gam} and~\eqref{equ:grnm}, this is equivalent to the existence of a uniform constant $C>0$ such that 
\begin{equation}\label{equ:WN}
\log\|s\|_m\ge-C(m+|\a|)
\end{equation}
for all $s\in R_m$ with Taylor expansion $s=z^\a+\text{higher order terms}$. Thanks to the boundedness property of $\n_\bullet$, it is enough to check this when $\n_\bullet=\n_{\bullet\phi}$ is the graded supnorm attached to a continuous metric $\phi$, which we now prove along the lines of~\cite[Lemma 5.4]{WN}. By the $K$-analytic inverse function theorem, the coordinates $(z_i)$ at $p$ induce a $K$-analytic isomorphism of an open neighborhood $U$ of $p$ in $X(K)$ with an open polydisc $\DD(r)^n$ in $K^n$. The section $s$ induces an analytic function $f$ on $U$, with $s=f\tau^m$, and hence $\log|s|_{m\phi}=\log|f|+m\log|\tau|_{\phi}$. The function $\log|\tau|_\phi$ being locally bounded, it will be enough to show that any analytic function $f$ on $\DD(r)^n$ such that
$$
f=z^\a+\text{higher order terms}
$$
satisfies $\sup_{\DD(r)^n}|f|\ge r^{|\a|}$, which follows from a repeated application of the maximum principle in one variable. 
\end{proof}

By~\eqref{equ:volgroundbis}, relative volumes are invariant under ground field extension: 

\begin{prop}\label{prop:gradedvolground} Assume that $X$ is geometrically integral. Let $\n_\bullet$, $\n'_\bullet$ be graded norms on $R$, and $F/K$ be a complete field extension. Then 
$$
\vol\left((\n_\bullet)_F,(\n'_\bullet)_F\right)=\vol(\n_\bullet,\n'_\bullet).
$$
\end{prop}
%
%
%
\subsection{Relative volumes of metrics}
We come back here to the general setting of Section~\ref{sec:Euler}, and thus assume that $X$ is geometrically reduced. Using Theorem~\ref{thm:chenmac}, we prove: 
 
\begin{thm}\label{thm:relvolmetr} Assume that $X$ is geometrically reduced, and let $L$ be a line bundle on $X$. 
\begin{itemize}
\item[(i)] The \emph{volume} 
$$
\vol(L):=\lim_{m\to\infty}\frac{n!}{m^n}\hnot(mL)
$$ 
exists in $\R_{\ge 0}$. 
\item[(ii)] For any two bounded metrics $\phi,\p$ on $L$, the \emph{relative volume} 
$$
\vol(L,\phi,\p):=\lim_{m\to\infty}\frac{n!}{m^{n+1}}\vol(\n_{m\phi},\n_{m\p})
$$
exists in $\R$. 
\end{itemize}
Denote further by $(X_\a)_{\a\in A}$ the set of top-dimensional irreducible components of $X$, and let $\phi_\a$, $\p_\a$ be the pullbacks of $\phi,\p$ to $L_\a:=L|_{X_\a}$. Then  
\begin{equation}\label{equ:voladd}
\vol(L)=\sum_\a\vol\left(L_\a\right)
\end{equation}
and
\begin{equation}\label{equ:volreladd}
\vol(L,\phi,\p)=\sum_\a\vol(L_\a,\phi_\a,\p_\a).
\end{equation}
\end{thm}

\begin{lem}\label{lem:relvolmetr} Using the notation of Theorem~\ref{thm:relvolmetr}, we have
\begin{equation}\label{equ:volbig}
\hnot(X,mL)=\sum_{\a\in A}\hnot(X_\a,mL_\a)+o(m^n)
\end{equation}
and
\begin{equation}\label{equ:relvolbig}
\vol(\n_{m\phi},\n_{m\p})=\sum_{\a\in A}\vol(\n_{m\phi_\a},\n_{m\p_\a})+o(m^{n+1}). 
\end{equation}
\end{lem}
\begin{proof} Set $Y:=\coprod_{\a\in A} X_\a$, and denote by $\mu:Y\to X$ the canonical morphism. For each $m$, we have an exact sequence
$$
0\to\Hnot(X,mL)\to\Hnot(Y,m\mu^\star L)=\bigoplus_{\a\in A}\Hnot(X_\a,mL)\to\Hnot(X,\cF(mL))
$$
with $\cF:=\mu_\star\cO_Y/\cO_X$. Note that $\mu$ is an isomorphism over the complement of $\bigcup_{\a\ne\b} X_\a\cap X_\b$, which has dimension at most $n-1$. Thus $\supp\cF$ has dimension at most $n-1$, and hence $\hnot(X,\cF(mL))=O(m^{n-1})$, which proves~\eqref{equ:volbig}. Under the identification $\Hnot(Y,m\mu^\star L)=\bigoplus_{\a\in A}\Hnot(X_\a,mL_\a)$, the supnorm $\n_{m\mu^\star\phi}$ corresponds to $\max_\a\n_{m\phi_\a}$, and a repeated application of Proposition~\ref{prop:vol}, combined with~\eqref{equ:estimate}, yields~\eqref{equ:relvolbig}. 
\end{proof}

\begin{proof}[Proof of Theorem~\ref{thm:relvolmetr}] By~\eqref{equ:volgroundbis}, we may assume that $K$ is algebraically closed. By Lemma~\ref{lem:relvolmetr}, it is enough to prove the result with $X$ replaced by any of its top-dimensional irreducible component. We may thus assume that $X$ is (geometrically) integral, in which case the result follows from Theorem~\ref{thm:chenmac}. 
\end{proof}
We each $m\in\Z_{>0}$, we note the obvious homogeneity property
\begin{equation}\label{equ:volhom}
\vol(mL,m\phi,m\p)=m^{n+1}\vol(L,\phi,\p).
\end{equation}
As a result, we can make sense of relative volumes of metrics on a $\Q$-line bundle $L$. As usual, we say that $L$ is \emph{big} if $\vol(L)>0$. 

\begin{cor}\label{cor:big} Let $L$ be a $\Q$-line bundle on $X$. 
\begin{itemize}
\item[(i)] $L$ is big iff $L|_{X_\a}$ is big for some top-dimensional irreducible component $X_\a$ of $X$. 
\item[(ii)] If $L$ is nef, then $\vol(L)=(L^n)$.
\end{itemize}
\end{cor}
\begin{proof} By homogeneity, we may assume that $L$ is an honest line bundle. (i) is then a direct consequence of~\eqref{equ:voladd}. Assume that $L$ is nef. Since $X_\a$ is irreducible, it is well-known that the volume of the nef line bundle $L|_{X_\a}$ is given by $\vol(L|_{X_\a})=(L|_{X_\a})^n=c_1(L)^n\cdot[X_\a]$. By~\eqref{equ:voladd}, we get
$$
\vol(L)=\sum_\a c_1(L)^n\cdot[X_\a]=c_1(L)^n\cdot[X]=(L^n).
$$
\end{proof}

\begin{prop}\label{prop:volmetr} The following properties hold for all bounded metrics on a $\Q$-line bundle $L$: 
\begin{itemize}
\item[(i)] \emph{cocycle formula:} $\vol(L,\phi_1,\phi_2)+\vol(L,\phi_2,\phi_3)+\vol(L,\phi_3,\phi_1)=0$; 
\item[(ii)] \emph{monotonicity:} $\phi\le\phi'\Longrightarrow\vol(L,\phi,\p)\le\vol(L,\phi',\p)$; 
\item[(iii] \emph{scaling:} $\vol(L,\phi+c,\p)=\vol(L,\phi,\p)+\vol(L)c$ for $c\in\R$; 
\item[(iv)] \emph{Lipschitz continuity:}
$$
\left|\vol(L,\phi,\p)-\vol(L,\phi',\p')\right|\le\vol(L)\left(\sup|\phi-\phi'|+\sup|\p-\p'|\right). 
$$
\item[(v)] \emph{birational invariance:} for any projective birational morphism $\mu:X'\to X$, we have
$$
\vol(\mu^\star L,\mu^\star\phi,\mu^\star\p)=\vol(L,\phi,\p); 
$$
\item[(vi)] \emph{envelopes:} if $L$ is semiample, then the psh envelopes $\env(\phi)$, $\env(\p)$ satisfy
$$
\vol(L,\phi,\p)=\vol(\env(\phi),\env(\p)).
$$ 
\end{itemize}
\end{prop}
In particular, if $L$ is not big, \ie $\vol(L)=0$, then $\vol(L,\phi,\p)=0$ for all bounded metrics on $L$. 

\begin{proof} We may assume that $L$ is a line bundle, by homogeneity. Properties (i) and (ii) follow immediately from the analogous properties for relative volumes of norms. Further, $\n_{m(\phi+c)}=e^{-mc}\n_{m\phi}$, hence 
$$
\vol(\n_{m(\phi+c)},\n_{m\p})=\vol(\n_{m\phi},\n_{m\p})+m\hnot(mL)c,
$$
by Proposition~\ref{prop:vol}. Since $n!\hnot(mL)/m^n\to\vol(L)$, this yields (iii), and (iv) is a formal consequence of (i), (ii) and (iii). To see (v), note that the embedding 
$$
\Hnot(X,mL)\hookrightarrow\Hnot(X',m\mu^\star L)
$$
is an isometry with respect to the supnorms. By the projection formula, the quotient 
$$
W_m:=\Hnot(X',m\mu^\star L)/\Hnot(X,mL)
$$
injects into $\Hnot(X,\cF(mL))$, where $\cF:=\mu_\star\cO_{X'}/\cO_X$. The support of the latter sheaf is contained in the image of the exceptional locus of $\mu$, and hence has dimension at most $n-1$, and hence
$\dim W_m=o(m^n)$. By Proposition~\ref{prop:vol}, we infer
$$
\left|\vol\left(\n_{m\mu^\star\phi},\n_{m\mu^\star\p}\right)-\vol\left(\n_{m\phi},\n_{m\p}\right)\right|
$$
$$
\le(\dim W_m)\dGI\left(\n_{m\mu^\star\phi},\n_{m\mu^\star\p}\right)+O(N_m\log N_m)
$$
$$
\le (\dim W_m)m\sup|\phi-\p|+O(N_m\log N_m)=o(m^{n+1}),
$$
which implies (v). Assume finally that $L$ is semiample. Corollary~\ref{cor:QP} shows that $\n_{m\phi}=\n_{m\env(\phi)}$ and $\n_{m\p}=\n_{m\env(\p)}$ for all $m$, hence (vi). 

\end{proof}

\begin{prop}\label{prop:volground} Let $\phi,\p$ be \emph{continuous} metrics on $L$, $F/K$ a complete field extension, and $\phi_F,\p_F$ the pullbacks of $\phi,\p$ to the base change $L_F$. Then
$$
\vol(L_F,\phi_F,\p_F)=\vol(L,\phi,\p).
$$
\end{prop}

\begin{proof} By Theorem~\ref{thm:supground}, we have 
$$
\dGI\left((\n_{m\phi})_F,\n_{m\phi_F}\right)=o(m),\,\,\,\,\dGI\left((\n_{m\p})_F,\n_{m\p_F}\right)=o(m).
$$
By the Lipschitz property of relative volumes and Lemma~\ref{lem:chenmac}, we infer
$$
\vol\left(\n_{m\phi_F},\n_{m\p_F}\right)=\vol\left((\n_{m\phi})_F,(\n_{m\p})_F\right)+o(m^{n+1})
$$
$$
=\vol\left(\n_{m\phi},\n_{m\p}\right)+o(m^{n+1}),
$$
and the result now follows. 
\end{proof}
%
%
%
\subsection{Monge--Amp\`ere energy} 
In this section, $X$ is geometrically reduced, and $L$ is a semiample $\Q$-line bundle on $X$. Since $L$ is in particular nef, we have 
$$
V:=\vol(L)=(L^n),
$$
by Corollary~\ref{cor:big}. 

\begin{defi} Let $\phi,\p$ be continuous psh metrics on $L$. We define the \emph{relative Monge--Amp\`ere energy}\footnote{Note that the present normalization, which is more convenient for the purpose of this paper, is not uniform accross the literature.}of $\phi,\p$ as
\begin{equation}\label{equ:en}
\en(\phi,\p):=\frac{1}{n+1}\sum_{j=0}^n\int (\phi-\p)(dd^c\phi)^j\wedge(dd^c\p)^{n-j}\wedge\d_X. 
\end{equation}
\end{defi}

Recall that $\phi-\p$ is a continuous function on $X^\an$, which may thus be integrated against each positive Radon measure $(dd^c\phi)^j\wedge(dd^c\p)^{n-j}\wedge\d_X$. By~\eqref{equ:changemetr}, we have
\begin{equation}\label{equ:endel}
\en(\phi,\p)=\frac{1}{n+1}\left(\langle\phi^{n+1}\rangle-\langle\p^{n+1}\rangle\right). 
\end{equation}
Given a continuous psh metric $\p$, the functional $\phi\mapsto\en(\phi,\p)$ is characterized as the unique antiderivative of the Monge--Amp\`ere operator $\phi\mapsto (dd^c\phi)^n$ that vanishes at $\p$, in the sense that
\begin{equation}\label{equ:endiff}
\frac{d}{dt}\bigg|_{t=0}\en((1-t)\phi+t\phi',\p)=\int_{X^\an}(\phi'-\phi)(dd^c\phi)^n
\end{equation}
for any two continuous psh metrics $\phi,\phi'$, see~\cite[\S 3.8]{trivval}. 

\begin{prop}\label{prop:en} The Monge--Amp\`ere energy satisfies the following properties.
\begin{itemize}
\item[(i)] cocycle formula: $\en(\phi_1,\phi_2)+\en(\phi_2,\phi_3)+\en(\phi_3,\phi_1)=0$; 
\item[(ii)] monotonicity: $\phi\le\phi'\Longrightarrow\en(\phi,\p)\le\en(\phi',\p)$; 
\item[(iii)] scaling: $\en(\phi+c,\p)=\en(\phi,\p)+V c$ for $c\in\R$; 
\item[(iv)] homogeneity: $\en(a\phi,a\p)=a^{n+1}\en(\phi,\p)$ for $a\in\Z_{>0}$.
\item[(v)] Lipschitz continuity: 
$$
\left|\en(\phi,\p)-\en(\phi',\p')\right|\le V\left(\sup|\phi-\phi'|+\sup|\p-\p'|\right).
$$
\item[(vi)] for each complete field extension $F/K$ we have 
$$
\en(\phi_F,\p_F)=\en(\phi,\p).
$$
\item[(vii)] birational invariance: for any projective birational morphism $\mu:X'\to X$ we have 
$$
\en(\mu^\star\phi,\mu^\star\p)=\en(\phi,\p).
$$
\end{itemize}
\end{prop}
\begin{proof} (i) follows from~\eqref{equ:endiff}, and (ii)--- (iv) follow directly from~\eqref{equ:en}, using that
$$
\int(dd^c\phi)^j\wedge(dd^c\p)^{n-j}\wedge\d_X=(L^n)=V.
$$
(v) is a formal consequence of (ii) and (ii), and (vi) follows from~\eqref{equ:en} and the compatibility of mixed Monge--Amp\`ere measures with ground field extension, cf.~Proposition~\ref{prop:mixedMA}. Finally, (vii) is a consequence of~\eqref{equ:en} and the projection formula $\mu_\star\d_{X'}=\d_X$. 
\end{proof}

%
%
\subsection{Proof of Theorem A}
The following result corresponds to Theorem A in the introduction. It was first established in~\cite{BB} when $K$ is Archimedean, and in~\cite{BG+} when $K$ is discretely valued.

\begin{thm}\label{thm:Euler} Let $X$ be a geometrically reduced projective $K$-scheme, and $L$ be a semiample $\Q$-line bundle on $X$. For any two continuous psh metrics $\phi,\p$ on $L$, we then have
$$
\vol(L,\phi,\p)=\en(\phi,\p).
$$
\end{thm}

\begin{cor}\label{cor:Euler} Let $\phi,\p$ be arbitrary continuous metrics on $L$. Let $\nu:\tilde X\to X$ be the normalization morphism, and assume that Conjecture~\ref{conj:env} holds. Then 
$$
\vol(L,\phi,\p)=\en\left(\env(\nu^\star\phi),\env(\nu^\star\p)\right).
$$
\end{cor}
\begin{proof} By Proposition~\ref{prop:volmetr}, $\vol(L,\phi,\p)=\vol(\nu^\star\phi,\nu^\star\p)=\vol\left(\env(\nu^\star\phi),\env(\nu^\star\p)\right)$. Since $\tilde X$ is normal, Conjecture~\ref{conj:env} implies that $\env(\nu^\star\phi)$ and $\env(\nu^\star\p)$ are continuous and psh, and Theorem~\ref{thm:Euler} yields the result. 

\end{proof}

\begin{lem}\label{lem:volmodel} Assume that $K$ is non-Archimedean, and let $\cL$, $\cM$ be models of $L$, with associated model metrics $\phi_\cL,\phi_{\cM}$ and graded norms $\n_{\Hnot(\bullet\cL)}$, $\n_{\Hnot(\bullet\cM)}$ (cf.~Example~\ref{exam:gradedmodel}). Then 
$$
\vol\left(L,\phi_\cL,\phi_{\cM}\right)=\lim_{m\to\infty}\frac{n!}{m^{n+1}}\vol\left(\n_{\Hnot(m\cL)},\n_{\Hnot(m\cM)}\right).
$$
\end{lem}

\begin{proof} By Theorem~\ref{thm:supmodel}, we have 
$$
\dGI\left(\n_{m\phi_\cL},\n_{\Hnot(m\cL)}\right)=O(1),\,\,\,\,\,\dGI\left(\n_{m\phi_{\cM}},\n_{\Hnot(m\cM)}\right)=O(1).
$$
By Lipschitz continuity of relative volumes of norms (Proposition~\ref{prop:vol}), this yields
$$
\vol\left(\n_{m\phi_{\cL}},\n_{m\phi_{\cM}}\right)=\vol\left(\n_{\Hnot(m\cL)},\n_{\Hnot(m\cM)}\right)+O(N_m),
$$
which implies the result. 
\end{proof}

\begin{proof}[Proof of Theorem~\ref{thm:Euler}] We claim that it is enough to prove the result when $K$ is algebraically closed and nontrivially valued, $X$ is irreducible, and $L$ is an ample line bundle. The first condition can be reached by passing to an appropriate complete field extension of $K$, by invariance of relative volumes and the Monge--Amp\`ere energy under ground field extension (Proposition~\ref{prop:volground} and Proposition~\ref{prop:en}). By Lemma~\ref{lem:chenmac}, we may further assume that $X$ is irreducible and $L$ is big. By Lemma~\ref{lem:Stein}, there exists a birational morphism $f:X\to Y$ and an ample $\Q$-line bundle $A$, unique up to isomorphism, such that $L=f^\star A$ and $f_\star\cO_X=\cO_Y$, and $\phi,\p$ descend to continuous psh metrics on $A$. By birational invariance of relative volumes (Proposition~\ref{prop:volmetr}) and of the Monge--Amp\`ere energy (Proposition~\ref{prop:en}), we may thus replace $L$ with $A$ and assume that $L$ is ample. By homogeneity~\eqref{equ:volhom}, we may finally assume that $L$ is an honest (ample) line bundle, which concludes the proof of the claim. 

Consider first the Archimedean case, \ie $K=\C$. By birational invariance of relative volumes and of the Monge--Amp\`ere energy, we can replace $X$ with a resolution of singularities. Then $X$ is smooth, $L$ is big and semiample, and the result is then a special case of~\cite[Theorem A]{BB} (which deals with an arbitrary big line bundle). 

We now assume that $K$ is non-Archimedean. We claim that it is enough to prove Theorem~\ref{thm:Euler} when $\phi,\p$ are model metrics determined by ample models $\cL,\cM$ of $L$ on some model $\cX$. By Theorem~\ref{thm:pshZhang}, $\phi$ and $\p$ are uniform limits of model metrics determined by nef $\Q$-models $\cL$, $\cM$ of $L$. Thanks to the Lipschitz continuity of relative volumes (Proposition~\ref{prop:volmetr}) and of the Monge--Amp\`ere energy (Proposition~\ref{prop:en}), it is thus enough to prove the result for such model metrics. After pulling-back to a higher model, we can assume that $\cL$, $\cM$ are determined on the same model $\cX$ of $X$, and also that $L$ further extends to an ample $\Q$-line bundle $\cH$ on $\cX$ (cf.~\cite[Lemma 4.12]{GM}). Denote by $D,E$ the vertical $\Q$-Cartier divisors on $\cX$ such that $\cH-\cL=D$, $\cH-\cM=E$. For each $\d\in\Q\cap(0,1)$, the $\Q$-line bundles 
$$
\cL_\d:=(1-\d)\cL+\d\cH=\cL+\d D,\,\,\,\,\cM_\d:=(1-\d)\cM+\d\cH=\cM+\d E
$$ 
are ample, and the model metrics $\phi_{\cL_\d}=\phi_\cL+\d\phi_D$, $\phi_{\cM_\d}=\phi_{\cL}+\d\phi_{E}$ they determine converge uniformly to $\phi_\cL,\phi_{\cM}$. Replacing $\cL$, $\cM$ with $\cL_\d$, $\cM_\d$, we may thus assume that $\cL$, $\cM$ are ample on $\cX$. After replacing $L$ with a multiple, we can finally assume that $\cL$, $\cM$ are honest line bundles, by homogeneity of $\vol$ and $\en$, which proves the claim above. 

Suppose thus $\phi=\phi_{\cL}$, $\p=\phi_{\cM}$ with $\cL,\cM$ ample models of $L$ on a model $\cX$ of $X$. By Lemma~\ref{lem:volmodel}, 
\begin{equation}\label{equ:volH0}
\vol(L,\phi,\p)=\lim_{m\to\infty}\frac{n!}{m^{n+1}}\vol\left(\n_{\Hnot(m\cL)},\n_{\Hnot(m\cM)}\right).
\end{equation}
By Serre vanishing, the higher cohomology of $m\cL$ and $m\cM$ vanishes for $m\gg 1$, and Corollary~\ref{cor:KM} yields Knudsen--Mumford expansions
$$
\det \Hnot(m\cL)=\frac{m^{n+1}}{(n+1)!}\langle\cL^{n+1}\rangle+O(m^n)
$$
and
$$
\det \Hnot(m\cM)=\frac{m^{n+1}}{(n+1)!}\langle\cM^{n+1}\rangle+O(m^n), 
$$
as $\Q$-line bundles on $\spec K^\circ$. Thus
$$
\vol\left(\n_{\Hnot(m\cL)},\n_{\Hnot(m\cM)}\right)=\log\frac{\det\n_{\Hnot(m\cM)}}{\det\n_{\Hnot(m\cL)}}
$$
$$
=\phi_{\det \Hnot(m\cL)}-\phi_{\det \Hnot(m\cM)}=\frac{m^{n+1}}{(n+1)!}\left(\phi_{\langle\cL^{n+1}\rangle}-\phi_{\langle\cM^{n+1}\rangle}\right)+O(m^n). 
$$
By Theorem~\ref{thm:delmetrNA}, we have $\phi_{\langle\cL^{n+1}\rangle}=\langle\phi^{n+1}\rangle$ and $\phi_{\langle\cM^{n+1}\rangle}=\langle\p^{n+1}\rangle$. We thus get as desired
$$
\vol(L,\phi,\p)=\frac{1}{n+1}\left(\langle\phi^{n+1}\rangle-\p^{n+1}\rangle\right)=\en(\phi,\p),
$$
by~\eqref{equ:endel}. 
\end{proof}

%
%
\section{Transfinite diameter and Fekete points}\label{sec:Feketetrans}
%
%

Following the strategy developed in~\cite{BB,BBW} in the (complex) Archimedean case, we rely on Theorem~\ref{thm:Euler} to show the existence of transfinite diameters, and then use it to establish an equidistribution result for Fekete points, assuming a differentiability property that holds under appropriate assumptions on $K$ and $X$, by~\cite{nama,BG+,trivval}. 

As in the previous section, $X$ is a geometrically reduced, projective scheme over a complete valued field $K$. 

%
%
\subsection{Existence of the transfinite diameter}
Let $L$ be a line bundle on $X$. Set $N:=\hnot(X,L)$, and define the \emph{Vandermonde embedding} 
$$
\Psi:\det \Hnot(X,L)\hookrightarrow \Hnot(X^N,L^{\boxtimes N})
$$
as the composition of the antisymmetrization operator 
\begin{equation*}
\begin{split}
\det \Hnot(X,L) & \hookrightarrow \Hnot(X,L)^{\otimes N}\\
s_1 \wedge \dots \wedge s_N & \mapsto \sum_{\sigma \in\mathfrak{S}_N} (-1)^{\Op{sgn} \sigma} s_{\sigma(1)} \otimes \dots \otimes s_{\sigma(N)}
\end{split}
\end{equation*} 
with the canonical isomorphism $\Hnot(X,L)^{\otimes N}\simeq \Hnot(X^N,L^{\boxtimes N})$. Given a basis $(s_1,\ldots,s_N)$ of $\Hnot(L)$, $\Psi(s_1\wedge\ldots\wedge s_N)$ can be more informally written as the Vandermonde (or Slater) determinant 
$$
\Psi(s_1\wedge\ldots\wedge s_N)(x_1,\ldots,x_N)=\det(s_i(x_j))_{1\le i,j\le N}. 
$$
\begin{defi} Let $\phi$ be a continuous metric on $L$, and $\n$ be a norm on $\Hnot(L)$. We define the \emph{Vandermonde function} of $\phi$ relative to $\n$ as 
$$
V_{\phi,\n}:=\frac{|\Psi(\om)|_{\phi^{\boxtimes N}}}{\det\|\om\|}\in  \cz\left(\left(X^N\right)^\an\right),
$$
and the \emph{diameter of $\phi$ relative to $\n$} as
$$
\d\left(\phi,\n\right):=\sup_{\left(X^N\right)^\an} V_{\phi,\n}=\frac{\|\Psi(\om)\|_{\phi^{\boxtimes N}}}{\det\|\om\|}.
$$
\end{defi}
Here $\om$ is a generator of $\det \Hnot(L)$, the definition being independent of the choice of $\om$. 

\begin{rmk}\label{rmk:projNA} While the canonical map $(X^N)^\an\to(X^\an)^N$ is of course a homeomorphism in the Archimedean case, it is merely continuous and surjective in general in the non-Archimedean case. 
\end{rmk}

Using Theorem~\ref{thm:chenmac}, we are going to establish the following existence result for transfinite diameters. 

\begin{thm}\label{thm:trans} For any two continuous metrics $\phi,\p$ on $L$, the \emph{transfinite diameter} 
$$
\d_\infty\left(\phi,\p\right):=\lim_{m\to\infty}\d\left(m\phi,\n_{m\p}\right)^{n!/m^{n+1}}
$$
exists in $\R_+$. Further, 
\begin{equation}\label{equ:trans}
\log\d_\infty\left(\phi,\p\right)=\vol(L,\p,\phi). 
\end{equation}
\end{thm}

Combining~\eqref{equ:trans} with Theorem~\ref{thm:Euler} and Corollary~\ref{cor:Euler}, we get: 

\begin{cor}\label{cor:trans} Assume that $L$ is semiample, and let $\phi,\p$ be continuous metrics on $L$.
\begin{itemize}
\item[(i)] If $\phi,\p$ are psh, then $\d_\infty(\phi,\p)=\exp\en(\p,\phi)$. 

\item[(ii)] If Conjecture~\ref{conj:env} holds, then 
$$
\d_\infty(\phi,\p)=\exp\en(\env(\nu^\star\p),\env(\nu^\star\phi)),
$$
with $\nu:\tX\to X$ the normalization morphism. 
\end{itemize}
\end{cor}

%


%

For each $m$, denote by 
$$
\Psi_m:\det \Hnot(mL)\hookrightarrow \Hnot\left((mL)^{\boxtimes N_m}\right)
$$
the Vandermonde embedding. Via $\Psi_m$, the supnorm $\n_{(m\phi)^{\boxtimes N_m}}$ restricts to a norm on the line $\det \Hnot(mL)$, which we denote by $\Psi_m^\star \n_{(m\phi)^{\boxtimes N_m}}$. The key fact leading to Theorem~\ref{thm:trans} is the following estimate. 

\begin{lem}\label{lem:trans} For each $\phi\in  \cz(L)$ we have 
$$
\dGI\left(\det\n_{m\phi},\Psi_m^\star \n_{(m\phi)^{\boxtimes N_m}}\right)=o(m^{n+1}).
$$ 
\end{lem}
Observe first that $\dGI\left(\det\n_{m\phi},\Psi_m^\star \n_{(m\phi)^{\boxtimes N_m}}\right)$ is a Lipschitz continuous function of $\phi\in  \cz(L)$, with Lipschitz constant $O(m^{n+1})$. It is thus enough to prove the result for $\phi$ in a dense subset of $ \cz(L)$. The proof will proceed by comparison with certain pure diagonalizable norms, an $L^2$ norm in the Archimedean case, and a lattice norm in the non-Archimedean. 

\begin{lem}\label{lem:Psiarch} Assume $K$ is Archimedean, pick a continuous metric $\phi$ on $L$ and a smooth volume form $\mu$ on $X$, and denote by $\n_{\mu,\phi}$ and $\n_{\mu^N,\phi^{\boxtimes N}}$ the induced $L^2$-norms on $\Hnot(L)$ and $\Hnot(L^{\boxtimes N})$. Then 
$$
\Psi^\star \n_{\mu^N,\phi^{\boxtimes N}}=\sqrt{N!}\det\n_{\mu,\phi}
$$
as norms on $\det \Hnot(L)$. 
\end{lem}
\begin{proof} The statement is equivalent to~\cite[Lemma 5.3]{BB}, and goes as follows. By Fubini, the $L^2$-norm $\n_{\mu^N,\phi^{\boxtimes N}}$ on $\Hnot(L^{\boxtimes N})$ corresponds to the tensor norm $\n_{\mu,\phi}^{\otimes N}$ under the isomorphism $\Hnot(L^{\boxtimes N})\simeq \Hnot(L)^{\otimes N}$. If $(s_i)$ is an orthonormal basis of $\Hnot(L)$ with respect to  $\n_{\mu,\phi}$, then the tensors $s_{i_1}\otimes\dots\otimes s_{i_N}$ form an orthonormal basis of $\Hnot(L)^{\otimes N}$ with respect to $\n_{\mu,\phi}^{\otimes N}$. This implies that the norm of $s_1\wedge\ldots\wedge s_N$ under the anti-symmetrization operator $\det \Hnot(L)\hookrightarrow \Hnot(L)^{\otimes N}$ has squared-norm equal to $N!$, and the result follows. 
\end{proof}

\begin{lem}\label{lem:PsiNA} Assume $K$ is non-Archimedean. Let $\cL$ be a model of $L$, and $\n_{\Hnot(\cL)}$, $\n_{\Hnot\left(\cL^{\boxtimes N}\right)}$ be the induced lattice norms on $\Hnot(L)$ and $\Hnot\left(L^{\boxtimes N}\right)$. Then 
$$
\Psi^\star \n_{\Hnot\left(\cL^{\boxtimes N}\right)}=\det\n_{\Hnot(\cL)}.
$$ 
\end{lem}
\begin{proof} The isomorphism $\Hnot(\cX^N,\cL^{\boxtimes N})\simeq \Hnot(\cX,\cL)^{\otimes N}$ shows that $\n_{\Hnot\left(\cL^{\boxtimes N}\right)}$ corresponds to the tensor norm $\n_{\Hnot(\cL)}^{\otimes N}$ under the isomorphism $\Hnot\left(L^{\boxtimes N}\right)\simeq \Hnot(L)^{\otimes N}$. On the other hand, if $(s_i)$ is an orthonormal basis of $\Hnot(L)$ with respect to $\n_{\Hnot(\cL)}$, then the tensors $s_{i_1}\otimes\dots\otimes s_{i_N}$ form an orthonormal basis of $\Hnot(L)^{\otimes N}$ with respect to $\n_{\Hnot(\cL)}^{\otimes N}$, which implies this time that the anti-symmetrization operator $\det \Hnot(L)\hookrightarrow \Hnot(L)^{\otimes N}$ is an isometric embedding with respect to $\det\n_{\Hnot(\cL)}$ and $\n_{\Hnot(\cL)}^{\otimes N}$.  
\end{proof}

\begin{proof}[Proof of Lemma~\ref{lem:trans}] Assume first that $K$ is Archimedean, and pick a smooth volume form $\mu$. By the Bernstein-Markov inequality~\cite[Lemma 3.2]{BB}, the supnorm $\n_{m\phi}$ and the $L^2$-norm $\n_{\mu,m\phi}$ on $\Hnot(mL)$ satisfy 
$$
\dGI\left(\n_{m\phi},\n_{\mu,m\phi}\right)=o(m),
$$
and hence 
\begin{equation}\label{equ:transA1}
\dGI\left(\det\n_{m\phi},\det\n_{\mu,m\phi}\right)=\left|\vol\left(\n_{m\phi},\n_{\mu,m\phi}\right)\right|=o(m^{n+1}),
\end{equation}
by Lipschitz continuity of relative volumes. As in~\cite[Step 2, p.378]{BB}, a successive application of the Bernstein--Markov inequality in each variable similarly shows that the induced supnorm $\n_{(m\phi)^{\boxtimes N_m}}$ and $L^2$-norm $\n_{\mu^{N_m},(m\phi)^{\boxtimes N_m}}$ on $\Hnot((mL)^{\boxtimes N_m})$ satisfy 
$$
\dGI\left(\n_{(m\phi)^{\boxtimes N_m}},\n_{\mu^{N_m},(m\phi)^{\boxtimes N_m}}\right)=o(m^{n+1}),
$$
and hence
\begin{equation}\label{equ:transA2}
\dGI\left(\Psi_m^\star \n_{(m\phi)^{\boxtimes N_m}},\Psi_m^\star \n_{\mu^{N_m},(m\phi)^{\boxtimes N_m}}\right)=o(m^{n+1})
\end{equation}
as well. Finally, since $\log\left(N_m!\right)=O(m^n\log m)=o(m^{n+1})$, Lemma~\ref{lem:Psiarch} yields 
$$
\dGI\left(\det\n_{\mu,m\phi},\Psi_m^\star \n_{\mu^{N_m},(m\phi)^{\boxtimes N_m}}\right)=o(m^{n+1}),
$$
which combines with \eqref{equ:transA1} and \eqref{equ:transA2} to yield the desired estimate 
$$
\dGI\left(\det\n_{m\phi},\Psi_m^\star \n_{(m\phi)^{\boxtimes N_m}}\right)=o(m^{n+1}).
$$
Assume now that $K$ is non-Archimedean. Arguing as in \S\ref{sec:Euler}, we may assume after ground field extension that $K$ is nontrivially valued, so that model metrics are dense in $ \cz(L)$. As already noticed, 
$\dGI\left(\Psi_m^\star \n_{(m\phi)^{\boxtimes N_m}},\det\n_{m\phi}\right)$
is a Lipschitz continuous function of $\phi\in  \cz(L)$, with Lipschitz constant $O(m^{n+1})$; by density, it is thus enough to prove the result when $\phi=\phi_\cL$ is a model metric, determined by a $\Q$-line bundle $\cL$ extending $L$ on some projective model $\cX$ of $X$. After replacing $\cX$ with a higher model, we may assume that $L$ also admits a model $\cM$ determined on $\cX$ (Lemma~\ref{lem:modelexist}). As in the proof of Theorem~\ref{thm:supground}, fix $a\ge 1$ such that $a\cL$ is a line bundle, and write $m=qa+r$ with $q,r\in\N$ and $r<a$. Since $a\cL$ and $\cM$ are line bundles on $\cX$, Theorem~\ref{thm:supmodelbis} shows the existence of $C>0$ independent of $m$ such that
$$
\dGI\left(\n_{(qa\phi+r\phi_\cM)^{\boxtimes N_m}},\n_{\Hnot\left((qa\cL+r\cM)^{\boxtimes N_m}\right)}\right)=O(N_m).
$$
As $\phi-\phi_\cM$ and $r$ are bounded, it follows that
$$
\dGI\left(\n_{(m\phi)^{\boxtimes N_m}},\n_{\Hnot\left((qa\cL+r\cM)^{\boxtimes N_m}\right)}\right)=O(N_m),
$$
and hence 
\begin{equation}\label{equ:transNA1}
\dGI\left(\Psi_m^\star \n_{(m\phi)^{\boxtimes N_m}},\Psi_m^\star \n_{\Hnot\left((qa\cL+r\cM)^{\boxtimes N_m}\right)}\right)=O(N_m)
\end{equation}
By Lemma~\ref{lem:PsiNA}, we have 
$$
\Psi_m^\star \n_{\Hnot\left((qa\cL+r\cM)^{\boxtimes N_m}\right)}=\det\n_{\Hnot\left(qa\cL+r\cM\right)}.
$$
On the other hand, Theorem~\ref{thm:supmodel} yields
$$
\dGI\left(\n_{qa\phi+r\phi_\cM},\n_{\Hnot\left(qa\cL+r\cM\right)}\right)=O(N_m),
$$
hence
$$
\dGI\left(\det\n_{m\phi},\det\n_{\Hnot\left(qa\cL+r\cH\right)}\right)=O(N_m)
$$
by boundedness of $\phi-\phi_\cM$, and we conclude that
$$
\dGI\left(\Psi_m^\star \n_{(m\phi)^{\boxtimes N_m}},\det\n_{m\phi}\right)=O(N_m)
$$
when $\phi$ is a model metric. 
\end{proof}

\begin{proof}[Proof of Theorem~\ref{thm:trans}] For any choice of generator $\om\in\det \Hnot(mL)$, we have 
$$
\sup_{\left(X^N\right)^\an} V_{m\phi,\n_{m\p}}=\frac{\|\Psi_m(\om)\|_{(m\phi)^{\boxtimes N_m}}}{\det\|\om\|_{m\p}}
$$
$$
=\left(\frac{\|\Psi_m(\om)\|_{(m\phi)^{\boxtimes N_m}}}{\det\|\om\|_{m\phi}}\right)\left(\frac{\det\|\om\|_{m\phi}}{\det\|\om\|_{m\p}}\right). 
$$
By Lemma~\ref{lem:trans}, we infer
$$
\frac{n!}{m^{n+1}}\log\d(m\phi,\n_m)=\frac{n!}{m^{n+1}}\vol\left(\n_{m\p},\n_{m\phi}\right)+o(1),
$$
and hence 
$$
\frac{n!}{m^{n+1}}\log\d(m\phi,\n_m)\to\vol(\n_\bullet,\n_{\bullet\phi}).
$$
\end{proof}
%
%
\subsection{Equidistribution of Fekete points}

\begin{defi} Let $\phi\in  \cz(L)$ be a continuous metric on a line bundle $L$. A \emph{Fekete configuration} for $\phi$ is a point $P\in(X^N)^\an$ such that 
$$
\|\Psi(\om)\|_{\phi^{\boxtimes N}}=\sup_{(X^N)^\an}|\Psi(\om)|_{\phi^{\boxtimes N}}
$$
is achieved at $P$ for some, hence any, generator $\om\in\det \Hnot(L)$. 
\end{defi}
In terms of the Vandermonde function $V_{\phi,\n}$ relative to any given norm $\n$ on $\Hnot(L)$, $P\in (X^N)^\an$ is a Fekete configuration iff
$$
\sup_{(X^N)^\an}V_{\phi,\n}=V_{\phi,\n}(P).
$$
Our final result is an equidistribution result for Fekete configurations of $m\phi$ as $m\to\infty$, first established in the Archimedean case in~\cite{BBW}. In order to cover various cases in one stroke, we introduce the following terminology. 

\begin{defi} Let $L$ be a semiample line bundle, and $\phi$ be a continuous metric on $L$. For brevity, we shall say that \emph{differentiability holds at $\phi$} if: 
\begin{itemize}
\item[(i)] the psh envelope $\env(\phi)$ is continuous (hence psh);
\item[(ii)] for all $f\in \cz(X^\an)$ we have 
\begin{equation}\label{equ:diffvol}
\frac{d}{dt}\bigg|_{t=0}\vol(L,\phi+t f,\phi)=\int f\,\left(dd^c\env(\phi)\right)^n\wedge\d_X.
\end{equation}
\end{itemize}
\end{defi} 
Differentiability is known to hold at all continuous metrics when $X$ is smooth, $L$ is ample, and one of the following conditions is satisfied: 
\begin{itemize}
\item $K$ is Archimedean~\cite{BB}; 
\item $K$ is non-Archimedean, trivially or discretely valued, of residue characteristic zero~\cite{siminag,trivval};
\item $K$ is discretely valued of characteristic $p$, $(X,L)$ is defined over a function field of
transcendence degree $d$, and resolution of singularities is assumed in dimension $n+d$~\cite{GJKM,BG+}.
\end{itemize}

\begin{thm}\label{thm:equi} Let $L$ be a big and semiample line bundle, of volume $V:=\vol(L)=(L^n)$. Let $\phi$ be a continuous metric on $L$, and assume that differentiability holds at $\phi$. For each $m\gg 1$, pick a Fekete configuration $P_m\in (X^{N_m})^\an$ for $m\phi$. Then $P_m$ equidistributes to the probability measure 
$$
\mu_\phi:=V^{-1}\left(dd^c\env(\phi)\right)^n\wedge\d_X.
$$
\end{thm}
In particular, Fekete configurations for $m\phi$ become asymptotically unique as $m\to\infty$. Equidistribution means that
$$
\int_{X} f\,\d_{P_m}\to\int_{X}f\,\mu_\phi
$$
for all $f\in  \cz(X^\an)$. Here $\d_P$ denotes the averaging measure over $P\in(X^N)^\an$, or rather its image in $(X^\an)^N$ (see Remark~\ref{rmk:projNA}). 

The proof of Theorem~\ref{thm:equi} follows the strategy of~\cite{BBW}, itself inspired by a variational argument due to Szpiro--Ullmo--Zhang~\cite{SUZ}.

\begin{proof} Set for any continuous metric $\p$ 
$$
F_m(\p):=-\frac{n!}{m^{n+1}}\log V_{m\p,\n_{m\phi}}(P_m). 
$$
Then 
$$
F_m(\p)\ge-\frac{n!}{m^{n+1}}\log\d(m\p,\n_{m\phi}),
$$
with equality for $\p=\phi$, since $P_m$ is a Fekete configuration for $m\phi$. By Theorem~\ref{thm:trans}, we infer
$$
\liminf_{m\to\infty} F_m(\p)\ge\vol(L,\p,\phi),\,\,\,\,\,\,\lim_{m\to\infty} F_m(\phi)=\vol(L,\phi,\phi)=0, 
$$
and hence
\begin{equation}\label{equ:liminffm}
\liminf_{m\to\infty}\left(F_m(\p)-F_m(\phi)\right)\ge\vol(L,\p,\phi).
\end{equation}
For each $f\in  \cz(X^\an)$, observe that
$$
F_m(\phi+f)=F_m(\phi)+c_m\int_{X} f\,\d_{P_m}
$$
with 
$$
c_m:=\frac{n!}{m^n} N_m\to\vol(L)=V. 
$$
By~\eqref{equ:liminffm}, we get
$$
\liminf_{m\to\infty}\int f\,\d_{P_m}\ge V^{-1}\vol(L,\phi+f,\phi).
$$
Replacing $f$ with $tf$, $t>0$, and using the differentiability property~\eqref{equ:diffvol}, we infer
$$
\liminf_{m\to\infty}\int f\,\d_{P_m}\ge V^{-1}\lim_{t\to 0_+}t^{-1}\vol(L,\phi+t f,\phi)=\int f\,\mu_\phi.
$$
Applying this to $-f$ in place of $f$, we conclude as desired $\lim_{m\to\infty}\int f\,\d_{P_m}=\int f\,\mu_\phi$. 
\end{proof}

%
%
\subsection{A pullback formula for transfinite diameters}
In this section, we assume that $L$ is an ample line bundle on $X$, and consider a \emph{polarized endomorphism} $f$ of $(X,L)$, \ie a morphism $f:X\to X$ together with the data of an isomorphism $f^\star  L\simeq d L$ for some positive integer $d>1$. Since $f^\star  L$ is ample and $(f^\star  L)^n=d^n(L^n)$, $f$ is finite, of degree $d^n$. By Theorem~\ref{thm:del}, we thus have a canonical isomorphism 
$$
\langle (f^\star  L)^{n+1}\rangle\simeq d^n\langle L^{n+1}\rangle,
$$
which combines with the given isomorphism $f^\star  L\simeq dL$ to yield
$$
d^{n+1}\langle L^{n+1}\rangle\simeq d^n\langle L^{n+1}\rangle.
$$
This defines a canonical section
\begin{equation}\label{def:resultantsec}
R_f\in d^n(d-1)\langle L^{n+1}\rangle,
\end{equation}
which we call the \emph{resultant section} (see Corollary~\ref{cor:ResPolar} below for the choice of terminology). 

On the other hand, the map $ \cz(L)\to \cz(L)$ defined by $\phi\mapsto d^{-1}f^\star \phi$, being $1/d$-Lipschitz continuous, admits a unique fixed point $\phi_f$, the \emph{equilibrium metric} of $f$. For any choice of Fubini--Study metric $\phi$, the metrics $d^{-j}(f^j)^\star \phi$ are Fubini--Study as well, and they converge uniformly to $\phi_f$, which is thus psh. 

\begin{exam} For any $d\ge 2$, the equilibrium metric of the polarized endomorphism $f$ of $(\P^n,\cO(1))$ induced by $(x_0,\ldots,x_n)\mapsto (x_0^d:\dots:x_n^d)$ is $\phi_f=\max_i\log|x_i|$. 
\end{exam}

\begin{lem}\label{lem:resultant} The resultant $R_f\in d^n(d-1)\langle L^{n+1}\rangle$ has norm $1$ with respect to the induced metric $d^n(d-1)\langle\phi_f^{n+1}\rangle$.
\end{lem}
\begin{proof} For any continuous psh metric $\phi$ on $L$, the isomorphism $\langle (f^\star  L)^{n+1}\rangle\simeq d^n\langle L^{n+1}\rangle$ is an isometry with respect to $\langle (f^\star \phi)^{n+1}\rangle$ and $d^n\langle\phi^{n+1}\rangle$. By definition of $\phi_f$, the isomorphism $f^\star  L\simeq dL$ is an isometry with respect to $f^\star \phi_f$ and $d\phi_f$. It follows that the induced isomorphism $d^n(d-1)\langle L^{n+1}\rangle\simeq K$ is an isometry with respect to $d^n(d-1)\langle\phi_f^{n+1}\rangle$ and the canonical metric on $K$, hence the result. 
\end{proof}
We now get the following pull-back formula for the transfinite diameter, which generalizes~\cite{DR},~\cite[\S 6.3]{BB} in view of Corollary~\ref{cor:ResPolar} below.

\begin{thm}\label{thm:pullback} For each continuous psh metric $\p$, and let $c(f,\p)$ be the positive constant such that 
$$
\log c(f,\p)=-\frac{1}{(n+1)d^{n+1}} \log |R_f|_{d^n(d-1)\langle \p^{n+1}\rangle}. 
$$
Then 
$$
\d_\infty(d^{-1} f^\star  \phi,\p)=c(f,\p)\,\d_{\infty}(\phi,\p)^{1/d},
$$
all continuous psh metrics $\phi$. 
\end{thm} 

\begin{lem}\label{lem:pullback} For any two continuous psh metrics $\phi,\p$ on $L$ we have 
$$
\en\left(d^{-1}f^\star\phi,d^{-1}f^\star\p\right)=d^{-1}\en(\phi,\p).
$$
\end{lem}
\begin{proof} By~\eqref{equ:en}, 
$$
\en\left(d^{-1}f^\star\phi,d^{-1}f^\star\p\right)
$$
$$
=\frac{1}{n+1}\sum_{i=0}^n\int\left(d^{-1}f^\star\phi-d^{-1}f^\star\p\right)\left(dd^c(d^{-1}f^\star\phi)\right)^i\wedge\left(dd^c (d^{-1}f^\star\p)\right)^{n-i}\wedge\d_X
$$
$$
=\frac{1}{(n+1)d^{n+1}}\sum_{i=0}^n\int f^\star\left((\phi-\p)\,\left(dd^c\phi\right)^i\wedge\left(dd^c\p\right)^{n-i}\right)\wedge\d_X
$$
$$
=\frac{d^n}{(n+1)d^{n+1}}\sum_{i=0}^n\int (\phi-\p)\,\left(dd^c\phi\right)^i\wedge\left(dd^c\p\right)^{n-i}\wedge\d_X=d^{-1}\en(\phi,\p)
$$
since $f_\star \d_X=d^n\d_X$. 
\end{proof}

\begin{proof}[Proof of Theorem~\ref{thm:pullback}] Since $d^{-1}f^\star \phi_f=\phi_f$, Corollary~\ref{cor:trans} and the cocycle formula for $\en$ yield
$$
\log\d_{\infty}(d^{-1} f^\star \phi,\p)=\en\left(\p,d^{-1} f^\star\phi\right)=\en(\p,\phi_f)+\en\left(d^{-1} f^\star\phi_f,d^{-1}f^\star\phi\right)
$$
$$
=\en(\p,\phi_f)+d^{-1}\en(\phi_f,\phi)=(1-d^{-1})\en(\phi_f,\p)+d^{-1}\log\d_{\infty}(\phi,\p). 
$$
On the other hand, Lemma~\ref{lem:resultant} yields 
$$
\log|R_f|_{d^n(d-1)\langle\p^{n+1}\rangle}=d^n(d-1)\left(\langle\phi_f^{n+1}\rangle-\langle\p^{n+1}\rangle\right)=d^n(d-1)(n+1)\en(\phi_f,\p), 
$$
and hence 
$$
\frac{1}{(n+1)d^{n+1}} \log |R_f|_{d^n(d-1)\langle\p^{n+1}\rangle}=(1-d^{-1})\en(\phi_f,\p),
$$
and we are done. 
\end{proof}

\subsection{The case toric varieties}
We now illustrate the previous pull-back formula in the toric case. We assume that $(X,L)$ is a smooth projective polarized toric variety with respect to a split torus $T\simeq\G_m^n$ with character lattice $M=\Hom(T, \G_m)$, which thus corresponds to a Delzant polytope $\D\subset M_\R$. 

For each integer $d\ge 2$, multiplication by $d$ on the dual of $M$ induces a polarized endomorphism $m_d$ of $(X,L)$, defining an equilibrium metric $\phi_d$ on $L$ and a resultant section $R_d\in d^n(d-1)\langle L^{n+1}\rangle$. 
The next lemma describes the moduli space of polarized endomorphisms of degree $d$ of $(X, L)$. 

\begin{lem} Set $N=\dim \Hnot(L)=\#(M\cap\D)$. The space of polarized morphisms of degree $d$ of $(X,L)$ is parametrized by a Zariski open subset of 
$$
\P\left(\Hom\left(\Hnot(L),\Hnot(dL)\right)\right)\simeq\P\left(\Hnot(dL)^N\right)
$$
whose complement $Z$ has codimension $N-n$. In particular, $Z$ has codimension greater than $1$ unless $(X,L)\simeq(\P^n,\cO(1))$. 
\end{lem}
\begin{proof} Since $X$ is smooth, $R(X,L)$ is generated in degree one, the data of a polarized endomorphism of $(X,L)$ of degree $d$ is equivalent to that of a linear map $\Hnot(L)\to \Hnot(dL)$ whose image is basepoint free. As a result, the space of polarized endomorphisms of degree $d$ is isomorphic to the complement in $\P\left(\Hnot(dL)^N\right)$ of the projection $Z$ of the incidence variety
$$
I =\left\{ \left([s_1:\dots:s_N],x\right) \in\P\left(\Hnot(dL)^N\right)\times X, s_1(x)=\dots=s_N(x)=0\right\}
$$
Since $dL$ is basepoint free, the elements of $\Hnot(dL)$ vanishing at a given closed point $x\in X$ is a hyperplane, and it follows that $\dim I=n-1+N(N_d-1)$. We claim that the restriction to $I$ of the first projection $\Hnot(dL)^N\times X\to \Hnot(dL)^N$ is generically finite, which will imply $\codim Z=N N_d-1-\dim I=N-n$. Indeed, if $s_1,\ldots,s_n\in \Hnot(dL)$ is a regular sequence of sections, the fiber of $I$ over the $N$-tuple $(s_1, \ldots, s_1, s_2,\ldots, s_{n}) \in \Hnot(dL)^N$ is finite. 

The last point of the lemma follows from the embedding $X\hookrightarrow\P \Hnot(L)$, which is an isomorphism iff $N-n=1$. 
\end{proof}

In the case $N=n-1$, i.e.  $(X,L)\simeq(\P^n,\cO(1))$, any polarized morphism of degree $d$ is given by $z \mapsto [f_0(z):\ldots:f_n(z)]$ for homogenous polynomials $f_0, \ldots, f_n$, of degree $d$, without common zeros, and the locus $Z\subset\P\left(\Hnot(L)^{n+1}\right)$ is an irreducible divisor of degree $(n+1)d^n$. As a result, it is defined by a unique polynomial $\Op{Res}(f_0, \ldots, f_n)$ of degree $(n+1)d^n$ in the coefficients of the $f_i$ and normalized by $\Res(x_0^d, \ldots, x_n^d)=1$, cf.~\cite[Ch. 13]{GKZ}.

\begin{cor}\label{cor:ResPolar} Let $f$ be a polarized morphism of degree $d$ of the smooth polarized toric variety $(X,L)$. 
\begin{itemize} 
\item If $(X,L)\simeq (\P^n,\cO(1))$, then $R_f=\Res(f) R_d$; 
\item if not, then $R_f=R_d$. 
\end{itemize}
\end{cor}
\begin{proof} If $(X,L)\simeq(\P^n,\cO(1))$, we can restrict along the hyperplane determined by $x_n=0$ and inductively compare how the Deligne products and the resultants change. In the case of the resultant, the transformation is described by the Poisson formula~\cite[Ch. 13, Theorem 1.2]{GKZ}, and the Deligne products transform accordingly. This shows they are equal up to some constant, and the constant is equal to 1 by evaluating at the polarized endomorphism $m_d$.

If $(X,L)$ is not isomorphic to $(\P^n,\cO(1))$, the previous lemma shows that the space of polarized degree $d$ endomorphisms of $(X,L)$ is isomorphic to an open subset $U\subset\P\left(\Hnot(dL)^N\right)$ whose complement $Z$ has codimension at least $2$. The map $f\mapsto R_f/R_d$ defines a morphism $U\to\G_m$, which is thus constant by normality and properness of $\P\left(\Hnot(dL)^N\right)$, and hence equal to $1$ by evaluating at $f=m_d$. 
\end{proof}

%

\appendix
%
%
\section{Determinant of cohomology and Deligne pairings}\label{sec:Deligne}
%
%
The goal of this Appendix is to discuss a generalization to arbitrary schemes of results of Knudsen--Mumford~\cite{KM}, Deligne~\cite{Del}, Elkik~\cite{Elkik}, Munoz-Garcia~\cite{MG} and Ducrot~\cite{Duc}, which provide a rough Riemann--Roch theorem for the determinant of cohomology. 
%
%
\subsection{Discussion of the results}\label{sec:discussion}
For a projective scheme $X$ over a field $K$, the \emph{determinant of cohomology} of a line bundle $L$ is the line (\ie one-dimensional $K$-vector space)
$$
\la(L):=\sum_{i=0}^n(-1)^i\det H^i(X,L),
$$
where we use additive notation for tensor products of lines. If $\pi:X\to Y$ is now a flat projective morphism of locally Noetherian schemes, it was shown by Knudsen and Mumford in~\cite{KM} that the fiberwise determinant of cohomology of a line bundle $L$ on $X$ glues together to define a line bundle $\la_{X/Y}(L)$ on $Y$. Indeed, the derived direct image $R\pi_\star  L$ is a perfect complex, \ie locally on $Y$ there exists a bounded complex $E^\bullet$ of vector bundles with $R^q\pi_\star  L$ as its $q$-th cohomology sheaf, and the determinant of cohomology of $L$ can then be locally described as 
$$
\la_{X/Y}(L)=\sum_i(-1)^i\det E_i. 
$$
Denoting by $n$ the relative dimension of $\pi$, the main result in F.~Ducrot's paper~\cite{Duc} implies that the functor $\la_{X/Y}:\cP(X)\to\cP(Y)$ so defined between the Picard categories of line bundles on $X$ and $Y$ admits a unique polynomial structure of degree $n+1$ compatible with base change and restriction to a relative Cartier divisor (see \S\ref{sec:KM} for a precise statement). This result recovers in one stroke the construction of Deligne pairings~\cite{Elkik1,MG} and the Knudsen--Mumford expansion~\cite{KM}. Indeed, it implies that the $(n+1)$-st iterated difference
$$
\langle L_0,\ldots,L_n\rangle_{X/Y}:=\sum_{I\subset\{0,\ldots,n\}}(-1)^{n+1-|I|}\la_{X/Y}\left(\sum_{i\in I} L_i\right)
$$
defines a multi-additive symmetric functor $\cP(X)^{n+1}\to\cP(Y)$, the \emph{Deligne pairing}, and that we have for each $L\in\cP(X)$ an expansion 
$$
\la_{X/Y}(mL)=\sum_{i=0}^{n+1}{m+i \choose i} M_i
$$
as a function of $m\in\Z$, the coefficients being line bundles $M_i$ on $Y$, and $M_{n+1}=\langle L^{n+1}\rangle_{X/Y}$. Whenever a Grothendieck--Riemann--Roch theorem is available, we infer
$$
c_1\left(\langle L_0,\ldots,L_n\rangle_{X/Y}\right)=\pi_\star (c_1(L_0)\cdot\ldots\cdot c_1(L_n)\rangle,
$$
so that Deligne pairings lift the natural push-forward operation on the right-hand side to the level of line bundles. 

\medskip

In the main body of the present paper, a version of these results in the possibly non-Noetherian setting of models over the valuation ring of a complete non-Archimedean field is required. The purpose of this appendix is to summarize the results leading to a generalization of the above results for arbitrary schemes. 
%
%
\subsection{Polynomial maps}\label{sec:diff}
In order to motivate the definitions in \S\ref{sec:polyfunc}, we briefly recall some background on polynomial maps and difference calculus. It is well-known that a map $f:\Z\to\Z$ is polynomial of degree (at most) $n$ if and only if it admits an expansion
$$
f(m)=\sum_{i=0}^n{m+i\choose i}b_i
$$
with coefficients $b_i\in\Z$. More generally, a map $f:A\to B$ between commutative groups is said to be \emph{polynomial of degree $n$} if for any given $x_1,\ldots,x_r\in A$ we have an expansion
$$
f(m_1x_1+\dots+m_r x_r)=\sum_{0\le i_1,\ldots,i_r\le n}{m_1+i_1\choose i_1}\dots{m_r+i_r\choose i_r}b_{i_1\dots i_r}
$$
for all $m_i\in\Z$, with coefficients $b_{i_1\dots i_r}\in B$. 

Polynomiality can be characterized in terms of difference calculus. For a map $f:\Z\to B$, define the difference $\D f:\Z\to B$ by 
$$
(\D f)(m):=f(m+1)-f(m). 
$$
Then
$$
\D{m+i\choose i}={m+i-1\choose i-1},
$$
which can be used to show by induction on $n$ that $f:\Z\to B$ is a polynomial map of degree $n$ if and only if $\D^{n+1} f=0$. For a map $f:\Z^r\to B$, one can introduce partial difference operators $\D_i$, and $f$ is polynomial of degree $n$ if and only if $\D^\a f=0$ for all multi-indices $\a\in\N^r$ of length $|\a|:=\sum_i\a_i=n+1$, where we have set $\D^\a =\D_1^{\a_1}\dots\D_r^{\a_r}$. 

Consider now a map $f:A\to B$ between commutative groups. Mimicking differential calculus, one defines the \emph{difference} $\d_x f:A\to B$ of $f$ at $x\in A$ by setting 
$$
(\d_x f)(y)=f(x+y)-f(x),
$$
and the \emph{$k$-th iterated difference} $\d^k_x f:A^k\to B$ by 
$$
(\d^k_x f)(x_1,\ldots,x_k):=\d_x\left(y\mapsto (\d^{k-1}_y f)(x_2,\ldots,x_k)\right)(x_1)
$$
$$
=(\d^{k-1}_{x+x_1}f)(x_2,\ldots,x_k)-(\d^{k-1}_x f)(x_2,\ldots,x_k).
$$
The map $\d^k_x f:A^k\to B$ so defined is symmetric, as follows from the explicit expression
\begin{equation}\label{equ:itdiff2}
(\d^k_x f)(x_1,\ldots,x_k)=\sum_{I\subset\{1,\ldots,k\}}(-1)^{k-|I|}f\left(x+\sum_{i\in I} x_i\right). 
\end{equation}
Given $x_1,\ldots,x_r\in A$, the map $g:\Z^r\to B$ defined by $g(m_1,\ldots,m_r)=f(\sum_i m_ix_i)$ satisfies 
\begin{equation}\label{equ:partialdiff}
(\D^\a g)(m_1,\ldots,m_r)=(\d^{|\a|}_{\sum_i m_i x_i} f)(x_1^{ a_1},\ldots,x_r^{ a_r})
\end{equation}
for all $\a\in\N^r$. 

Using this, we conclude that $f:A\to B$ is polynomial of degree $n$ if and only if $\d^{n+1}_x f=0$ for all $x\in A$. It is in fact enough to check this condition for $x=0$. Indeed, the operator $\d^k:=\d^k_0$ determines all $\d^k_x$ by
\begin{equation}\label{equ:dn} 
(\d^k_x f)(x_1,\dots)=(\d^k f)(x+x_1,\dots)-(\d^k f)(x,\dots). 
\end{equation}
It further satisfies
$$
(\d^{k+1} f)(x_1,y_1,\dots)=(\d^k  f)(x_1+y_1,\dots)-(\d^k f)(x_1,\dots)-(\d^k f)(y_1,\dots),
$$
and we thus see that $f$ is polynomial of degree $n$ if and only if $\d^n f$ is multi-additive. Note also that $\d^k$ can be understood as a polarization operator, in the sense that $\d^k(L(x^k))=k! L$ for any symmetric multi-additive map $L:A^k\to B$; we can thus view the multi-additive map $\d^n f:A^n\to B$ associated to a polynomial map $f:A\to B$ of degree $n$ as the polarization of its degree $n$ part. 

\begin{exam} Let $X$ be an $n$-dimensional projective scheme over a field $K$, with structure morphism $\pi:X\to\spec K$, and pick a line bundle $L$ on $X$. The \emph{Euler characteristic} 
$$
\chi(L):=\sum_{i=0}^n(-1)^i\dim_K H^i(X,L)
$$
only depends on the class of $L$ in the Picard group $\Pic(X)$, and Snapper's theorem implies that $\chi:\Pic(X)\to\Z$ is a polynomial map of degree $n$, with degree $n$ polarization given by the intersection pairing, \ie 
$$
(\d^n\chi)(L_1,\ldots,L_n)=(L_1\cdot\ldots\cdot L_n):=\deg\pi_\star \left(c_1(L_1)\cdot\ldots\cdot c_1(L_n)\cdot[X]\right).  
$$
\end{exam} 
For later use, we note: 
\begin{lem}\label{lem:polyexp} Let $f:A\to B$ is a polynomial map of degree $n$, and set for all $0\le i\le n$
$$
f_{n,i}(x):=\sum_{k=i}^n (-1)^{k-i}{k \choose i}(\d^k f)(x^k).
$$
Then $f(mx)=\sum_{i=0}^n {m+i\choose i}f_{n,i}(x)$ for all $x\in A$ and $m\in\Z$. 
\end{lem}
\begin{proof} Since $f$ is polynomial of degree $n$, we have 
$$
g(m):=f(mx)=\sum_{i=0}^n{m+i\choose i} b_i
$$
for some $b_i\in B$. Since $\D_0^k{m+i\choose i}=1$ for $k\le i$ and $0$ otherwise, (\ref{equ:partialdiff}) yields
$$
(\d^k f)(x^k)=\D_0^k g=\sum_{i\ge k} b_i,
$$
and hence $b_i=f_{n,i}(x)$. 
\end{proof}
%
%
\subsection{Polynomial functors}\label{sec:polyfunc}
A \emph{commutative Picard category} is a 'category version' of a commutative group, or more precisely a symmetric monoidal groupoid where every object is invertible with respect to the monoidal structure. In other words, it is a category $\cA$ in which all arrows are isomorphisms, together with an additivity functor and functorial associativity and commutativity isomorphisms satisfying the expected compatibility conditions, and such that for any object $x$ the endofunctors $y \mapsto x + y$ and $y \mapsto y + x$ are autoequivalences. 

These axioms imply the existence of a neutral object $0$ and of an inverse $-x$ for each object $x$, both unique up to unique isomorphism. The sum $\sum_{i\in I} x_i$ of a finite family $(x_i)_{i\in I}$ of objects in $\cA$ is well-defined up to unique isomorphism, and satisfies the expected associativity rules.  A commutative Picard category $\cA$ is \emph{strictly commutative} if the commutativity isomorphism induces the identity on $x+x$ for each $x$, in which case $x$ and $-x$ can be contracted within a sum without raising any sign issue. 

In practice for us, $\cA$ will be the category of line bundles or $\Q$-line bundles on a given scheme, and isomorphisms between them, both of which are strictly commutative Picard categories. Note that a commutative group can also be viewed as a strictly commutative Picard category. 


In what follows, $\cA$ and $\cB$ are strictly commutative Picard categories. An \emph{additive functor} $F:\cA\to\cB$ is a functor equipped with a functorial additivity isomorphism $F(x+y)\simeq F(x)+F(y)$ which is \emph{commutative}, expressed by the commutativity of the induced diagram 
$$
\xymatrix{
F(x+y) \ar[r] \ar[d] & F(x)+F(y) \ar[d] \\ 
F(y+x) \ar[r] &  F(y)+F(x) 
},
$$
and \emph{associative}, \ie the commutativity of the diagram
$$
\xymatrix{
F((x+y)+z) \ar[r] \ar[d] &  F(x+y)+F(z) \ar[r]   & \left(F(x)+F(y)\right)+F(z) \ar[d] \\
F(x+(y+z)) \ar[r] & \ar[r] F(x)+F(y+z) & F(x)+\left(F(y)+F(z)\right).
}$$
These conditions then yield a consistent system of functorial additivity isomorphisms $F(\sum_{i\in I} x_i)\simeq \sum_{i\in I} F(x_i)$ for all finite families $(x_i)_{i\in I}$ in $\cA$. If $\cA$ and $\cB$ are small, the associated sets of isomorphism classes $A$ and $B$ are commutative groups, and $F$ induces a homomorphism $F:A\to B$. 

A \emph{multi-additive functor} $F:\cA^n\to\cB$ is defined as a functor equipped with functorial commutative and associative additivity data in each variable, such that expanding out sums in the variables does not depend on the order the operation is performed. A \emph{symmetric functor} $F:\cA^n\to\cB$ is a functor equipped with symmetry isomorphisms
$$
F(x_{\sigma(1)},\ldots,x_{\sigma(n)})\simeq F(x_1,\ldots,x_n)
$$
for each permutation $\sigma\in\fS_r$, compatible with the group law on $\fS_r$, and a \emph{symmetric, multi-additive} functor has both structures, with the expected compatibility condition.

Define the \emph{$k$-th iterated difference} at an object $x$ in $\cA$ of a functor $F:\cA\to\cB$ as the symmetric functor $\d^k_x F:\cA^k\to\cB$ defined by setting 
$$
(\d^k_x F)(x_1,\ldots,x_k):=\sum_{I\subset\{1,\ldots,k\}}(-1)^{k-|I|}F\left(x+\sum_{j\in I} x_j\right),
$$
For $x=0$, we simply set $\d^k:=\d^k_0$. Recalling from \S\ref{sec:diff} that a map $f$ between commutative groups is polynomial of degree $n$ if and only if $\d^n f$ is multi-additive, we introduce:  

\begin{defi}\label{defi:poly} A \emph{polynomial structure of degree $n$} on $F$ is defined as a structure of a multi-additive functor on $\d^n F$, compatible with its canonical symmetry. 
\end{defi}
Ducrot introduces in~\cite[Definition 1.6.1]{Duc} the notion of \emph{$k$-cube structure} on $F$. By~\cite[Proposition 1.9]{Duc}, an $(n+1)$-cube structure on $F$ induces a polynomial structure of degree $n$ on $F$ (and the converse is probably true as well, by~\cite[1.5.1, (d)]{Duc}). Ducrot's terminology comes from the following well-known result. 

\begin{exam} If $L$ is a line bundle on an abelian variety $A$, the \emph{theorem of the cube} asserts that for any variety $S$, the functor $F_L:A(S)\to\cP(S)$ defined by $F_L(x):=x^\star  L$ admits a $3$-cube structure. It is thus quadratic in our sense, \ie $(x,y)\mapsto (x+y)^\star  L-x^\star  L-y^\star  L+0^\star  L$ is biadditive. Further, the whole structure is compatible with base change. 
\end{exam}

In analogy with Lemma~\ref{lem:polyexp}, we have: 

\begin{lem}\label{lem:poly} Suppose that $F:\cA\to\cB$ admits a polynomial structure of degree $n$, and define for $0\le i\le n$ a functor $F_{n,i}:\cA^i\to\cB$ by setting
$$
F_{n,i}(x):=\sum_{k=i} (-1)^{k-i} {k\choose i}(\d^k f)(x^k).
$$ 
For all $x\in\cA$ and $m\in\Z$, we then have canonical functorial isomorphisms
$$
F(mx)\simeq\sum_{i=0}^n{m+i \choose i}F_{n,i}(x). 
$$
\end{lem}
\begin{proof} Define $g:\Z\to\cB$ by $g(m):=F(mx)-\sum_{i=0}^n{m+i\choose i} F_{n,i}(x)$. Since
$$
\D_m^k F(mx)=\sum_{i=0}^k(-1)^{k-i}{k\choose i} F((m+i x)=(\d^k_{mx} F)(x^k), 
$$
and $\D^k_0{m+i\choose i}=1$ for $k\le i$ and $0$ otherwise, we have canonical isomorphisms
$(\D^k g)(0)\simeq 0$ for $k=0,\ldots,n$. Further, 
$$
(\D^{n-1} g)(m+1)-(\D^{n-1} g)(m)=(\D^n g)(m)=(\d^n_{mx} F)(x^n)-(\d^n F)(x^n)
$$
$$
\simeq(\d^n F)((m+1)x,x,\ldots,x)-(\d^n F)(mx,x,\ldots,x)-(\d^n F)(x,\ldots,x)\simeq 0,
$$
for all $m\in\Z$, by multiadditivity of $\d^n F$. Summing up these relations, we get 
$$
(\D^{n-1} g)(m)\simeq(\D^{n-1} g)(0)\simeq 0,
$$
and iterating the argument finally yields $g(m)\simeq 0$ for all $m$. 
\end{proof}

%
%
\subsection{Coherence of direct images}\label{sec:cohdir}
The following discussion is inspired in part by~\cite{Ull}. Let $X$ be a scheme. A \emph{vector bundle} on $X$ is a finite, locally free $\cO_X$-module. A complex $F^\bullet$ of $\cO_X$-modules is \emph{pseudo-coherent} (resp.~\emph{perfect}) if $F^\bullet$ is locally quasi-isomorphic to a bounded above (resp.~bounded) complex of vector bundles. In particular, an $\cO_X$-module $F$ is pseudo-coherent (resp.~perfect) if it locally admits a resolution (resp.~a finite resolution) by vector bundles. 

A morphism of schemes $f:X\to Y$ is \emph{pseudo-coherent} if $X$ is locally realized as a closed subscheme of a smooth $Y$-scheme $Z$ such that $\cO_X$ is pseudo-coherent as an $\cO_Z$-module. In particular, $f$ is locally finitely presented; conversely, every flat, locally finitely presented morphism is pseudo-coherent~\cite[Tag 0695]{stacks-project}.  

The following general results hold for arbitrary schemes $X,Y$, and are respectively proved in~\cite[Theorem 2.9]{Kie}, \cite[III, Corollaire 2.3]{SGA6}, and~\cite[Tag 0B91]{stacks-project}. 

\begin{thm}\label{thm:kiehl} Let $f:X\to Y$ be a proper, pseudo-coherent morphism, and let $F$ be a pseudo-coherent $\cO_X$-module. Then: 
\begin{itemize}
\item[(i)] the derived direct image $Rf_\star F$ is pseudo-coherent;
\item[(ii)] if $Y$ is quasi-compact and $L$ is an $f$-ample line bundle on $X$, then $R^q f_\star F(mL)=0$ for all $q\ge 1$ and $m\gg 1$. 
\end{itemize}
\end{thm}

\begin{thm}\label{thm:perfect} Let $f:X\to Y$ be proper, flat, finitely presented (and hence pseudo-coherent) morphism. If $F^\bullet$ is a perfect complex on $X$, then $Rf_\star F^\bullet$ is perfect on $Y$, and the construction is further compatible with arbitrary base change. 
\end{thm}

We now discuss the relation between Theorem~\ref{thm:kiehl} and their better known versions in the Noetherian case (coherence of direct images and Serre vanishing). 

Recall that a ring $A$ is \emph{coherent} if every finitely generated ideal is finitely presented; the ring $A$ is \emph{stably coherent} if every polynomial ring $A[t_1,\ldots,t_r]$ is coherent. We shall say that a scheme $X$ is coherent (resp.~stably coherent) if it is locally the spectrum of a coherent (resp.~stably coherent) ring. A scheme $X$ is coherent iff its structure sheaf $\cO_X$ is coherent. 

\begin{exam} Every locally Noetherian scheme is stably coherent. 
\end{exam}

By~\cite[I, Corollaire 3.5]{SGA6}, we have: 

\begin{lem}\label{lem:schcoh} On a coherent scheme $X$, the following conditions for an $\cO_X$-module $F$ are equivalent:
\begin{itemize}
\item[(i)] $F$ is coherent; 
\item[(ii)] $F$ is locally finitely presented; 
\item[(iii)] $F$ is pseudo-coherent. 
\end{itemize}
More generally, a complex of $\cO_X$-modules $F^\bullet$ is pseudo-coherent iff its cohomology sheaves $H^q(F^\bullet)$ are coherent for all $q$, and zero for $q$ sufficiently large, locally uniformly on $X$.
\end{lem}

A \emph{Pr\"ufer domain} is an integral domain $A$ which satisfies one of the following equivalent conditions: 
\begin{itemize}
\item[(i)] every localization of $A$ at a prime ideal is a valuation ring;
\item[(ii)] every finitely generated ideal of $A$ is invertible;
\item[(iii)] every torsion-free $A$-module is flat.
\end{itemize}
The main example for us is the valuation ring $K^\circ$ of a non-Archimedean field $K$. 

\begin{exam} Every Pr\"ufer domain $A$ is stably coherent. Indeed, every finitely generated ideal $I$ of $A[t_1,\ldots,t_r]$ is torsion free, and hence flat over $A$.  By~\cite[Th\'eor\`eme 3.4.6]{RG}, $I$ is thus a finitely presented $A[t_1,\ldots,t_r]$-module. 
\end{exam}

\begin{lem}\label{lem:stabcoh} Let $f:X\to Y$ be a locally finitely presented morphism of schemes, and assume that $Y$ is stably coherent. Then:
\begin{itemize}
\item[(i)] $f$ is pseudo-coherent; 
\item[(ii)] $X$ is (stably) coherent;
\end{itemize}
\end{lem}
\begin{proof} After passing to affine open subschemes, we may assume that $X$ and $Y$ are spectra of rings $A,B$ such that $B=A[t]/I$ with $I\subset A[t]=A[t_1,\ldots,t_r]$ a finitely generated ideal. By assumption, $A$ is stably coherent, and hence $A[t]$ is coherent. By Lemma~\ref{lem:schcoh}, the finitely presented $A[t]$-module $B$ is thus pseudo-coherent, which proves (i). 

Let now $J\subset B$ be a finitely generated ideal. As modules over $A[t]$, $B$ is finitely presented, and $J$ is a finitely generated submodule. By coherence of $A[t]$, $J$ is a finitely presented as an $A[t]$-module, and hence also as a $B$-module, which proves (ii). 
\end{proof}

We can now state a version of Theorem~\ref{thm:kiehl} that recovers in particular the usual statement for locally Noetherian schemes.

\begin{cor}\label{cor:kiehl} Let $f:X\to Y$ be a proper, locally finitely presented morphism of schemes. Assume that $Y$ is stably coherent (\eg locally Noetherian, or locally finitely presented over a Pr\"ufer domain), and let $F$ be a coherent $\cO_X$-module. Then: 
\begin{itemize} 
\item[(i)] $R^q f_\star F$ is coherent for all $q$; 
\item[(ii)] if $Y$ is quasi-compact and $L$ is an $f$-ample line bundle on $X$, then $R^q f_\star F(mL)=0$ for all $q\ge 1$ and $m\gg 1$.
\end{itemize}
\end{cor}
\begin{proof} By Lemma~\ref{lem:stabcoh}, $f$ is pseudo-coherent, and $X$ is coherent. Thus $F$ is pseudo-coherent. Theorem~\ref{thm:kiehl} directly yields (ii), and shows that $R f_\star F$ is pseudo-coherent. Since $Y$ is coherent, this amounts to (i), by Lemma~\ref{lem:schcoh}. 
\end{proof}

%
%
\subsection{The determinant of a perfect complex}

Let $X$ be a scheme. The \emph{determinant} of a vector bundle $E$ on $X$ is the line bundle $\det E:= \bigwedge^{\rk E} E$. If 
\begin{equation}\label{exseq}
0\to E'\to E\to E''\to 0
\end{equation} 
is an exact sequence of vector bundles, then there is a canonical isomorphism  \begin{equation}\label{additive}
\det E\simeq\det E'+\det E'',
\end{equation} where, in additive notation, $+$ denotes the tensor product of line bundles, . However, given two vector bundles $E,F$, the isomorphism 
$$
\det E+\det F\simeq\det(E\oplus F)\simeq\det(F\oplus E)\simeq\det F+\det E
$$ 
induced by the canonical isomorphism $E\oplus F\simeq F\oplus E$ coincides with the canonical commutativity isomorphism only up to a factor $(-1)^{(\rk E)(\rk F)}$. 

To deal with this sign issue, one introduces the graded determinant functor $E\mapsto(\det E,\rk E)$ with values in the (non-strictly) commutative Picard category $\cP(X)\times\Z$ of graded line bundles, in which the commutativity isomorphism is modified according to the Koszul rule of signs as above. For the purpose of the present paper, it will however be enough to view $\det E$ as an object in the strictly commutative Picard category $\cP(X)_\Q$ of $\Q$-line bundles on $X$, and we can thus ignore the previous sign issue. 

In~\cite[Theorem 2]{KM}, Knudsen and Mumford showed that setting 
$$
\det E^\bullet:=\sum_i(-1)^i\det E^i
$$
for each bounded complex of vector bundles $E^\bullet$ gives rise to a functor $F^\bullet\mapsto\det F^\bullet$ from the category of perfect complexes on $X$ and quasi-isomorphims between them to $\cP(X)_\Q$. This functor commutes with base change, it is additive with respect to short exact sequences of complexes in the sense of~\eqref{additive}. It can be uniquely characterized up to unique isomorphism by imposing further properties. By~\cite[p.43, (b)]{KM}, we have:

\begin{lem}\label{lem:det free} If the cohomology sheaves $H^q(F^\bullet)$ of a perfect complex $F^\bullet$ are perfect (\eg locally free), then 
$$
\det F^\bullet=\sum_q (-1)^q\det H^q(F^\bullet).
$$
\end{lem}
%
%
\subsection{Regular sections}\label{sec:regsec}
Let $f:X\to Y$ be a flat, locally finitely presented morphism of schemes, and let $s$ be a global section of a line bundle $L$ on $X$, defining a closed subscheme $Z\subset X$. Recall that $s$ is \emph{$f$-regular at $x\in X$} if $Z$ is a relative Cartier divisor at $x$, \ie $s$ is a nonzerodivisor in $\cO_{X,x}$, and $Z$ is $f$-flat at $x$. If this holds for all $x\in X$, then $s$ is simply called $f$-regular. By~\cite[IV.11.3.7]{EGA}, $s$ is $f$-regular at $x$ if and only the restriction of $s$ to the fiber through $x$ is not a zerodivisor at $x$, and the set of $x\in X$ at which this holds is open. 


\begin{lem}\label{lem:regsec} Let $f:X\to Y$ be flat, proper, finitely presented morphism of schemes, and $s$ be a global section of a line bundle $L$ on $X$. Pick $y\in Y$, and assume that $s$ is nonzero at each associated point of $X_y$. Then $s$ is relatively regular over an open neighborhood of $y$. 
\end{lem}
\begin{proof} Denote by $U\subset X$ the open set of points at which $s$ is relatively regular. Since $X_y$ is Noetherian (being of finite type over a field), the assumption implies that $s|_{X_y}$ is a nonzerodivisor at each $x\in X_y$, and hence that $X_y\subset U$ by the above results. As $f$ is closed, it follows that $f^{-1}(V)\subset U$ for some open neighborhood $V$ of $y$. 
\end{proof}

As a consequence, we then have the following useful existence result for relatively regular sections. 

\begin{prop}\label{prop:reg} Let $f:X\to Y$ be a flat, projective, finitely presented morphism of schemes, and let $L$ be a $f$-ample line bundle on $X$. Then $mL$ admits a relatively regular section locally over $Y$ for $m\gg 1$. 
\end{prop}
\begin{proof} Pick $y\in Y$. After replacing $Y$ with an affine neighborhood of $y$ and  $L$ with a large enough multiple, we may assume that there exists a closed embedding $X\hookrightarrow\P^N_Y$ over $Y$ with $L=\cO(1)|_X$. The set of associated points $S$ of $X_y$ being finite, prime avoidance yields for $m\gg 1$ a section $s\in \Hnot(\P_Y^N,\cO(m))$ that does not vanish at any point of $S$, and the restriction of $s$ to $X$ is thus relatively regular over a neighborhood of $y$, by Lemma~\ref{lem:regsec}. 
\end{proof}

%
%
%
\subsection{The determinant of cohomology} 
Let $f:X\to Y$ be a flat, proper, finitely presented morphism between arbitrary schemes $X,Y$, with Picard categories $\cP(X)$, $\cP(Y)$. By Theorem~\ref{thm:perfect}, $Rf_\star$ takes a perfect complex to a perfect complex, and we can thus consider its determinant. We will only be concerned with the case when the source perfect complex is a line bundle, so we wil be content with the following definition, which could be stated more generally:

\begin{defi}\label{defi:detcoh} The \emph{determinant of cohomology} is the functor $\la_{X/Y}:\cP(X)\to\cP(Y)_\Q$ that takes a line bundle $L$ on $X$ to the $\Q$-line bundle 
$$
\la_{X/Y}(L):=\det Rf_\star L. 
$$
\end{defi}

By Lemma~\ref{lem:det free}, we have:

\begin{lem}\label{lem:detcohfree} Let $L$ be a line bundle on $X$ is such that $R^qf_\star L$ is a vector bundle, for all $q$. Then 
$$
\la_{X/Y}(L)=\sum_{q\in\N} (-1)^q\det R^q f_\star L.
$$
\end{lem}
This holds in particular when $f$ has relative dimension $0$, \ie a finite flat morphism. Indeed, any finite morphism $f:X\to Y$ satisfies $R^q f_\star F=0$ for any $\cO_X$-module $F$ and $q>0$~\cite[Tag 03QP]{stacks-project}. By~\cite[Tag 02KB]{stacks-project}, $f$ is flat iff $f_\star \cO_X$ is a vector bundle, and we then have $\la_{X/Y}(L)=\det f_\star\cO_X$.

\begin{prop}\label{prop:basechange} The determinant of cohomology satisfies the following two compatibility properties. 
\begin{itemize}
\item[(i)] It commutes with arbitrary base change: for any morphism $g:Y'\to Y$, let 
$$\xymatrix{X' \ar[r]^{h} \ar[d]^{f'} & X \ar[d]^{f} \\
Y' \ar[r]^{g} & Y}
$$ 
be the corresponding Cartesian square. Then we have a canonical functorial isomorphism 
$$
\la_{X'/Y'}(h^\star  L)\simeq g^\star \la_{X/Y}(L).
$$
\item[(ii)] If $Z$ is an effective relative Cartier on $X$, then we have a canonical functorial isomorphism 
\begin{equation}\label{equ:detrestr}
\la_{X/Y}(L)-\la_{X/Y}(L-Z)\simeq\la_{Z/Y}(L|_Z). 
\end{equation}
\end{itemize}
\end{prop}
\begin{proof} Since $L$ is flat over $Y$, we have $Rf'_\star h^\star L =Lg^\star  Rf_\star  L$, cf.~~\cite[Tag 0A1D]{stacks-project}. On the other hand, the determinant functor satisfies $\det Lg^\star  = g^\star  \det$, hence (i). If $Z$ is an effective relative Cartier divisor on $X$, then the natural exact sequence $0\to L(-Z)\to L\to L|_Z\to 0$ induces an exact sequence of perfect complexes $0\to Rf_\star L(-Z)\to Rf_\star L\to R(f|_Z)_\star L|_Z\to 0$, and (ii) follows by additivity of $\det$ in exact sequences. 
\end{proof}

%
%
\subsection{Deligne pairings and Knudsen--Mumford expansion} \label{sec:KM} 
In this section, we fix a flat, projective, finitely presented morphism $f:X\to Y$ of constant relative dimension $n$. 

When $n=0$, $f:X\to Y$ is a finite flat morphism, and the determinant of cohomology $\la_{X/Y}$ provides a canonical and functorial construction of the norm of a line bundle~\cite[II.6.5]{EGA} (compare for instance~\cite[Proposition 3.3]{Fer}). To see this, recall first that the norm $N_{X/Y}(h)\in\cO_Y$ of $h\in f_\star \cO_X$ is defined as the determinant of the endomorphism of the vector bundle $f_\star \cO_X$ defined by multiplication by $h$, yielding a multiplicative map 
\begin{equation}\label{equ:norm}
N_{X/Y}:f_\star \cO_X\to\cO_Y. 
\end{equation}
Now define a functor $N_{X/Y}:\cP(X)\to\cP(Y)$ by setting
$$
N_{X/Y}(L)=(\det f_\star L)-(\det f_\star \cO_X)
$$
$$
=\la_{X/Y}(L)-\la_{X/Y}(\cO_X)=(\d\la_{X/Y})(L),
$$
\begin{lem}\label{lem:norm} For each line bundle $L$ on $X$, $N_{X/Y}(L)$ coincides with the norm of $L$ as defined in~\cite[II.6.5]{EGA}.
\end{lem}
\begin{proof} Observe that if $u\in \Hnot(X,\cO_X^\star )$ is a unit and $L$ is a line bundle on $X$, multiplication by $u$ defines an isomorphism $L\simeq L$, whose induced isomorphism $\det f_\star L\simeq\det f_\star L$, and hence also $N_{X/Y}(L)\simeq N_{X/Y}(L)$, are both given by multiplication by $N_{X/Y}(u)$. 

By~\cite[Tag 0BUT]{stacks-project}, $L$ is trivial in a neighborhood of each fiber of $ f$, and $Y$ therefore admits an open cover $(Y_i)$ with $L|_{X_i}\simeq\cO_{X_i}$ on $X_i:= f^{-1}(Y_i)$. Set $Y_{ij}=Y_i\cap Y_j$, $X_{ij}=X_i\cap X_j= f^{-1}(Y_{ij})$, and denote by $u_{ij}\in \Hnot(X_{ij},\cO_{X_{ij}}^\star )$ the corresponding cocycle. The transition isomorphism $\cO_{X_{ij}}\simeq L|_{X_{ij}}\simeq \cO_{X_{ij}}$ is given by multiplication by $u_{ij}$. By the above observation, applying the functor $N_{X/Y}$ yields an isomorphism $\cO_{Y_{ij}}\simeq N_{X/Y}(L)|_{Y_{ij}}\simeq\cO_{Y_{ij}}$ given by multiplication by $N_{X_{ij}/Y_{ij}}(u_{ij})$, which precisely means that $N_{X/Y}(L)$ coincides with the norm of $L$ as defined in~\cite[II.6.5]{EGA}.
\end{proof}

When $Y$ is the spectrum of a field $K$ and $X$ is the spectrum of a finite flat $K$-algebra $A$, the norm functor admits the following concrete description (see \eg~\cite[Lemma 1.13]{Mor}). For $a \in A$, we have 
\begin{equation}\label{eq:normmet} N_{A/K}(a)= \prod_{n_i \in \Spec(A)}N_{(A/n_i)/K} (a_i)^{m_i}.
\end{equation} 
Here $a_i$ is the image of $a$ in $A/n_i$, $m_i = \operatorname{length}_{B_{n_i}} A_{n_i}$ so that $\sum m_i [a_i]$ is the fundamental cycle of $[A].$

Arguing as in~\cite[4.1.1]{Duc}, we next prove: 

\begin{prop}\label{prop:norm} There is a unique way to assign to each finite flat morphism $ f:X\to Y$ an additivity structure on the norm functor $N_{X/Y}:\cP(X)\to\cP(Y)$ that is compatible with base change and such that the following diagram commutes
$$
\xymatrix{
N_{X/Y}(L) \ar[r] \ar[d] & N_{X/Y}(L) \ar[d] \\ 
N_{X/Y}(L+\cO_X) \ar[r] &  N_{X/Y}(L)+N_{X/Y}(\cO_X)  
}.
$$
Here the upper row is the identity, the lower row is the additivity isomorphism, and the vertical rows are deduced from the natural isomorphism of the type $L \simeq L + \cO_X$ and the identification $N_{X/Y}(\cO_X)=\cO_Y$. 
\end{prop}
\begin{proof} Pick two line bundles $L,L'$ on $X$, and choose an open cover $(Y_i)$ of $Y$ with trivializations $L|_{X_i}\simeq\cO_{X_i}$, $L'|_{X_i}\simeq\cO_{X_i}$ on $X_i:=f^{-1}(Y_i)$. Given an additivity isomorphism 
$$
N_{X/Y}(L+L')\simeq N_{X/Y}(L)+N_{X/Y}(L')
$$ 
with the desired properties, the induced isomorphisms
$$
N_{X_i/Y_i}(\cO_{X_i}+\cO_{X_i})\simeq N_{X_i/Y_i}(\cO_{X_i})+N_{X_i/Y_i}(\cO_{X_i})
$$
are necessarily equal to the canonical ones obtained by multiplicativity of (\ref{equ:norm}), which proves uniqueness. To establish existence, it is then enough to argue locally on $Y$, since compatibility on overlaps will follow from uniqueness, and the result is then straightforward, using again that any line bundle on $X$ is trivial locally over $Y$. 
\end{proof}

In the terminology of \S\ref{sec:polyfunc}, Proposition~\ref{prop:norm} says that the functor $\la_{X/Y}$ is polynomial of degree $1$ when $n=0$. This is generalized by the next result, due to F.~Ducrot. 

\begin{thm}\label{thm:duc}\cite[Theorem 4.2]{Duc} There exists a unique way to assign to each flat, projective, finitely presented morphism $ f:X\to Y$ of relative dimension $n$ a polynomial structure of degree $n+1$ on $\la_{X/Y}:\cP(X)\to\cP(Y)_\Q$ with the following properties: 
\begin{itemize}
\item[(i)] it commutes with base change; 
\item[(ii)] it coincides with the above one when $n=0$; 
\item[(iii)] for any relative effective divisor $Z$ on $X$, the polynomial structures on $\la_{X/Y}$ and $\la_{Z/Y}$ are compatible with the canonical restriction isomorphism 
$$
\la_{X/Y}(L)-\la_{X/Y}(L-Z)\simeq\la_{Z/Y}(L|_Z). 
$$ 
\end{itemize}
\end{thm} 
More precisely,~\cite[Theorem 4.2]{Duc} proves the existence of a canonical $(n+2)$-cube structure on $\la_{X/Y}$, which yields (and is probably equivalent to) a polynomial structure of degree $n+1$ on $\la_{X/Y}$, as discussed in \S\ref{sec:polyfunc}. Stricty speaking, the proof of Theorem~\ref{thm:duc} assumes $X$ and $Y$ to be locally Noetherian, but all the arguments apply in the general case, once Theorem~\ref{thm:perfect} and Proposition~\ref{prop:reg} are available. Since a polynomial structure of degree $n+1$ on $\la_{X/Y}$ is by definition a multi-additive structure on the $(n+1)$-st difference 
$$
(\d^{n+1}\la_{X/Y})(L_0,\ldots,L_n)=\sum_{I\subset\{0,\ldots,n\}}(-1)^{n+1-|I|}\la_{X/Y}\left(\sum_{i\in I} L_i\right),
$$
and we can thus define the \emph{Deligne pairing} as the functor $\cP(X)^{n+1}\to\cP(Y)_\Q$ that takes line bundles $L_0,\ldots,L_n$ on $X$ to the $\Q$-line bundle
\begin{equation}\label{defi:Delignepairing}
\langle L_0,\ldots,L_n\rangle_{X/Y}:=(\d^{n+1}\la_{X/Y})(L_0,\ldots,L_n). 
\end{equation}

\begin{thm}\label{thm:del} The Deligne pairing satisfies the following properties. 
\begin{itemize}
\item[(i)] it is multi-additive, symmetric, and commutes with base change; 
\item[(ii)] for each relative effective Cartier divisor $Z$ on $X$, we have canonical multi-additive functorial isomorphisms
$$
\langle\cO_X(Z),L_1,\ldots,L_n\rangle_{X/Y}\simeq\langle L_1|_Z,\ldots,L_n|_Z\rangle_{Z/Y}; 
$$
\item[(iii)] if $g:X'\to X$ is a finite flat of degree $e$, then we have canonical functorial isomorphisms
$$
\langle g^\star  L_0, \ldots, g^\star  L_n \rangle_{X'/Y} \simeq e \langle L_0, \ldots, L_n \rangle_{X/Y}.
$$
\end{itemize}
\end{thm}
\begin{proof} (i) follows directly from Theorem~\ref{thm:duc}. Given $Z$ as in (ii), taking the $n$-th iterated difference of the restriction isomorphism 
$$
\la_{Z/Y}(L|_Z)\simeq\la_{X/Y}(L)-\la_{X/Y}(L-Z)
$$
yields 
$$
\langle L_1|_Z,\ldots,L_n|_Z\rangle_{Z/Y}=(\d^n\la_{Z/Y})(L_1|_Z,\ldots,L_n|_Z)
$$
$$
\simeq(\d^n\la_{X/Y})(L_1,\ldots,L_n)-(\d^n_{-Z}\la_{X/Y})(L_1,\ldots,L_n)\simeq-(\d^{n+1}\la_{X/Y})(-Z,L_1,\ldots,L_n),
$$
which is isomorphic to
$$
(\d^{n+1}\la_{X/Y})(Z,L_1,\ldots,L_n)=\langle Z,L_1,\ldots,L_n\rangle_{X/Y},
$$
by multiadditivity of $\d^{n+1}\la_{X/Y}$. Finally, let $g:X'\to X$ be finite and flat of degree $e$, so that $E:=g_\star \cO_{X'}$ is a rank $e$ vector bundle. By the projection formula we have 
$$
R( f\circ g)_\star  (g^\star  L) = R f_\star  Rg_\star  (g^\star  L) = R f_\star  (L \otimes E). 
$$ 
Thus (iii) follows from~\cite[Proposition 4.7.1]{Duc}, as the latter yields a canonical isomorphism between the $(n+1)$-st difference of $L\mapsto\det R f_\star (L\otimes E)$ and $e\,\d^{n+1}\la_{X/Y}$. 
\end{proof}

By Lemma~\ref{lem:poly}, we finally get the following generalization of~\cite[Theorem 4]{KM}. 

\begin{cor}\label{cor:KM} For each $0\le i\le n+1$, define a functor $F_{n+1,i}:\cP(X)^i\to\cP(Y)_\Q$ by 
$$
F_{n+1,i}(L):=\sum_{k=i}^{n+1}(-1)^{k-i}{k\choose i}\left(\d^k\la_{X/Y}\right)(L^k)$$
For each line bundle $L$ on $X$ and $m\in\Z$, we then have functorial isomorphisms
$$
\la_{X/Y}(mL)\simeq\sum_{i=0}^{n+1}{m+i\choose i}F_{n+1,i}(L)=\frac{m^{n+1}}{(n+1)!}\langle L^{n+1}\rangle_{X/Y}+O(m^n),
$$
compatible with base change. 
\end{cor}

%
%
%
%

\end{document}